\documentclass[a4paper,11pt]{article}
\usepackage[utf8]{inputenc}
\usepackage[T1]{fontenc}
\usepackage{textcomp}
\usepackage{fancyhdr}
\usepackage{fullpage}
\usepackage{lmodern}
\usepackage{amsmath,amssymb,bm,bbm}
\usepackage{algorithm}
\usepackage{algorithmic}
\usepackage{float}
\usepackage{graphicx}
\usepackage{extarrows}
\usepackage{caption}
\usepackage{graphicx, subfig}
\usepackage{titling}
\usepackage{yhmath}
\usepackage{mathrsfs}
\usepackage{cite}
\usepackage{footnote}
\usepackage{amsthm}
\usepackage{color}
\usepackage{slashed}
\usepackage{verbatim}
\usepackage[colorlinks=true,linkcolor=red,citecolor=blue]{hyperref}
\usepackage{amsfonts,euscript}
\usepackage{latexsym, multicol}
\usepackage{epstopdf}
\usepackage{psfrag}
\usepackage{mathrsfs}
\usepackage{xcolor}
\usepackage{comment}
\usepackage{authblk}
\usepackage{indentfirst}
\usepackage{lipsum}
\DeclareFontFamily{U}{mathx}{\hyphenchar\font45}
\DeclareFontShape{U}{mathx}{m}{n}{
<5> <6> <7> <8> <9> <10>
<10.95> <12> <14.4> <17.28> <20.74> <24.88>
mathx10}{}
\DeclareSymbolFont{mathx}{U}{mathx}{m}{n}
\DeclareFontSubstitution{U}{mathx}{m}{n}
\DeclareMathAccent{\widecheck}{0}{mathx}{"71}
\numberwithin{equation}{section}
\allowdisplaybreaks[3]
\def\uti{\widetilde{u}}
\renewcommand{\c}{\cdot}

\newcommand{\unl}{\underline{L}}

\newcommand{\pa}{\partial}
\newcommand{\pr}{\pa}
\newcommand{\unu}{{\underline{u}}}
\newcommand{\unc}{{\underline{C}}}
\newcommand{\una}{\underline{\alpha}}
\newcommand{\unb}{\underline{\beta}}
\newcommand{\bg}{\mathbf{g}}
\newcommand{\mo}{\mathcal{O}}
\newcommand{\unchi}{\underline{\chi}}
\newcommand{\uome}{\underline{\omega}}
\newcommand{\ue}{\underline{\eta}}
\newcommand{\sld}{d}
\def\aac{\mathring{\aa}}
\def\aas{\underline{\slashed{\a}}}
\def\rg{\b^{(1)}}
\def\mobr{\breve{\mo}}
\def\sfr{\mathfrak{s}}
\def\sk{\sfr}
\def\ovla{\overset{\circ}{\la}}
\def\Rfk{\mathfrak{R}}
\newcommand{\kk}{\mathcal{K}}
\newcommand{\mr}{\mathcal{R}}
\newcommand{\ur}{\underline{\mr}}
\newcommand{\uo}{\underline{\mo}}

\newcommand{\chr}{\widecheck{\rho}}
\newcommand{\M}{\mathcal{M}}
\newcommand{\D}{\mathbf{D}}

\newcommand{\uu}{\underline{\mu}}
\newcommand{\uf}{\underline{F}}

\newcommand{\tio}{{\widetilde{\mo}}}

\newcommand{\tir}{{\widetilde{\mr}}}
\newcommand{\cuv}{{C_u^V}}
\newcommand{\ucuv}{{\unc_\unu^V}}
\newcommand{\dd}{{\mathfrak{d}}}

\newcommand{\Ric}{{\ric}}
\DeclareMathOperator{\sRic}{Ric}
\newcommand{\g}{\bg}
\newcommand{\R}{{\mathbf{R}}}
\newcommand{\K}{\mathbf{K}}
\def\ac{\mathring{\a}}
\def\as{\slashed{\a}}
\def\hot{\widehat{\otimes}}
\def\uobr{\breve{\uo}}
\def\IIbr{\breve{\II}_0}
\def\fb{\underline{f}}

\def\Lb{\unl}
\def\lot{l.o.t.}
\def\II{\mathcal{I}}
\def\la{\lambda}
\def\rhoc{\chr}
\def\Rk{\mathfrak{R}_0^2}
\def\fc{\widecheck{f}}

\def\MM{\mathcal{M}}
\def\ub{\unu}
\def\chib{\unchi}
\def\Cb{\unc}
\def\omb{\uome}
\def\om{\omega}
\def\ze{\zeta}
\def\Om{\Omega}
\def\etab{\ue}
\def\aa{\una}
\def\bb{\unb}
\def\Si{\Sigma}
\def\si{\sigma}
\def\mub{{\underline{\mu}}}
\def\dk{\dd}
\def\dkb{\slashed{\dk}}

\def\ga{\gamma}
\def\Ga{\Gamma}
\def\xib{\underline{\xi}}
\def\a{\alpha}
\def\b{\beta}
\def\hch{\widehat{\chi}}
\def\hchb{\widehat{\chib}}
\def\trch{\tr\chi}
\def\trchb{\tr\chib}
\def\trchc{\widecheck{\trch}}
\def\trchbc{\widecheck{\trchb}}

\def\de{\delta}
\def\De{\Delta}
\def\nab{\nabla}
\def\ov{\overline}
\def\omc{\widecheck{\om}}
\def\ombc{\widecheck{\omb}}
\def\bbb{\underline{b}}
\def\Gab{\Ga_b}
\def\Gag{\Ga_g}
\def\Gaa{\Ga_a}
\def\ep{\epsilon}
\def\les{\lesssim}
\def\RR{\mr}
\def\RRb{\ur}
\def\OO{\mo}

\def\sic{\widecheck{\si}}

\def\upa{\wideparen{u}}

\def\Omc{\widecheck{\Om}}
\def\omtrchc{\widecheck{\Om\trch}}
\def\omtrchbc{\widecheck{\Om\trchb}}
\DeclareMathOperator{\curl}{curl}
\DeclareMathOperator{\err}{Err}

\DeclareMathOperator{\grad}{grad}

\DeclareMathOperator{\tr}{tr}
\DeclareMathOperator{\sdiv}{div}
\def\bdiv{\mathbf{Div}}
\def\ric{\mathbf{Ric}}
\DeclareMathOperator{\osc}{Osc}
\newtheorem{thm}{Theorem}[section]
\newtheorem{prop}[thm]{Proposition}
\newtheorem{lem}[thm]{Lemma}
\newtheorem{cor}[thm]{Corollary}
\newtheorem{rk}[thm]{Remark}
\newtheorem{df}[thm]{Definition}

\title{Stability of Minkowski spacetime in exterior regions}
\author{Dawei Shen\footnote{Email adress: dawei.shen@polytechnique.edu \par \indent\hspace{0.26cm} Laboratoire Jacques-Louis Lions, Sorbonne Universit\'e, 75252 Paris, France}}
\begin{document}
\maketitle
\begin{abstract}
In 1993, the global stability of Minkowski spacetime has been proven in the celebrated work of Christodoulou and Klainerman \cite{Ch-Kl} in a maximal foliation. In 2003, Klainerman and Nicol\`o \cite{kn} gave a second proof of the stability of Minkowski in the case of the exterior of an outgoing null cone. In this paper, we give a new proof of \cite{kn}. Compared to \cite{kn}, we reduce the number of derivatives needed in the proof, simplify the treatment of the last slice, and provide a unified treatment of the decay of initial data. Also, concerning the treatment of curvature estimates, we replace the vectorfield method used in \cite{kn} by the $r^p$--weighted estimates of Dafermos and Rodnianski \cite{Da-Ro}.
\end{abstract}
\tableofcontents
\section{Introduction}\label{sec6}
\subsection{Einstein vacuum equations and the Cauchy problem}
A Lorentzian $4$--manifold $(\MM,\g)$ is called a vacuum spacetime if it solves the Einstein vacuum equations:
\begin{equation}\label{EVE}
    \Ric(\g)=0\quad \mbox{ in }\MM,
\end{equation}
where $\Ric$ denotes the Ricci tensor of the Lorentzian metric $\g$. The Einstein vacuum
equations are invariant under diffeomorphisms, and therefore one considers equivalence
classes of solutions. Expressed in general coordinates, (1.1) is a non-linear geometric coupled system of partial differential equations of order 2 for $\g$. In suitable coordinates, for example so-called wave coordinates, it can be shown that \eqref{EVE} is hyperbolic and hence admits an initial value formulation. \\ \\
The corresponding initial data for the Einstein vacuum equations is given by specifying a
triplet $(\Si,g,k)$ where $(\Si,g)$ is a Riemannian $3$--manifold and $k$ is the traceless symmetric $2$--tensor on $\Si$ satisfying the constraint equations:
\begin{align}
    \begin{split}\label{constraintk}
        R&=|k|^2-(\tr k)^2,\\
        D^j k_{ij}&=D_i(\tr k),
    \end{split}
\end{align}
where $R$ denotes the scalar curvature of $g$, $D$ denotes the Levi-Civita connection of $g$ and
\begin{align*}
    |k|^2:=g^{ad}g^{bc}k_{ab}k_{cd},\qquad\quad \tr k:=g^{ij}k_{ij}.
\end{align*}
In the future development $(\MM,\g)$ of such initial data $(\Si,g,k)$, $\Si\subset \MM$ is a spacelike hypersurface with induced metric $g$ and second fundamental form $k$.\\ \\
The seminal well-posedness results for the Cauchy problem obtained in \cite{cb,cbg} ensure that for any smooth Cauchy data, there exists a unique smooth maximal globally hyperbolic development $(\MM,\g)$ solution of Einstein equations \eqref{EVE} such that $\Si\subset \MM$ and $g$, $k$ are respectively the first and second fundamental forms of $\Si$ in $\MM$. \\ \\
The prime example of a vacuum spacetime is Minkowski space:
\begin{equation*}
    \MM=\mathbb{R}^4,\qquad \g=-dt^2+(dx^1)^2 +(dx^2)^2+(dx^3)^2,
\end{equation*}
for which Cauchy data are given by
\begin{equation*}
    \Si=\mathbb{R}^3,\qquad g=(dx^1)^2+(dx^2)^2+(dx^3)^2,\qquad k=0.
\end{equation*}
In the present work, we consider the problem of the stability of Minkowski spacetime and start with reviewing the literature on this problem.
\subsection{Previous works of the stability of Minkowski spacetime}\label{ssec6.1}
In 1993, Christodoulou and Klainerman \cite{Ch-Kl} proved the stability of Minkowski for the Einstein-vacuum equations, a milestone in the domain of mathematical general relativity. In 2003, Klainerman and Nicol\`o \cite{kn} gave a second proof of this result in the exterior of an outgoing cone. Moreover, Klainerman and Nicol\`o \cite{kncqg} showed that under stronger asymptotic decay and regularity properties than those used in \cite{Ch-Kl,kn}, asymptotically flat initial data sets lead to solutions of the Einstein vacuum equations which have strong peeling properties. Given that the goal of this paper is to provide a new proof of the stability of Minkowski in the exterior region, we will state the results of \cite{Ch-Kl} and \cite{kn} in section \ref{ssec6.2}.\\ \\
The proofs in \cite{Ch-Kl} and \cite{kn} are based respectively on the maximal foliation and the double null foliation. Lindblad and Rodnianski \cite{lr1,lr2} gave a new proof of the stability of the Minkowski spacetime using \emph{wave-coordinates} and showing that the Einstein equations verify the so called \emph{weak null structure} in that gauge. Bieri \cite{bieri} gave a proof requiring less derivative and less vectorfield compared to \cite{Ch-Kl}. Huneau \cite{huneau} proved the nonlinear stability of Minkowski spacetime with a translation Killing field using generalised wave-coordinates. Using the framework of Melrose’s b-analysis, Hintz and Vasy \cite{hintz} reproved the stability of Minkowski space. Graf \cite{graf} proved the global nonlinear stability of Minkowski space in the context of the spacelike-characteristic Cauchy problem for Einstein vacuum equations, which together with \cite{kn} allows to reobtain \cite{Ch-Kl}.\\ \\
There are also stability results concerning Einstein's equations coupled with non trivial matter fields:
\begin{itemize}
 \item Einstein-Maxwell system: Zipser \cite{zipser} extended the framework of \cite{Ch-Kl} to show the stability of the Minkowski spacetime solution to the Einstein–Maxwell system. In \cite{lo09}, Loizelet used the framework of \cite{lr1,lr2} to demonstrate the stability of the Minkowski spacetime solution of the Einstein-scalar field-Maxwell system in $(1+n)$-dimensions $(n\geq 3)$. Speck \cite{speck} gave a proof of the global nonlinear stability of the $(1+3)$-dimensional Minkowski spacetime solution to the coupled system for a family of electromagnetic fields, which includes the standard Maxwell fields.
\item Einstein-Klein-Gordon system: Lefloch and Ma \cite{lefloch} and Wang \cite{wang} proved the global stability of Minkowski for the Einstein-Klein-Gordon system with initial data coinciding with the Schwarzschild solution with small mass outside a compact set. Ionescu and Pausader \cite{ionescu} proved the global stability of Minkowski for the Einstein-Klein-Gordon system for general initial data.
\item Einstein-Vlasov system: Taylor \cite{taylor} considered the massless case where the initial data for the Vlasov part is compactly supported on the mass shell. Fajman, Joudioux and Smulevici \cite{fajman} considered the massive case where the initial data coincides with Schwarzschild in the exterior region and with compact support assumption only in space on the Vlasov part. Lindblad and Taylor \cite{lt} considered the massive case where the initial data has compact support for the Vlasov part. Bigorgne, Fajman, Joudioux, Smulevici and Thaller \cite{bigo} considered the massless case for general initial data. Wang \cite{wxc} considered the massive case for general initial data.
\end{itemize}
\subsection{Minkowski Stability in \texorpdfstring{\cite{Ch-Kl}}{} and \texorpdfstring{\cite{kn}}{}}\label{ssec6.2}
We recall in this section the results in \cite{Ch-Kl,kn}. First, we recall the definition of a \emph{maximal hypersurface}, which plays an important role in the statements of the main theorems in \cite{Ch-Kl,kn}.
\begin{df}\label{def6.1}
An initial data $(\Si,g,k)$ is posed on a maximal hypersurface if it satisfies 
\begin{equation}
    \tr k=0.
\end{equation}
In this case, we say that $(\Si,g,k)$ is a maximal initial data set, and the constraint equations \eqref{constraintk} reduce to
\begin{equation}
    R=|k|^2,\qquad \sdiv k=0,\qquad \tr k=0.
\end{equation}
\end{df}
We introduce the notion of \emph{$s$--asymptotically flat initial data}.
\begin{df}\label{def6.3}
Let $s$ be a real number with $s>3$. We say that a data set $(\Sigma_0,g,k)$ is $s$--asymptotically flat if there exists a coordinate system $(x^1,x^2,x^3)$ defined outside a sufficiently large compact set such that\footnote{The notation $f=o_l(r^{-m})$ means $\pr^\a f=o(r^{-m-|\a|})$, $|\a|\leq l$.}
\begin{align}
    \begin{split}\label{old1.3}
        g_{ij}=\left(1-\frac{2M}{r}\right)^{-1} dr^2+ r^2 d\si_{\mathbb{S}^2}+o_4(r^{-\frac{s-1}{2}}),\\
        k_{ij}=o_3(r^{-\frac{s+1}{2}}).
    \end{split}
\end{align}
\end{df}
We also introduce the following functional associated to any asymptotically flat initial data set:
\begin{equation}
    J_0(\Sigma_0,g,k):=\sup_{\Sigma_0}\Big((d_0^2+1)^3 |\sRic|^2 \Big)+\int_{\Sigma_0}\sum_{l=0}^3(d_0^2+1)^{l+1}|D^l k|^2 +\int_{\Sigma_0}\sum_{l=0}^1 (d_0^2+1)^{l+3}|D^l B|^2,\label{6.4}
\end{equation}
where $d_0$ is the geodesic distance from a fixed point $O\in \Sigma_0$, and $B_{ij}:=(\curl\widehat{\overline{R}})_{ij}$ is the \emph{Bach tensor}, $\widehat{\overline{R}}$ is the traceless part of $\sRic$. Now, we can state the main theorems of \cite{Ch-Kl} and \cite{kn}.
\begin{thm}[Global stability of Minkowski space \cite{Ch-Kl}]
There exists an $\epsilon>0$ sufficiently small such that if $J_0(\Sigma_0,g,k)\leq \epsilon^2$, then the initial data set $(\Sigma_0,g,k)$, $4$--asymptotically flat (in the sense of Definition \ref{def6.3}) and maximal, has a unique, globally hyperbolic, smooth, geodesically complete solution. This development is globally asymptotically flat, i.e. the Riemann curvature tensor tends to zero along any causal or space-like geodesic. Moreover, there exists a global maximal time function $t$ and an optical function $u$\footnote{An optical function $u$ is a scalar function satisfying $\g^{\a\b}\pr_\a u\pr_\b u=0$.} defined everywhere in an external region.
\end{thm}
\begin{thm}[Minkowski Stability in the exterior region \cite{kn}]\label{knmain}
Consider an initial data set $(\Sigma_0,g,k)$, $4$--asymptotically flat and maximal, and assume $J_0(\Sigma_0,g,k)$ is bounded. Then, given a sufficiently large compact set $K\subset\Sigma_0$ such that $\Sigma_0 \setminus K$ is diffeomorphic to $\mathbb{R}^3\setminus \overline{B}_1$, and under additional smallness assumptions, there exists a unique development $(\M,\bg)$ with the following properties:\\
(1) $(\M,\bg)$ can be foliated by a double null foliation $\{C_u\}$ and $\{\unc_\unu\}$ whose outgoing leaves $C_u$ are complete.\\
(2) We have detailed control of all the quantities associated with the double null foliations of the spacetime, see Theorem 3.7.1 of \cite{kn}.
\end{thm}
\begin{rk}\label{peelingcqg}
\cite{kncqg} extends Theorem \ref{knmain} to the case $s>7$ by assuming sufficient regular initial data $(\Si_0,g,k)$.
\end{rk}
The goal of this paper is to reprove \cite{kn} by a different method and to treat more general asymptotic behavior than $s=4$ in \eqref{old1.3}.
\subsection{Rough version of the main theorem}\label{ssec6.3}
In this section, we state a simple version of our main theorem. For the explicit statement, see Theorem \ref{th8.1}.
\begin{thm}[Main Theorem (first version)]\label{old1.6}
Let $s>3$, and let an initial data set $(\Si_0,g,k)$ which is $s$--asymptotically flat in the sense of Definition \ref{def6.3}. Let a sufficiently large compact set $K\subset \Si_0$ such that $\Si_0\setminus K$ is diffeomorphic to $\mathbb{R}^3\setminus \ov{B}_1$. Assume that we have a smallness conditions in an initial layer region $\kk_{(0)}$\footnote{The initial data layer $\kk_{(0)}$ is defined in section \ref{sssec7.1.1}.} near $\Si_0\setminus K$. Then, there exists a unique future development $(\M,\g)$ in its future domain of dependence with the following properties:
\begin{itemize}
    \item $(\M,\g)$ can be foliated by a double null foliation $(C_u,\unc_\unu)$ whose outgoing leaves $C_u$ are complete for all $u\leq u_0$;
    \item We have detailed control of all the quantities associated with the double null foliations of the spacetime, see Theorem \ref{th8.1}.
\end{itemize}
\begin{rk}
    In the particular case $s=4$, we reobtain the results of \cite{kn}. Also, in the case $s>7$, we reobtain the strong peeling properties of \cite{kncqg}.
\end{rk}
\end{thm}
The proof of Theorem \ref{old1.6} has the same structure as in \cite{kn}, see section \ref{ssec8.6}. Below, we compare the proof of this paper and that of \cite{kn}:
\begin{enumerate}
\item In \cite{kn}, the functional \eqref{6.4} is introduced to fix the initial conditions on the initial hypersurface. Here we fix the initial conditions in an initial layer region $\kk_{(0)}$ near the initial hypersurface $\Si_0$.
\item To estimate the norms of curvature components, \cite{kn} uses the vectorfield method introduced in \cite{Ch-Kl}. Here, we estimate the curvature norms by $r^p$-weighted estimate, a method introduced by Dafermos and Rodnianski in \cite{Da-Ro}. This allows for a simpler treatment of the curvature estimates.
\item \cite{kn} uses second order derivatives of curvature and third order derivatives of Ricci coefficients. Thanks to the use of $r^p$--weighted estimates, we only need first order derivatives of both curvature and Ricci coefficients.
\item In order to control one more derivative for Ricci coefficients compared to curvature, \cite{kn} relies on the canonical foliation on the last slice $\Cb_*$. Since we control the same number of derivatives of curvature and Ricci coefficients, instead of introducing the canonical foliation, we use the geodesic foliation to simplify the estimates on the last slice $\Cb_*$.
\end{enumerate}
\subsection{Structure of the paper}\label{ssec6.4}
\begin{itemize}
    \item In section \ref{sec7}, we recall the fundamental notions and the basic equations.
    \item In section \ref{sec8}, we present the main theorem. We then state intermediate results, and prove the main theorem. The rest of the paper focuses on the proof of these intermediary results. In sections \ref{sec9}--\ref{sec11}, we consider the case $s\in[4,6]$ and deal with the other cases in Appendices \ref{secc} and \ref{secd}.
    \item In section \ref{sec9}, we apply $r^p$--weighted estimates to Bianchi equations to control curvature.
    \item In section \ref{sec10}, we estimate the Ricci coefficients using the null structure equations.
    \item In section \ref{sec11}, we estimate the norms of curvature components and Ricci coefficients in the initial layer region $\kk_{(0)}$ and on the last slice $\unc_*$. We also show how to extend the spacetime in the context of a bootstrap argument.
    \item In Appendix \ref{secb}, we prove the equivalence between optical functions and area radius associated to various frames used in the proof.
    \item In Appendix \ref{secc}, we prove Theorem \ref{old1.6} in the case $s\in (3,4)$. 
    \item In Appendix \ref{secd}, we prove Theorem \ref{old1.6} in the case $s> 6$ by applying the $r^p$--weighted method to the Teukolsky equation of $\a$.
\end{itemize}
\subsection{Acknowledgements} The author is very grateful to J\'er\'emie Szeftel for his support, discussions, encouragements and patient guidance.
\section{Preliminaries}\label{sec7}
\subsection{Geometric set-up}\label{ssec7.1}
\subsubsection{Double null foliation}
We first introduce the geometric setup. We denote $(\MM,\bg)$ a spacetime $\MM$ with the Lorentzian metric $\g$ and $\D$ its Levi-Civita connection. Let $u$ and $\unu$ be two optical functions on $\MM$, that is
\begin{equation*}
    \bg(\grad u,\grad u)=\bg(\grad\unu,\grad\unu)=0.
\end{equation*}
The spacetime $\M$ is foliated by the level sets of $u$ and $\unu$ respectively, and the functions $u,\unu$ are required to increase towards the future. We use $C_u$ to denote the outgoing null hypersurfaces which are the level sets of $u$ and use $\unc_\unu$ to denote the incoming null hypersurfaces which are the level sets of $\unu$. We denote
\begin{equation}
    S(u,\unu):=C_u\cap\unc_\unu,
\end{equation}
which are space-like $2$--spheres. We introduce the vectorfields $L$ and $\unl$ by
\begin{equation*}
    L:=-\grad u,\quad \unl:= -\grad\unu.
\end{equation*}
We define a positive function $\Omega$ by the formula
\begin{equation*}
    \bg(L,\unl)=-\frac{1}{2\Omega^2},
\end{equation*}
where $\Om$ is called the lapse function. We then define the normalized null pair $(e_3,e_4)$ by
\begin{equation*}
   e_3=2\Omega\unl,\quad e_4=2\Omega L, 
\end{equation*}
and define another null pair by
\begin{equation*}
    \underline{N}=\Omega e_3, \quad N=\Omega e_4.
\end{equation*}
On a given two sphere $S(u,\unu)$, we choose a local frame $(e_1,e_2)$, we call $(e_1,e_2,e_3,e_4)$ a null frame. As a convention, throughout the paper, we use capital Latin letters $A,B,C,...$ to denote an index from 1 to 2 and Greek letters $\alpha,\beta,\gamma,...$ to denote an index from 1 to 4, e.g. $e_A$ denotes either $e_1$ or $e_2$.\\ \\
The spacetime metric $\bg$ induces a Riemannian metric $\ga$ on $S(u,\unu)$. We use $\nabla$ to denote the Levi-Civita connection of $\ga$ on $S(u,\unu)$. Using $(u,\unu)$, we introduce a local coordinate system $(u,\unu,\phi^A)$ on $\M$ with $e_4(\phi^A)=0$. In that coordinates system, the metric $\g$ takes the form:
\begin{equation}\label{metricg}
\g=-2\Omega^2 (d\unu\otimes du+du\otimes d\unu)+\ga_{AB}(d\phi^A-\bbb^Adu)\otimes (d\phi^B-\bbb^Bdu),
\end{equation}
and we have 
\begin{equation*}
   \underline{N}=\pa_u+\bbb,\qquad N=\pa_{\unu},\qquad \bbb:=\bbb^A \pr_{\phi^A}.
\end{equation*}
We recall the null decomposition of the Ricci coefficients and curvature components of the null frame $(e_1,e_2,e_3,e_4)$ as follows:
\begin{align}
\begin{split}\label{defga}
\chib_{AB}&=\g(\D_A e_3, e_B),\qquad\quad \chi_{AB}=\g(\D_A e_4, e_B),\\
\xib_A&=\frac 1 2 \g(\D_3 e_3,e_A),\qquad\quad\,\,  \xi_A=\frac 1 2 \g(\D_4 e_4, e_A),\\
\omb&=\frac 1 4 \g(\D_3e_3 ,e_4),\qquad\quad \,\,\,\,\, \om=\frac 1 4 \g(\D_4 e_4, e_3), \\
\etab_A&=\frac 1 2 \g(\D_4 e_3, e_A),\qquad\quad\,  \eta_A=\frac 1 2 \g(\D_3 e_4, e_A),\\
\ze_A&=\frac 1 2 \g(\D_{e_A}e_4, e_3),
\end{split}
\end{align}
and
\begin{align}
\begin{split}\label{defr}
\a_{AB} &=\R(e_A, e_4, e_B, e_4) , \qquad \,\,\,\aa_{AB} =\R(e_A, e_3, e_B, e_3), \\
\b_{A} &=\frac 1 2 \R(e_A, e_4, e_3, e_4), \qquad\,\, \;\bb_{A}=\frac 1 2 \R(e_A, e_3, e_3, e_4),\\
\rho&= \frac 1 4 \R(e_3, e_4, e_3, e_4), \qquad\,\;\;\,\,\;\, \si =\frac{1}{4}{^*\R}( e_3, e_4, e_3, e_4),
\end{split}
\end{align}
where $^*\R$ denotes the Hodge dual of $\R$. The null second fundamental forms $\chi, \chib$ are further decomposed in their traces $\trch$ and $\trchb$, and traceless parts $\hch$ and $\hchb$:
\begin{align*}
\trch&:=\de^{AB}\chi_{AB},\qquad\quad \,\hch_{AB}:=\chi_{AB}-\frac{1}{2}\de_{AB}\trch,\\
\trchb&:=\de^{AB}\chib_{AB},\qquad\quad \, \hchb_{AB}:=\chib_{AB}-\frac{1}{2}\de_{AB}\trchb.
\end{align*}
We define the horizontal covariant operator $\nab$ as follows:
\begin{equation*}
\nab_X Y:=\D_X Y-\frac{1}{2}\chib(X,Y)e_4-\frac{1}{2}\chi(X,Y)e_3.
\end{equation*}
We also define $\nab_4 X$ and $\nab_3 X$ to be the horizontal projections:
\begin{align*}
\nab_4 X&:=\D_4 X-\frac{1}{2} \g(X,\D_4e_3)e_4-\frac{1}{2} \g(X,\D_4e_4)e_3,\\
\nab_3 X&:=\D_3 X-\frac{1}{2} \g(X,\D_3e_3)e_3-\frac{1}{2} \g(X,\D_3e_4)e_4.
\end{align*}
A tensor field $\psi$ defined on $\MM$ is called tangent to $S$ if it is a priori defined on the spacetime $\M$ and all the possible contractions of $\psi$ with either $e_3$ or $e_4$ are zero. We use $\nabla_3 \psi$ and $\nabla_4 \psi$ to denote the projection to $S(u,\unu)$ of usual derivatives $\D_3\psi$ and $\D_4\psi$. As a direct consequence of \eqref{defga}, we have the Ricci formulas:
\begin{align}
\begin{split}\label{ricciformulas}
    \D_A e_B&=\nab_A e_B+\frac{1}{2}\chi_{AB} e_3+\frac{1}{2}\chib_{AB}e_4,\\
    \D_A e_3&=\chib_{AB}e_B+\ze_A e_3,\\
    \D_A e_4&=\chi_{AB}e_B-\ze_A e_4,\\
    \D_3 e_A&=\nab_3 e_A+\eta_A e_3+\xib_A e_4,\\
    \D_4 e_A&=\nab_4 e_A+\etab_A e_4+\xi_A e_4,\\
    \D_3 e_3&=-2\omb e_3+2\xib_B e_B,\\ 
    \D_3 e_4&=2\omb e_4+2\eta_B e_B,\\
    \D_4 e_4&=-2\om e_4+2\xi_B e_B,\\
    \D_4 e_3&=2\om e_3+2\etab_B e_B.
\end{split}
\end{align}
In addition to 
\begin{equation}
    \xi=\xib=0,
\end{equation}
the following identities hold for a double null foliation:
\begin{align}
\begin{split}\label{6.6}
    \nabla\log\Omega&=\frac{1}{2}(\eta+\ue),\qquad\;\;\;\,\omega=-\frac{1}{2}\D_4(\log\Omega), \qquad\;\;\;\,\omb=-\frac{1}{2}\D_3(\log\Omega),\\ 
    \eta&=\zeta+\nabla\log\Omega,\qquad \ue=-\zeta+\nabla\log\Omega,
\end{split}
\end{align}
see for example (6) in \cite{kr}.
\subsubsection{Initial layer region}\label{sssec7.1.1}
Let $\Si_0$ be a spacelike hypersurface. Let $K\subset\Si_0$ be a compact subset such that $\Si_0\setminus K$ is diffeomorphic to $\mathbb{R}^3 \setminus \ov{B_1}$ where $\ov{B_1}$ is the unit closed ball in $\mathbb{R}^3$. We fix a foliation on the initial hypersurface $\Sigma_0 \setminus K$ by the level sets of a scalar function $w$. The leaves are denoted by
\begin{equation}
    S_{(0)}(w_1)=\{p\in\Si_0/\, w(p)=w_1\},
\end{equation}
where $w_1\in \mathbb{R}$. We assume that
\begin{equation}
    \pa K=\{p\in\Si_0/\,w(p)=w_0\},\qquad K=\{p\in\Si_0/\, w(p)\leq w_0\},
\end{equation}
where $w_0$ is the area radius of $\pr K$ defined by
\begin{equation}\label{defw0}
    w_0:=\sqrt{\frac{|\pr K|}{4\pi}}.
\end{equation}
We can construct a double null foliation in a neighborhood of $\Si_0\setminus K$. For this, we denote $T$ the normal vectorfield of $\Sigma_0$ oriented towards the future and $N$ the unit vectorfield tangent to $\Si_0$, oriented towards infinity and orthogonal to the leaves $S_0(w)$. We define two null vectors on $\Sigma_0\setminus K$ by
\begin{equation}
    L_{(0)}:=N(w)(T+N),\qquad\unl:=N(w)(T-N).\label{tntn}
\end{equation}
We extend the definition of $L_{(0)}$ and $\unl$ to a neighborhood of $\Sigma_0\setminus K$ by the geodesic equations:
\begin{equation}
    \D_{L_{(0)}} L_{(0)}=0,\qquad \D_{\unl} \unl=0.\label{geo}
\end{equation}
The lapse function $\Om_{(0)}$ is defined by
\begin{equation}
    \g(L_{(0)},\Lb)=-\frac{1}{2{\Om_{(0)}}^2}.
\end{equation}
Then, we define two optical functions $u_{(0)}$ and $\unu$ satisfying the initial conditions
\begin{equation*}
    u_{(0)}\big|_{\Sigma_0\setminus K}=-w,\qquad \unu\big|_{\Sigma_0\setminus K}=w,
\end{equation*}
and the equations
\begin{equation*}
    L_{(0)}=-\grad u_{(0)},\qquad \unl=-\grad \unu.
\end{equation*}
Recall that $u_{(0)}+\unu=0$ on $\Sigma_0\setminus K$, we define the region
\begin{equation}\label{defde0}
    \kk_{(0)}:=\{p/\,0\leq u_{(0)}(p)+\unu(p)\leq 2\de_0\},
\end{equation}
where $0<\de_0<1$ is a constant. We call $\kk_{(0)}$ the \emph{initial layer region} of height $\delta_0$, the double null foliation $(u_{(0)},\unu)$ is called the \emph{initial layer foliation}. Its leaves are denoted by
\begin{equation*}
    S_{(0)}(u_{(0)},\unu):=(C_{(0)})_{u_{(0)}}\cap \unc_\unu.
\end{equation*}
We define 
\begin{align}
    w:=\frac{\ub-u_{(0)}}{2} \quad \mbox{ in }\kk_{(0)},
\end{align}
which extend the definition of $w$ from $\Si_0\setminus K$ to $\kk_{(0)}$. Moreover, we transport the coordinates ${\phi}_{(0)}^A$ from $\Si_0\setminus K$ to $\kk_{(0)}$ by $\Lb({\phi}_{(0)}^A)=0$. Hence, we deduce that the metric $\g$ in $\kk_{(0)}$ takes the following form in the $(\ub,u_{(0)},\phi_{(0)}^A)$ coordinates system:
\begin{equation}
    \g=-2{\Om_{(0)}}^2(d\ub\otimes du_{(0)}+du_{(0)}\otimes d\ub)+(\ga_{(0)})_{AB}(d{\phi}_{(0)}^A-b^A d\ub)(d{\phi}_{(0)}^B-b^B d\ub),
\end{equation}
where
\begin{equation*}
b^A=\Om_{(0)}(e_4)_{(0)}({\phi}_{(0)}^A).
\end{equation*}
\subsubsection{Bootstrap region}\label{sssec7.1.2}
Now, we define the bootstrap region and its double null foliation. We choose a value $\unu_*>w_0$, then $w=\unu_*$ defines a leaf of $\Sigma_0\setminus K$, denoted by $S_0(\ub_*)$. We construct a incoming cone $\Cb_*$ called the \emph{last slice} from $S_0(\ub_*)$. We denote
\begin{equation}
    \unl_*:=\unl\big|_{\unc_*},
\end{equation}
which is fixed by \eqref{tntn} and \eqref{geo}. Notice that $\unl_*$ is a generator of $\unc_*$. We denote $u^*$ an affine parameter, that is a function defined on $\unc_*$ satisfying
\begin{equation}\label{6.12}
    \unl_* (u^*)=2.
\end{equation}
Thus, $\Cb_*$ is endowed with a geodesic foliation of affine parameter $u^*$. For every $u$, $\{u^*=u\}$ is a sphere on $\unc_*$. We can therefore construct an outgoing cone from $\{u^*=u\}$ to the past, denoted by $C_u$. Then, $C_u\cap\Si_0$ is a sphere on $\Si_0$, which does a priori not coincide with the leaves of $\Sigma_0\setminus K$, see Figure \ref{fig3}.
\begin{figure}
  \centering
  \includegraphics[width=0.5\textwidth]{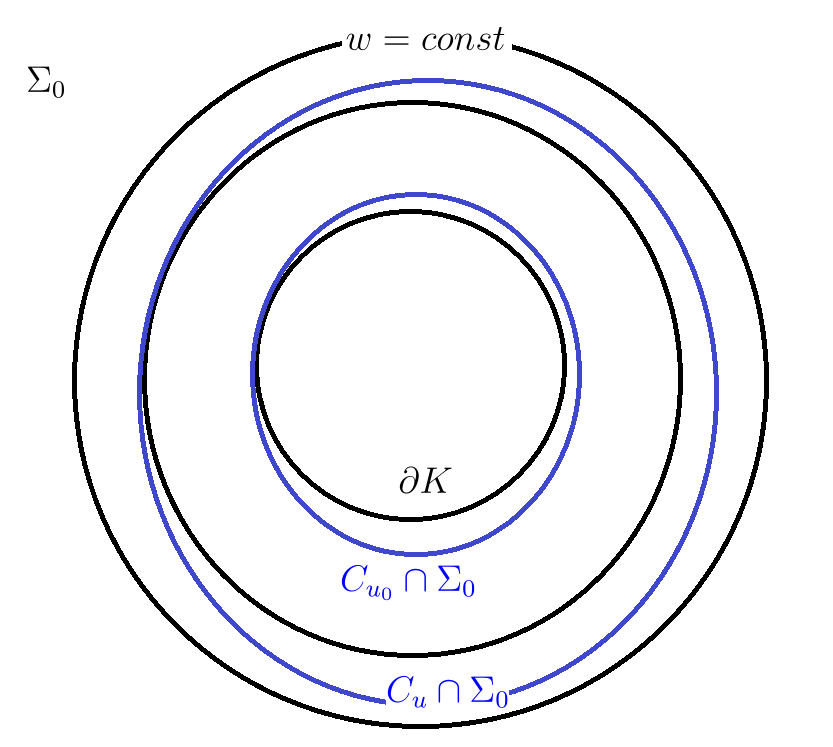}
  \caption{The leaves of $\Sigma_0$}\label{fig3}
\end{figure}
There exists a unique cone $C_{u_0}$ located outside of the $J^+(K)$, the causal future of $K$, but touch the boundary $\pa K$, see Figure \ref{fig3}. By \eqref{6.12}, $u^*$ is unique up to a constant. We choose $u_0=-w_0$ where $w_0$ is defined in \eqref{defw0}. Combining with \eqref{6.12}, this fixes $u^*$.
\begin{rk}
    Remark that we fix $u^*$ by choosing the value near $\pr K$. We do not have
    \begin{equation*}
        u^*|_{S_0(\ub_*)}=-\ub_*.
    \end{equation*}
    The reason is that we need to integrate the quantities in the initial layer forward instead of backward. This will lead to a logarithm difference which is sufficient to prove the equivalences, see Appendix \ref{secb}.
\end{rk}
Now we define the null vector $L$ on $\unc_*$ by the relation
\begin{equation}
    \bg(L,\unl)=-2,
\end{equation}
Hence, we have $\Omega=\frac{1}{2}$ on $\unc_*$. By the equations
\begin{equation}
    -\grad u=L,\qquad u|_{\unc_*}=u^*,
\end{equation}
we can extend $u$ to the causal past of $\Cb_*$. So we have constructed a double null foliation $(u,\unu)$ which is geodesic on $\unc_*$. For every $(u,\unu)$, we denote
\begin{equation}
    V(u,\unu):=J^-(S(u,\unu))\cap J^+(\Sigma_0),
\end{equation}
where for a domain $\MM$, $J^+(\MM)$ and resp. $J^-(\MM)$ denotes the causal future and resp. causal past of $\MM$. We also denote the bootstrap region by
\begin{equation}
    \kk:=V(u_0,\ub_*).
\end{equation}
Remark that $\kk$ is covered by the double null foliation $(u,\unu)$, and that the boundaries of $\kk$ are:
\begin{enumerate}
\item a finite region of $\Sigma_0 \setminus K$; 
\item a portion of an outgoing cone $C_0$, where $C_0:=\{p\in\kk |u(p)=u_0\}$;
\item a portion of the last slice $\unc_*$.
\end{enumerate}
Note that we have the following types of manifolds in this paper: spacetime regions $\kk$ and $\kk_{(0)}$, initial hypersurface $\Si_0$, and the spheres $S(u,\unu)$ and $S_{(0)}(u_{(0)},\ub)$. Every manifold has its metric, Levi-Civita connection and curvature tensor:
\begin{align*}
 (\kk,\g,\D,\R),\quad (\kk_{(0)},\g_{(0)},\D_{(0)},\R_{(0)}),\quad (\Si_0,g,D,R),\quad (S,\ga,\nab,\K),\quad (S_{(0)},\ga_{(0)},\nab_{(0)},\K_{(0)}),
\end{align*}
where $\K$ (resp. $\K_{(0)}$) is the Gauss curvature of $S$ (resp. $S_{(0)}$).
\subsection{Integral formulas}\label{ssecave}
We define the $S$-average of scalar functions.
\begin{df}\label{average}
Given any scalar function $f$, we denote its average and its average free part by
\begin{equation*}
    \overline{f}:=\frac{1}{|S|}\int_{S}f\,d\ga,\qquad \fc:=f-\overline{f}.
\end{equation*}
\end{df}
The following lemma follows immediately from the definition.
\begin{lem}\label{chav}
For any two scalar functions $u$ and $v$, we have
\begin{equation*}
    \overline{uv}=\overline{u}\,\overline{v}+\overline{\widecheck{u}\widecheck{v}},
\end{equation*}
and
\begin{equation*}
uv-\overline{uv}=\widecheck{u}\overline{v}+\overline{u}\widecheck{v}+\left(\widecheck{u}\widecheck{v}-\overline{\widecheck{u}\widecheck{v}}\right).
\end{equation*}
\end{lem}
We recall the following integral formulas, which will be used repeatedly in this paper.
\begin{lem}\label{dint}
For any scalar function $f$, the following equations hold:
\begin{align*}
    \Om e_4\left(\int_{S(u,\unu)} f d\ga\right) &= \int_{S(u,\unu)} \left(\Om e_4(f) + \Omega \tr \chi f \right) d\ga ,\\
    \Om e_3 \left(\int_{S(u,\unu)} f d\ga\right) &= \int_{S(u,\unu)} \left(\Om e_3(f)+ \Om\trchb f \right) d\ga .
\end{align*}
Taking $f=1$, we obtain
\begin{equation*}
    e_4(r)=\frac{\overline{\Omega\tr\chi}}{2\Omega}r,\qquad e_3(r)=\frac{\overline{\Omega\tr\unchi}}{2\Omega}r,\label{e3e4r}
\end{equation*}
where $r$ is the \emph{area radius} defined by
\begin{equation*}
    r(u,\unu):=\sqrt{\frac{|S(u,\unu)|}{4\pi}}.
\end{equation*}
\end{lem}
\begin{proof}
See Lemma 2.26 of \cite{kl-ro} or Lemma 3.1.3 of \cite{kn}.
\end{proof}
\begin{cor}\label{dav}
For any scalar function $f$, the following equations hold:
\begin{align*}
    \Om e_4(\ov{f}) &=\ov{\Om e_4(f)}+\ov{\omtrchc\, \fc} ,\\
    \Om e_3 (\ov{f}) &=\ov{\Om e_3(f)}+\ov{\omtrchbc\, \fc}.
\end{align*}
\end{cor}
\begin{proof}
Applying Lemma \ref{dint}, we infer
\begin{align*}
    \Om e_4 (\ov{f})&=\Om e_4 \left(\frac{1}{|S|}\int_S f d\ga \right)=-\frac{\Om e_4(|S|)}{|S|^2}\int_S f d\ga +\frac{1}{|S|}\int_S \left(\Om e_4(f)+\Om\trch f\right)d\ga \\
    &=-\frac{2\Om e_4(r)}{r}\ov{f}+\ov{\Om e_4 (f)}+\ov{\Om\trch\, f}=-\ov{\Om\trch}\,\ov{f}+\ov{\Om e_4(f)}+\ov{\Om\trch\, f}\\
    &=\ov{\Om e_4(f)}+\ov{\omtrchc \,\widecheck{f}},
\end{align*}
which implies the first identity. The second identity can be obtained similarly. This concludes the proof of Corollary \ref{dav}.
\end{proof}
\subsection{Hodge systems}\label{ssec7.2}
\begin{df}\label{tensorfields}
For tensor fields defined on a $2$--sphere $S$, we denote by $\sfr_0:=\sfr_0(S)$ the set of pairs of scalar functions, $\sfr_1:=\sfr_1(S)$ the set of $1$--forms and $\sfr_2:=\sfr_2(S)$ the set of symmetric traceless $2$--tensors.
\end{df}
\begin{df}\label{def7.2}
Given $\xi\in\sfr_1$, we define its Hodge dual
\begin{equation*}
    {^*\xi}_A := \in_{AB}\xi^B.
\end{equation*}
Clearly $^*\xi \in \sfr_1$ and
\begin{equation*}
    ^*(^*\xi)=-\xi.
\end{equation*}
Given $U \in \sfr_2$, we define its Hodge dual
\begin{equation*}
{^*U}_{AB}:= \in_{AC} {U^C}_B.
\end{equation*}
Observe that $^*U\in\sfr_2$ and
\begin{equation*}
     ^*(^*U)=-U.
\end{equation*}
\end{df}
\begin{df}
    Given $\xi,\eta\in\sfr_1$, we denote
\begin{align*}
    \xi\cdot \eta&:= \de^{AB}\xi_A \eta_B,\\
    \xi\wedge \eta&:= \in^{AB} \xi_A \eta_B,\\
    (\xi\hot \eta)_{AB}&:=\xi_A \eta_B +\xi_B \eta_A -\de_{AB}\xi\cdot \eta.
\end{align*}
Given $\xi\in \sfr_1$, $U\in \sfr_2$, we denote
\begin{align*}
    (\xi \cdot U)_A:= \de^{BC} \xi_B U_{AC}.
\end{align*}
Given $U,V\in \sfr_2$, we denote
\begin{align*}
    (U\wedge V)_{AB}:=\in^{AB}U_{AC}V_{CB}.
\end{align*}
\end{df}
\begin{df}
    For a given $\xi\in\sfr_1$, we define the following differential operators:
    \begin{align*}
        \sdiv \xi&:= \de^{AB} \nab_A\xi_B,\\
        \curl \xi&:= \in^{AB} \nab_A \xi_B,\\
        (\nab\hot\xi)_{AB}&:=\nab_A \xi_B+\nab_B \xi_A-\de_{AB}(\sdiv\xi).
    \end{align*}
\end{df}
\begin{df}
    We define the following Hodge type operators, as introduced in section 2.2 in \cite{Ch-Kl}:
    \begin{itemize}
        \item $d_1$ takes $\sfr_1$ into $\sfr_0$ and is given by:
        \begin{equation*}
            d_1 \xi :=(\sdiv\xi,\curl \xi),
        \end{equation*}
        \item $d_2$ takes $\sfr_2$ into $\sfr_1$ and is given by:
        \begin{equation*}
            (d_2 U)_A := \nab^{ B} U_{AB}, 
        \end{equation*}
        \item $d_1^*$ takes $\sfr_0$ into $\sfr_1$ and is given by:
        \begin{align*}
            d_1^*(f,f_*)_{ A}:=-\nab_A f +\in_{AB} \nab_B f_*,
        \end{align*}
        \item $d_2^*$ takes $\sfr_1$ into $\sfr_2$ and is given by:
        \begin{align*}
            d_2^* \xi := -\frac{1}{2} \nab \hot \xi.
        \end{align*}
    \end{itemize}
\end{df}
We have the following identities:
\begin{align}
    \begin{split}\label{dddd}
        d_1^*d_1&=-\De_1+\mathbf{K},\qquad\qquad  d_1 d_1^*=-\De_{{0}},\\
        d_2^*d_2&=-\frac{1}{2}\De_2+\mathbf{K},\qquad\quad d_2 d_2^*=-\frac{1}{2}(\De_1+\mathbf{K}).
    \end{split}
\end{align}
where $\mathbf{K}$ denotes the Gauss curvature on $S$. See for example (2.2.2) in \cite{Ch-Kl} for a proof of \eqref{dddd}. 
\begin{df}\label{dfdkb}
We define the weighted angular derivatives $\dkb$ as follows:
\begin{align*}
    \dkb U &:= rd_2 U,\qquad \forall U\in \sk_2,\\
    \dkb \xi&:= rd_1 \xi,\qquad\,\,\, \forall \xi\in \sk_1,\\
    \dkb f&:= rd_1^* f,\qquad \,\,\forall f\in \sk_0.
\end{align*}
We denote for any tensor $h\in\sk_k$, $k=0,1,2$,
\begin{equation*}
    h^{(0)}:=h,\qquad h^{(1)}:=(h,\dkb h).
\end{equation*}
\end{df}
\subsection{Elliptic estimates for Hodge systems}
\begin{df}\label{Lpnorms}
For a tensor field $f$ on a $2$--sphere $S$, we denote its $L^p$--norm:
\begin{equation}
    |f|_{p,S}:= \left(\int_S |f|^p \right)^\frac{1}{p}.
\end{equation}
\end{df}
\begin{prop}[$L^p$ estimates for Hodge systems]\label{prop7.3}
Assume that $S$ is an arbitrary compact 2-surface with positive bounded Gauss curvature. Then the following statements hold for all $p\in(1,+\infty)$:
\begin{enumerate}
\item Let $\phi\in\sfr_0$ be a solution of $\De \phi=f$. Then we have
\begin{align*}
    |\nabla^2 \phi|_{p,S}+r^{-1}|\nabla \phi|_{p,S}+r^{-2} |\phi-\overline{\phi}|_{p,S}\les |f|_{p,S}.
\end{align*}
\item Let $\xi\in\sfr_1$ be a solution of $d_1\xi=(f,f_*)$. Then we have
\begin{align*}
    |\nab\xi|_{p,S}+r^{-1}|\xi|_{p,S}\les |(f,f_*)|_{p,S}.
\end{align*}
\item Let $U\in\sfr_2$ be a solution of $d_2 U=f$. Then we have
\begin{align*}
   |\nab U|_{p,S}+r^{-1}|U|_{p,S}\les |f|_{p,S}.
\end{align*}
\end{enumerate}
\end{prop}
\begin{proof}
See Corollary 2.3.1.1 of \cite{Ch-Kl}.
\end{proof}
\subsection{Schematic notation \texorpdfstring{$\Gag$}{}, \texorpdfstring{$\Gab$}{} and \texorpdfstring{$\Gaa$}{}}
We introduce the following schematic notations for the Ricci coefficients.
\begin{df}\label{gammag}
We divide the Ricci coefficients into three parts:
\begin{align*}
    \Gamma_g&:=\left\{\eta,\,\ue,\,\zeta,\,\widehat\chi,\,\widecheck{\Om\tr\chi},\,\widecheck{\Om\tr\unchi},\,\widecheck{\om},\,\nab\log\Om\right\},\\
    \Gamma_b&:=\{\widehat\unchi,\,\widecheck{\omb}\},\\
    \Ga_a&:=\left\{\om,\,\omb,\,\trch-\frac{2}{r},\,\trchb+\frac{2}{r}\right\}.
\end{align*}
We also denote:
\begin{align*}
    \Gag^{(1)}:=(r\nab)^{\leq 1}\Gag\cup\{r\b,r\rhoc,r\si\},\qquad \Gab^{(1)}:=(r\nab)^{\leq 1}\Gab\cup\{r\bb\},\qquad \Gaa^{(1)}:=(r\nab)^{\leq 1}\Gaa\cup\{r\ov{\rho}\}.
\end{align*}
\end{df}
\begin{rk}
    The justification of Definition \ref{gammag} has to do with the expected decay properties of the Ricci coefficients, see Lemma \ref{decayGagGabGaa}.
\end{rk}
\begin{rk}
In the sequel, we choose the following conventions:
\begin{itemize}
    \item For a quantity $h$ satisfying the same or even better decay as $\Ga_{i}$, for $i=g,b,a$, we write
    \begin{equation*}
        h\in\Ga_i ,\quad i=g,b,a.
    \end{equation*}
    \item For a sum of schematic notations, we ignore the terms which have same or even better decay. For example, we write
    \begin{equation*}
        \Gag+\Gab=\Gab,\qquad \Gag+\Gaa=\Gaa,
    \end{equation*}
    since $\Gag$ has better decay than $\Gab$ and $\Gaa$ throughout this paper.
\end{itemize}
\end{rk}
\subsection{Null structure equations}
We recall the null structure equations for a double null foliation, see for example (3.1)-(3.5) of \cite{kr}.
\begin{prop}\label{nulles}
We have the null structure equations:
\begin{align}
\begin{split}
\nabla_4\eta&=-\chi\cdot(\eta-\ue)-\beta,\\
\nabla_3\ue&=-\unchi\cdot(\ue-\eta)+\unb, \\
\nabla_4\widehat{\chi}+(\tr\chi)\widehat{\chi}&=-2\omega\widehat{\chi}-\alpha,\\
\nabla_4\tr\chi+\frac{1}{2}(\tr\chi)^2&=-|\widehat\chi|^2-2\omega\tr\chi,\\
\nabla_3\widehat{\unchi}+(\tr\unchi)\widehat{\unchi}&=-2\uome\widehat{\unchi}-\una,\\
\nabla_3\tr\unchi+\frac{1}{2}(\trchb)^2&=-|\widehat{\unchi}|^2-2\uome\tr\unchi,\\
\nabla_4\widehat{\unchi}+\frac{1}{2}(\tr\chi)\widehat{\unchi}&=\nabla\widehat\otimes\ue+2\omega\widehat{\unchi}-\frac{1}{2}(\tr\unchi)\widehat{\chi}+\ue\widehat\otimes\ue,\\
\nabla_3\widehat{\chi}+\frac{1}{2}(\tr\unchi)\widehat{\chi}&=\nabla\widehat\otimes\eta+2\uome\widehat{\chi}-\frac{1}{2}(\tr\chi)\widehat{\unchi}+\eta\widehat\otimes\eta,\\
\nabla_4\tr\unchi+\frac{1}{2}(\tr\chi)\tr\unchi &=2\omega\tr\unchi+2\rho-\widehat\chi\cdot\widehat\unchi+2\sdiv\ue+2|\ue|^2,\\
\nabla_3\tr\chi+\frac{1}{2}(\tr\unchi)\tr\chi &=2\uome\tr\chi+2\rho-\widehat\chi\cdot\widehat\unchi+2\sdiv\eta+2|\eta|^2.
\end{split}
\end{align}
We have the Codazzi equations:
\begin{align}
\begin{split}
\sdiv\widehat\chi=\frac{1}{2}\nabla\tr\chi-\zeta\cdot\left(\widehat\chi-\frac{1}{2}\tr\chi\right)-\beta,\\ 
\sdiv\widehat\unchi=\frac{1}{2}\nabla\tr\unchi+\zeta\cdot\left(\widehat\unchi-\frac{1}{2}\tr\unchi\right)+\unb,
\label{codazzi}
\end{split}
\end{align}
the torsion equation:
\begin{equation}
\curl\eta=-\curl\ue=\sigma+\frac{1}{2}\widehat\unchi\wedge\widehat\chi,\label{torsion}
\end{equation}
and the Gauss equation:
\begin{equation}
    \mathbf{K}=-\frac{1}{4}\tr\chi\tr\unchi+\frac{1}{2}\widehat\unchi\cdot\widehat\chi-\rho. \label{gauss}
\end{equation}
Moreover,
\begin{align}
    \begin{split}
\nabla_4\uome&=2\omega\uome+\frac{3}{4}|\eta-\ue|^2-\frac{1}{4}(\eta-\ue)\cdot(\eta+\ue)-\frac{1}{8}|\eta+\ue|^2+\frac{1}{2}\rho,\\
\nabla_3\omega&=2\omega\uome+\frac{3}{4}|\eta-\ue|^2+\frac{1}{4}(\eta-\ue)\cdot(\eta+\ue)-\frac{1}{8}|\eta+\ue|^2+\frac{1}{2}\rho. \label{omom}
    \end{split}
\end{align}
\end{prop}
\begin{proof}
See Section 7.4 of \cite{Ch-Kl}.
\end{proof}
\subsection{Bianchi equations}\label{ssec7.4}
We recall the Bianchi equations in double null foliation, see Proposition 3.2.4 of \cite{kn}.
\begin{prop}\label{prop7.4}
In a double null frame, the Bianchi equations take the form
\begin{align*}
\nabla_4\una+\frac{1}{2}\tr\chi\una&=-\nabla\widehat{\otimes}\unb+h[\una_4],\\
h[\una_4]&=4\omega\una-3(\widehat\unchi\rho-{^*\widehat\unchi}\sigma)+(\zeta-4\ue)\widehat{\otimes}\unb,\\
\nabla_3\unb+2\tr\unchi\unb&=-\sdiv\una+h[\unb_3],\\
h[\unb_3]&=-2\uome\unb+(2\zeta-\eta)\cdot\una,\\
\nabla_4\unb+\tr\chi\unb&=-\nabla\rho+{^*\nabla\sigma}+h[\unb_4],\\
h[\unb_4]&=2\omega\unb+2\widehat\unchi\cdot\beta-3(\ue\rho-{^*\ue}\sigma),\\
\nabla_3\rho+\frac{3}{2}\tr\unchi\,\rho&=-\sdiv\unb+h[\rho_3],\\
h[\rho_3]&=-\frac{1}{2}\widehat\chi\cdot\una+\ze\cdot\unb-2\eta\cdot\unb,\\
\nabla_4\rho+\frac{3}{2}\tr\chi\rho&=\sdiv\beta+h[\rho_4],\\
h[\rho_4]&=-\frac{1}{2}\widehat\unchi\cdot\alpha+\zeta\cdot\beta+2\ue\cdot\beta,\\
\nabla_3\sigma+\frac{3}{2}\tr\unchi\sigma&=-\sdiv{^*\unb}+h[\sigma_3],\\
h[\sigma_3]&=\frac{1}{2}\widehat\chi\cdot{^*\una}-\zeta\cdot{^*\unb}+2\eta\cdot{^*\unb},\\
\nabla_4\sigma+\frac{3}{2}\tr\chi\sigma&=-\sdiv{^*\beta}+h[\sigma_4],\\
h[\sigma_4]&=\frac{1}{2}\widehat\unchi\cdot{^*\alpha}-\zeta\cdot{^*\beta}-2\ue\cdot{^*\beta},\\
\nabla_3\beta+\tr\unchi\beta&=\nabla\rho+{^*\nabla}\sigma+h[\beta_3],\\
h[\beta_3]&=2\uome\beta+2\widehat\chi\cdot\unb+3(\eta\rho+{^*\eta}\sigma),\\
\nabla_4\beta+2\tr\chi\beta&=\sdiv\alpha+h[\beta_4],\\
h[\beta_4]&=-2\omega\beta+(2\zeta+\ue)\alpha,\\
\nabla_3\alpha+\frac{1}{2}\tr\unchi\alpha&=\nabla\widehat{\otimes}\beta+h[\alpha_3],\\
h[\alpha_3]&=4\uome\alpha-3(\widehat\chi\rho+{^*\widehat\chi}\sigma)+(\zeta+4\eta)\widehat{\otimes}\beta.
\end{align*}
\end{prop}
We also derive equations for $\rhoc$ and $\ov{\rho}$.
\begin{lem}\label{prop7.5}
We have the following equations:
\begin{align*}
\nabla_4\unb+\tr\chi\unb&=-\nabla\widecheck\rho+{^*\nabla\sigma}+h[\bb_4],\\
     \nabla_3\beta+\tr\unchi\beta&=\nabla\widecheck\rho+{^*\nabla}\sigma+h[\b_3],\\ 
        \nabla_3\chr+\frac{3}{2}\tr\unchi\chr&=-\sdiv\unb+h[\chr_3],\\ h[\chr_3]&:=\Gag\cdot\bb+h[\rho_3]-\overline{h[\rho_3]},\\
        \nabla_3\ov{\rho}+\frac{3}{2}\tr\unchi\ov{\rho}&=h[\ov{\rho}_3],\\ h[\ov{\rho}_3]&:=\Gag\cdot\bb+\overline{h[\rho_3]},\\
        \nabla_4\chr+\frac{3}{2}\tr\chi\chr&=\sdiv\beta+h[\chr_4],\\
        h[\chr_4]&:=\Gag\cdot\rhoc+h[\rho_4]-\overline{h[\rho_4]},\\
        \nab_4\ov{\rho}+\frac{3}{2}\trch\ov{\rho}&=h[\ov{\rho}_4],\\
        h[\ov{\rho}_4]&:=\Gag\cdot\rhoc+\overline{h[\rho_4]},
\end{align*}
where $h[\bb_4]$, $h[\b_3]$, $h[\rho_3]$ and $h[\rho_4]$ are defined in Proposition \ref{prop7.4}.
\end{lem}
\begin{proof}
Note that $\nab\rhoc=\nab\rho$, hence the first two equations follow directly from Proposition \ref{prop7.4}. On the other hand, applying Corollary \ref{dav}, we infer
\begin{align*}
    \Om\nab_3 \overline{\rho}&=\ov{\omtrchbc\,\rhoc}+\ov{\Om\nab_3\rho}\\
    &=\ov{-\frac{3}{2}\Om\trchb \rho-\Om\sdiv\bb+h[\rho_3]}+\Gag\cdot\rhoc\\
    &=-\frac{3}{2}\Om\trchb\,\ov{\rho}+\ov{h[\rho_3]}+\ov{\bb\cdot\nab\Om}+\Gag\cdot\rhoc,\\
    &=-\frac{3}{2}\Om\trchb\,\ov{\rho}+\ov{h[\rho_3]}+\Gag\cdot\bb.
\end{align*}
So, we have
\begin{align*}
\nab_3\ov{\rho}+\frac{3}{2}\trchb\,\ov{\rho}&=\Gag\cdot\bb+\ov{h[\rho_3]},
\end{align*}
which implies the fourth equation. Combining with the equation of $\nab_3\rho$ in Proposition \ref{prop7.4}, we infer
\begin{align*}
    \nabla_3\rhoc+\frac{3}{2}\trchb\,\rhoc=-\sdiv\unb+\Gag\cdot\bb+h[\rho_3]-\overline{h[\rho_3]},
\end{align*}
which implies the third equation. The proof of the fifth and sixth equations are similar. This concludes the proof of Lemma \ref{prop7.5}.
\end{proof}
We rewrite the Bianchi equations of Proposition \ref{prop7.4} and Lemma \ref{prop7.5} in the following Corollary.
\begin{cor}\label{prop7.6}
\begin{align*}
\nabla_4\una+\frac{1}{2}\tr\chi\una&=-\nabla\widehat{\otimes}\unb+h[\una_4],\\
\nabla_3\unb+2\tr\unchi\unb&=-\sdiv\una+h[\unb_3],\\
\nabla_4\unb+\tr\chi\unb&=-\nabla\widecheck\rho+{^*\nabla\sigma}+h[\unb_4],\\
\nabla_3\widecheck\rho+\frac{3}{2}\tr\unchi\widecheck\rho&=-\sdiv\unb+h[\chr_3], \\
\nabla_4\rhoc+\frac{3}{2}\tr\chi\widecheck\rho&=\sdiv\beta+h[\chr_4],\\
\nab_3\ov{\rho}+\frac{3}{2}\trchb\ov{\rho}&=h[\ov{\rho}_3],\\
\nab_4\ov{\rho}+\frac{3}{2}\trch\ov{\rho}&=h[\ov{\rho}_4],\\
\nabla_3\sigma+\frac{3}{2}\tr\unchi\sigma&=-\sdiv{^*\unb}+h[\sigma_3],\\
\nabla_4\sigma+\frac{3}{2}\tr\chi\sigma&=-\sdiv{^*\beta}+h[\sigma_4],\\
\nabla_3\beta+\tr\unchi\beta&=\nabla\widecheck\rho+{^*\nabla}\sigma+h[\beta_3],\\
\nabla_4\beta+2\tr\chi\beta&=\sdiv\alpha+h[\beta_4],\\
\nabla_3\alpha+\frac{1}{2}\tr\unchi\alpha&=\nabla\widehat{\otimes}\beta+h[\alpha_3],
\end{align*}
where $h[\aa_4]$, $h[\bb_3]$ , ... , $h[\a_3]$ are defined in Proposition \ref{prop7.4} and Lemma \ref{prop7.5}.
\end{cor}
\subsection{Commutation identities}
We recall the following commutation identities for scalar functions, see Proposition 4.8.1 in \cite{kn}.
\begin{prop}\label{prop7.7}
For any scalar function $f$, we have:
\begin{align*}
    [\nabla_4,\nabla]f&=-\chi\cdot\nabla f+(\nabla\log\Omega)\nabla_4 f,\\
    [\nabla_3,\nabla]f&=-\unchi\cdot\nabla f+(\nabla\log\Omega)\nabla_3 f,\\
    [\nabla_4,\nabla_3]f&=2\omega\nabla_3 f -2\uome\nabla_4 f-4\zeta\cdot\nabla f.
\end{align*}
\end{prop}
The following corollary is a direct consequence of Proposition \ref{prop7.7}.
\begin{cor}\label{commcor}
For any scalar function $f$, we have:
\begin{align*}
    [\Om\nabla_4,\nabla]f&=-\Om\chi\cdot\nabla f,\\
    [\Om\nabla_3,\nabla]f&=-\Om\unchi\cdot\nabla f,\\
    [\Om\nab_4,r\nab]f&=\Gag\cdot r\nab f,\\
    [\Om\nab_3,r\nab]f&=\Gab\cdot r\nab f.
\end{align*}
\end{cor}
We also need the commutation identities for more general tensor fields. For this, we record the following commutation lemma.
\begin{lem}\label{comm}
Let $U_{A_1...A_k}$ be an $S$-tangent $k$-covariant tensor on $(\M,\bg)$. Then
\begin{align*}
    [\nabla_4,\nabla_B]U_{A_1...A_k}&=-\chi_{BC}\nab_C U_{A_1...A_k}+F_{4BA_1...A_k},\\
    F_{4BA_1...A_k}&:=(\zeta_B+\ue_B)\nabla_4 U_{A_1...A_k}+\sum_{i=1}^k (\chi_{A_iB}\,\ue_C-\chi_{BC}\,\ue_{A_i}+\in_{A_iC}{^*\beta}_B)U_{A_1...C...A_k},\\
    [\nabla_3,\nabla_B]U_{A_1...A_k}&=-\unchi_{BC}\nabla_C U_{A_1...A_k}+F_{3BA_1...A_k},\\
    F_{3BA_1...A_k}&:=(\eta_B-\ze_B)\nab_3U_{A_1...A_k}+\sum_{i=1}^k(\unchi_{A_iB}\,\eta_C-\unchi_{BC}\,\eta_{A_i}+\in_{A_iC}{^*\unb}_B)U_{A_1...C...A_k},\\
    [\nabla_3,\nabla_4]U_{A_1...A_k}&=F_{34A_1...A_k},\\
    F_{34A_1...A_k}&:=-2\om\nab_3 U+2\omb\nab_4 U+4\ze_B\nab_B U_{A_1...A_k}+2\sum_{i=1}^k(\etab_{A_i}\,\eta_C-\etab_{A_i}\,\eta_C+\in_{A_iC}\si)U_{A_1...C...A_k}.
\end{align*}
\end{lem}
\begin{proof}
It is a direct consequence of Lemma 7.3.3 in \cite{Ch-Kl}.
\end{proof}
Notice that Lemma \ref{comm} implies the following proposition.
\begin{prop}\label{commutation}
We have the following simple schematic consequences of the commutator identities
\begin{align*}
[r\nabla,\Om\nabla_4]&=\Gamma_g\cdot r\nabla+\Gag^{(1)},\\
[r\nabla,\Om\nabla_3]&=\Gamma_b\cdot r\nabla+\Gab^{(1)},\\
[\Om\nab_4,\Om\nab_3]&=\Gag\cdot\nab+r^{-1}\Gag^{(1)}.
\end{align*}
\end{prop}
\subsection{Teukolsky equations}
We state the following \emph{Teukolsky equations} first derived by Teukolsky in the linearized setting in \cite{teukolsky}.
\begin{prop}
We have the following Teukolsky equations for $\a$ and $\aa$:
\begin{equation}
\begin{split}\label{teukolsky}
    \Om\nab_3(\Om\nab_4(r\Om^2\a))+2r\Om^2d_2^*d_2(\Om^2\a)&=-\frac{2\Om}{r}\nab_3(r\Om^2\a)+(\Gag\cdot(\b,\a))^{(1)}, \\
    \Om\nab_4(\Om\nab_3(r\Om^2\aa))+2r\Om^2d_2^*d_2(\Om^2\aa)&=\frac{2\Om}{r}\nab_4(r\Om^2\aa)+(\Gab\cdot(\aa,\bb))^{(1)}.
\end{split}
\end{equation}
\end{prop}
\begin{proof}
    See for example Propositions 3.4.6 and 3.4.7 in \cite{DHRT}.\footnote{Remark that $\Om=1+O(\ep)$ in \cite{DHRT} while $\Om=\frac{1}{2}+O(\ep)$ in this paper.}
\end{proof}
\begin{lem}\label{teulm}
We define the following quantities:
\begin{align}
    \begin{split}
        \ac&:=\frac{1}{r^4}\nab_4(r^5\a)\in \sk_2,\qquad\quad \as:=rd_2\a \in \sk_1,\\
        \aac&:=\frac{1}{r^4}\nab_3(r^5\aa)\in \sk_2,\qquad\quad \aas:=rd_2\aa \in \sk_1.
    \end{split}
\end{align}
Then, we have the following equations:
\begin{align}
    \begin{split}\label{teu}
    \nab_3\ac&=-2d_2^*\as+\frac{4\a}{r}+\Gaa\cdot(\b,\a)^{(1)}+\Gag^{(1)}\cdot(\b,\a),\\
    \nab_4\as+\frac{5}{2}\trch \,\as&=d_2\ac+\Gaa\cdot(\b,\a)^{(1)}+\Gag\c\ac+\Gag^{(1)}\cdot(\b,\a),
    \end{split}
\end{align}
and
\begin{align}
\begin{split}\label{teuaa}
\nab_4\aac&=-2d_2^*\aas+\frac{4\aa}{r}+\Gaa\cdot(\bb,\aa)^{(1)}+(\Gab\cdot(\bb,\aa))^{(1)},\\
\nab_3\aas+\frac{5}{2}\trchb\,\aas&=d_2\aac+\Gaa\cdot(\bb,\aa)^{(1)}+\Gab\c\aac+(\Gab\cdot(\bb,\aa))^{(1)}.
\end{split}
\end{align}
\end{lem}
\begin{proof}
We have from \eqref{teukolsky}
\begin{align*}
    \nab_3(\nab_4(r\a))+2rd_2^*d_2\a=-\frac{4}{r}\nab_3(r\a)+\Gaa\cdot(\b,\a)^{(1)}+\Gag^{(1)}\cdot(\b,\a).
\end{align*}
Hence, we have
    \begin{align*}
        \nab_3\ac&=\nab_3(r^{-4}\nab_4(r^5\a))\\
        &=\nab_3(r^{-4}(4r^3e_4(r)r\a+r^4\nab_4(r\a))\\
        &=\nab_3(4\a+\nab_4(r\a))+\Gaa\cdot(\b,\a)^{(1)}+\Gag^{(1)}\cdot(\b,\a)\\
        &=\frac{4}{r}r\nab_3\a+\nab_3\nab_4(r\a)+\Gaa\cdot(\b,\a)^{(1)}+\Gag^{(1)}\cdot(\b,\a)\\
        &=\frac{4}{r}\nab_3(r\a)-4e_3(r)\frac{\a}{r}-2rd_2^*d_2\a-\frac{4}{r}\nab_3(r\a)+\Gaa\cdot(\b,\a)^{(1)}+\Gag^{(1)}\cdot(\b,\a)\\
        &=\frac{4\a}{r}-2d_2^*\as+\Gaa\cdot(\b,\a)^{(1)}+\Gag^{(1)}\cdot(\b,\a),
    \end{align*}
    which implies the first equation in \eqref{teu}. Applying Proposition \ref{commutation}, we have
    \begin{align*}
        \Om\nab_4\as&=\Om\nab_4\left(r^{-5}(rd_2(r^5\a))\right)\\
        &=-5r^{-6}\Om e_4(r)r^5(rd_2)\a +r^{-5}\Om\nab_4 rd_2 (r^5\a)\\
        &=-\frac{5\Om}{r} rd_2\a+r^{-1}rd_2 r^{-4}(\Om\nab_4 (r^5\a))+\Gaa\cdot(\b,\a)^{(1)}+\Gag^{(1)}\cdot(\b,\a)\\
        &=-\frac{5}{2}\Om\trch\,\as+\Om d_2\ac+\Gaa\cdot(\b,\a)^{(1)}+\Gag\c\ac+\Gag^{(1)}\cdot(\b,\a),
    \end{align*}
     which implies the second equation in \eqref{teu}. The proof of \eqref{teuaa} is similar. This concludes the proof of Lemma \ref{teulm}.
\end{proof}
\subsection{Sobolev inequalities}
\begin{df}\label{L2flux}
For a tensor field $h$ on a null cone $C$, we define its $L^2$--flux:
\begin{equation}
    \|h\|_{2,C}:= \left(\int_C |h|^2 \right)^\frac{1}{2}.
\end{equation}
\end{df}
We recall the following Sobolev inequalities.
\begin{prop}\label{prop7.8}
Let $F$ be a tensor field, tangent to $S:=S(u,\unu)$ at every point. We have the following estimates:
\begin{align*}
        |rF|_{4,S}&\les \|F\|_{2,C_u\cap V(u,\ub)}+\|r\nab F\|_{2,C_u\cap V(u,\ub)}+\|r\nab_4F\|_{2,C_u\cap V(u,\ub)},\\
        |r^\frac{1}{2}|u|^\frac{1}{2}F|_{4,S}&\les\|F\|_{2,\Cb_\ub\cap V(u,\ub)}+\|r\nab F\|_{2,\Cb_\ub\cap V(u,\ub)}+|u|\|\nab_3F\|_{2,\Cb_\ub\cap V(u,\ub)}.
\end{align*} 
\end{prop}
\begin{proof}
See Corollary 3.2.1.1 in \cite{Ch-Kl} and Corollary 4.1.1 in \cite{kn}.
\end{proof}
We also recall the following standard Sobolev inequalities.
\begin{prop}\label{standardsobolev}
Let $F$ be a tensor field on a $2$--sphere $S$. Then, we have
\begin{equation*}
    \sup_{S} r^{\frac{1}{2}}|F| \les |F|_{4,S}+|r\nab F|_{4,S}.
\end{equation*}
\end{prop}
\begin{proof}
See Lemma 4.1.3 of \cite{kn}.
\end{proof}
\subsection{Evolution lemma}
We recall the following evolution lemma, which will be used repeatedly in Sections \ref{sec10} and \ref{ssec11.5}.
\begin{lem}[Evolution lemma]\label{evolutionlemma}
Consider the spacetime $\mathcal{K}$ foliated by a double null foliation $(u,\ub)$. Assume that, for $\ep>0$ small enough, we have
\begin{equation}\label{old2.33}
    |\Gag|\leq\frac{\ep}{r^2}.
\end{equation}
Then, the following holds:
\begin{enumerate}
\item Let $U,F$ be $k$-covariant $S$-tangent tensor fields satisfying the outgoing evolution equation
\begin{equation}
    \nab_4 U +\la_0\trch\, U= F,
\end{equation}
where $\lambda_0\geq 0$. Denoting $\lambda_1=2(\lambda_0-\frac{1}{p})$, we have along $C_u$
\begin{equation}\label{transubU}
|r^{\la_1}U|_{p,S}(u,\ub)\les |r^{\la_1}U|_{p,S}(u,\ub_*)+\int_{\ub}^{\ub_*}  |r^{\la_1}F|_{p,S}(u,\ub')d\ub'.
\end{equation}
Here $\unu_*$ is the value that the function $\unu(p)$ assumes on $\unc_*$.
\item Let $V,\uf$ be k-covariant and S-tangent tensor fields satisfying the incoming evolution equation
\begin{equation}\label{old2.36}
\nab_3 V+\lambda_0\trchb\, V=\uf.
\end{equation}
Denoting $\lambda_1 = 2(\lambda_0 - \frac{1}{p})$, we have along $\unc_\unu$
\begin{equation}\label{transuV}
    |r^{\lambda_1}V|_{p,S}(u,\unu) \lesssim |r^{\lambda_1}V|_{p,S(u_0(\unu),\unu)} + \int_{u_0(\unu)}^{u}  |r^{\lambda_1}\underline{F}|_{p,S}(u',\unu)du',
\end{equation}
where $S(u_0(\unu),\unu)= C_{u_0(\unu)} \cap {\unc_\unu}$ and $u_0(\ub)$ is the unique value of $u$ such that $S(u,\ub)$ is in the future of $\Si_0$ and touches $\Si_0$.
\end{enumerate}
\end{lem}
\begin{proof}
See Lemma 4.1.5 in \cite{kn}.
\end{proof}
\section{Main theorem}\label{sec8}
\subsection{Fundamental norms}\label{ssec8.1}
Our result holds for $s$--asymptotically flat initial data sets in the sense of Definition \ref{def6.3} with $s>3$. We will focus on the case $s\in[4,6]$ in sections \ref{sec8} to \ref{sec11}, and we postpone the necessary adaptations to the case $s\in(3,4)$ and $s>6$ to Appendices \ref{secc} and \ref{secd}. The norms in sections \ref{secRnorms}-\ref{Ostar} are defined with respect to the foliation $(u,\ub)$ in $\kk$ while the norms in section \ref{initialO0} are defined  with respect to the foliation $(u_{(0)},\ub)$ in $\kk_{(0)}$.
\subsubsection{\texorpdfstring{$\mr$}{} norms (curvature)}\label{secRnorms}
We define the norms of Ricci coefficients in the bootstrap region $\kk$. We denote
\begin{equation*}
\mr_0^{S}:=\mr_0^S[\a]+\mr_0^S[\b],\qquad\ur_0^{S}:=\ur_0^{S}[\b]+\ur_0^{S}[\chr,\sigma]+\ur_0^{S}[\unb]+\ur_0^S[\aa]+\ur_0^S[\ov{\rho}],
\end{equation*}
where
\begin{align*}
\mr_0^S[\a]&:=\sup_\kk\sup_{p\in[2,4]}|r^{\frac{s+3}{2}-\frac{2}{p}}\a|_{p,S(u,\unu)},\\
\mr_0^S[\beta]&:=\sup_\kk\sup_{p\in[2,4]}|r^{\frac{7}{2}-\frac{2}{p}}|u|^{\frac{s-4}{2}}\b|_{p,S(u,\unu)},\\
\ur_0^S[\beta]&:=\sup_\kk\sup_{p\in[2,4]}|r^{\frac{s+2}{2}-\frac{2}{p}}|u|^{\frac{1}{2}}\b|_{p,S(u,\unu)},\\
\ur_0^S[\chr,\sigma]&:=\sup_\kk\sup_{p\in[2,4]}|r^{3-\frac{2}{p}}|u|^{\frac{s-3}{2}} (\chr,\sigma) |_{p,S(u,\unu)},\\
\ur_0^S[\unb]&:=\sup_\kk\sup_{p\in[2,4]}|r^{2-\frac{2}{p}}|u|^{\frac{s-1}{2}}\bb |_{p,S(u,\unu)},\\
\ur_0^S[\aa]&:=\sup_\kk\sup_{p\in[2,4]}|r^{1-\frac{2}{p}}|u|^{\frac{s+1}{2}}\aa |_{p,S(u,\unu)},\\
\ur_0^S[\ov{\rho}]&:=\sup_\kk|r^{3}\ov{\rho}|,
\end{align*}
where the $L^p$--norms are defined in Definition \ref{Lpnorms}.
Then we define the norms of flux of curvature components. We denote
\begin{equation*}
    \mr:=\mr_{[1]}+\ur_{[1]},
\end{equation*}
where
\begin{align*}
    \mr_{[1]}&:=\mr_{0}+\mr_0^S+\mr_{1},\qquad\ur_{[1]}:=\ur_{0}+\ur_0^S+\ur_{1},
\end{align*}
with
\begin{align*}
    \mathcal{R}_{0}&:= \left(\mathcal{R}_{0}[\alpha]^2 + \mathcal{R}_{0}[\beta]^2 + \mathcal{R}_{0}[(\chr, \sigma)]^2 + \mathcal{R}_{0}[\unb]^2 \right)^{\frac{1}{2}} ,\\
        \underline{\mathcal{R}}_{0}&:= \left( \underline{\mathcal{R}}_{0}[\beta]^2 + \underline{\mathcal{R}}_{0}[(\chr, \sigma)]^2 + \underline{\mathcal{R}}_{0}[\unb]^2 +
        \underline{\mathcal{R}}_{0}[\una]^2 \right)^{\frac{1}{2}},\\
    \mathcal{R}_{1}&:= \left(\mathcal{R}_{1}[\alpha]^2 + \mathcal{R}_{1}[\beta]^2 + \mathcal{R}_{1}[(\chr, \sigma)]^2 + \mathcal{R}_{1}[\unb]^2 +\mr_1[\a_4]^2\right)^{\frac{1}{2}} ,\\
        \underline{\mathcal{R}}_{1}&:=\left(\underline{\mathcal{R}}_{1}[\beta]^2 + \underline{\mathcal{R}}_{1}[(\chr, \sigma)]^2 + \underline{\mathcal{R}}_{1}[\unb]^2 +
        \underline{\mathcal{R}}_{1}[\una]^2 +\ur_1[\aa_3]^2\right)^{\frac{1}{2}},
\end{align*}
and
\begin{align*}
    \mathcal{R}_{0,1}[w]&:= \sup_{\mathcal{K}} \mathcal{R}_{0,1} [w](u,\unu), \\
    \underline{\mathcal{R}}_{0,1} [w]&:= \sup_{\mathcal{K}} \underline{\mathcal{R}}_{0,1} [w](u,\unu).
\end{align*}
It remains to define $\mathcal{R}_{q}[w](u,\unu)$ and $\underline{\mathcal{R}}_{q}[w](u,\unu)$ for $q=0,1$:
\begin{align*}
         \mathcal{R}_{q}[\alpha](u,\unu)&:= \Vert r^{\frac{s}{2}}(r\nab)^q\a\Vert_{2,C_u \cap V(u,\unu)},\\
         \mathcal{R}_{q}[\beta](u,\unu)&:= \Vert r^{2}|u|^{\frac{s-4}{2}}(r\nab)^q\b\Vert_{2,C_u \cap V(u,\unu)},\\
         \mathcal{R}_{q}[(\chr,\sigma)](u,\unu)&:= \Vert r|u|^{\frac{s-2}{2}}(r\nab)^q (\rhoc,\si) \Vert_{2, C_u \cap V(u,\unu)},\\
         \mathcal{R}_{q}[\unb](u,\unu)&:= \Vert |u|^{\frac{s}{2}}(r\nab)^q\bb \Vert_{2,C_u \cap V(u,\unu)} ,\\
         \underline{\mathcal{R}}_{q}[\beta](u,\unu)&:= \Vert r^{\frac{s}{2}}(r\nab)^q \b \Vert_{2, {\unc_\unu} \cap V(u,\unu)},\\
         \underline{\mathcal{R}}_{q}[(\rhoc,\sigma)](u,\unu)&:=  |u|^{\frac{s-4}{2}} \Vert r^{2}(r\nab)^q (\rhoc,\sigma)\Vert_{2, {\unc_\unu} \cap V(u,\unu)},\\
         \underline{\mathcal{R}}_{q}[\unb](u,\unu)&:= |u|^{\frac{s-2}{2}} \Vert r(r\nab)^q\unb\Vert_{2, {\unc_\unu} \cap V(u,\unu)},\\
         \underline{\mathcal{R}}_{q}[\una](u,\unu)&:= |u|^{\frac{s}{2}} \Vert (r\nab)^q\una \Vert_{2, {\unc_\unu} \cap V(u,\unu)},\\
         \mathcal{R}_1[\a_4](u,\unu)&:=\Vert r^{\frac{s+2}{2}}\nab_4\a\Vert_{2,C_u \cap V(u,\unu)},\\
         \underline{\mathcal{R}}_1[\aa_3](u,\unu)&:= |u|^{\frac{s+2}{2}} \Vert \nab_3\aa\Vert_{2, {\unc_\unu} \cap V(u,\unu)}.
\end{align*}
\subsubsection{\texorpdfstring{$\mo$}{} norms (Ricci coefficients)}\label{sssec8.1.3}
We define the norms of Ricci coefficients in the bootstrap region $\kk$. Denoting
\begin{equation}
    \mathcal{O}:= \mathcal{O}_{[1]}+\uo_{[1]},
\end{equation}
where
\begin{align*}
    \mo_{[1]}&:=\mo_1+\mo_{0}+\mo_0(\Om\omb)+\sup_{\mathcal{K}}\left|r^2 \left(\overline{\tr\chi}-\frac{2}{r}\right)\right|+\sup_{\mathcal{K}}\left|r\left(\Omega-\frac{1}{2}\right)\right|,\\
        \uo_{[1]}&:=\uo_1+\uo_{0}+\mo_0(\Om\om)+\sup_{\mathcal{K}}\left|r|u| \left(\overline{\tr\unchi}+\frac{2}{r}\right)\right|.
\end{align*}
We define $\mo_q$ and $\uo_q$ as follows:
\begin{align*}
\mo_q&:=\mo_q(\widecheck{\Om\trch})+\mo_q(\widehat{\chi})+\mo_q(\eta)+\mo_q(\widecheck{\Om\omb})+\mo_q(\Omc),\qquad\,\, q=0,1,\\
\uo_q&:=\mo_q(\widecheck{\Om\trchb})+\mo_q(\widehat{\unchi})+\mo_q(\ue)+\mo_q(\widecheck{\Om\om}),\qquad\qquad\qquad\; q=0,1,
\end{align*}
where
\begin{equation*}
    \mo_q(\Ga):=\sup_{p\in [2,4]}\sup_\kk\mo_q^{p,S}(\Ga)(u,\ub).
\end{equation*}
It remains to define $\mo_q^{p,S}(\Gamma)(u,\unu)$ for $q=0,1$:
\begin{align*}
    \mo_q^{p,S}(\widecheck{\Om\trch})(u,\unu)&:=|r^{2+q-\frac{2}{p}}|u|^{\frac{s-3}{2}}\nab^q \widecheck{\Om\trch}|_{p,S(u,\unu)} ,\\
    \mo_q^{p,S}(\widecheck{\Om\trchb})(u,\unu)&:=|r^{2+q-\frac{2}{p}}|u|^{\frac{s-3}{2}} \nab^q\widecheck{\Om\trchb}|_{p,S(u,\unu)}, \\
  \mo_q^{p,S}(\widehat\chi)(u,\unu)&:=|r^{2+q-\frac{2}{p}}|u|^{\frac{s-3}{2}}\nabla^q \widehat\chi|_{p,S(u,\unu)}, \\
\mo_q^{p,S}(\widehat\unchi)(u,\unu)&:=|r^{1+q-\frac{2}{p}}|u|^{\frac{s-1}{2}}\nabla^q\widehat\unchi|_{p,S(u,\unu)}, \\
\mo_q^{p,S}(\eta)(u,\unu)&:=|r^{2+q-\frac{2}{p}}|u|^{\frac{s-3}{2}}\nabla^q \eta|_{p,S(u,\unu)}, \\
\mo_q^{p,S}(\etab)(u,\unu)&:=|r^{2+q-\frac{2}{p}}|u|^{\frac{s-3}{2}}\nabla^q \underline{\eta}|_{ p,S(u,\unu)}, \\
\mo_0^{p,S}(\Om\om)(u,\unu)&:=|r^{2-\frac{2}{p}}\Om\om|_{p,S(u,\unu)}, \\
\mo_0^{p,S}(\Om\omb)(u,\unu)&:=|r^{2-\frac{2}{p}}\Om\omb|_{p,S(u,\unu)},\\
\mo_q^{p,S}(\widecheck{\Om\om})(u,\unu)&:=|r^{2+\frac{s-3}{6}-\frac{2}{p}}|u|^{\frac{s-3}{3}}(r\nab)^q\widecheck{\Om\om}|_{p,S(u,\unu)},\\
\mo_q^{p,S}(\widecheck{\Om\omb})(u,\unu)&:=|r^{1-\frac{2}{p}}|u|^{\frac{s-1}{2}}(r\nab)^q\widecheck{\Om\omb}|_{p,S(u,\unu)},\\
\mo_{q}^{p,S}(\Omc)(u,\unu)&:=|r^{1-\frac{2}{p}}|u|^{\frac{s-3}{2}}(r\nab)^q\Omc|_{p,S(u,\unu)}.
\end{align*}
We also define
\begin{align*}
\mobr:=\sum_{q=0}^1\left(\mo_q(\widecheck{\Om\trch})+\mo_q(\eta)+\mo_q(\Omc)\right)+\sup_{\kk}r\left|\Om-\frac{1}{2}\right|+\sup_\kk r^2\left|\ov{\Om\trch}-\frac{1}{r}\right|,
\end{align*}
which appears in sections \ref{sec10} and \ref{ssec11.1}.
\subsubsection{\texorpdfstring{$\uo(\Si_0\setminus K)$}{}-norms and \texorpdfstring{$\Rfk_0$}{}-norms (Initial data on the foliation of \texorpdfstring{$\kk$}{})}\label{Si0setminusK}
Notice that for every $\unu$, there exists a unique leaf $S(u_0(\unu),\unu)$ of $\unc_\unu$, which located in the future of $\Sigma_0\setminus K$ and touches $\Sigma_0\setminus K$. Moreover, we have
\begin{equation*}
    S(u_0(\ub),\ub)\subset \kk_{(0)}.
\end{equation*}
The following norms are defined on the union of spheres $S(u_0(\ub),\ub)$. We define
\begin{equation*}
     \uo(\Si_0\setminus K):=\uo_1(\Si_0\setminus K)+\uo_0(\Si_0\setminus K)+\uo_0(\Si_0\setminus K)(\Om\om)+\sup_{\ub}\left|r^2\left(\Om\trchb+\frac{1}{r}\right)\right|,
\end{equation*}
where
\begin{equation*}
    \uo_q(\Si_0\setminus K):=\sum_{\Gamma\in\{\widecheck{\Om^{-1}\trchb},\,\etab,\,\widecheck{\Om\om}\}}\uo_q(\Sigma_0\setminus K)(\Ga),\qquad q=0,1,
\end{equation*}
with
\begin{equation*}
    \uo_q(\Sigma_0\setminus K)(\Ga):=\sup_{p\in [2,4]}\sup_{\unu}\mo_q^{p,S}(\Sigma_0\setminus K)(\Gamma)(\unu).
\end{equation*}
It remains to define $\uo_q^{p,S}(\Sigma_0\setminus K)(\Gamma)(\unu)$:
\begin{align*}
\mo_q^{p,S}(\Sigma_0\setminus K)(\widecheck{\Om^{-1}\trchb})(\unu)&:=|r^{\frac{s+1}{2}-\frac{2}{p}}(r\nab)^q \widecheck{\Om^{-1}\trchb}|_{p,S(u_0(\unu),\unu)}, \\
\mo_q^{p,S}(\Sigma_0\setminus K)(\underline{\eta})(\unu)&:=|r^{\frac{s+1}{2}-\frac{2}{p}}(r\nab)^q \underline{\eta}|_{ p,S(u_0(\unu),\unu)}, \\
\mo_0^{p,S}(\Sigma_0\setminus K)(\Om\om)(\unu)&:=|r^{2-\frac{2}{p}} \Om\om|_{p,S(u_0(\unu),\unu)}, \\
\mo_q^{p,S}(\Sigma_0\setminus K)(\widecheck{\Om\om})(\unu)&:=|r^{\frac{s+1}{2}-\frac{2}{p}}(r\nab)^q \widecheck{\Om\om}|_{p,S(u_0(\unu),\unu)}.
\end{align*}
We also define
\begin{equation*}
    \uobr(\Sigma_0\setminus K):=\sum_{q=0}^1\left(\uo_q(\Sigma_0\setminus K)(\widecheck{\Om^{-1}\trchb})+\uo_q(\Sigma_0\setminus K)(\widecheck{\Om\om})\right)+\uo_0(\Si_0\setminus K)(\Om\om),
\end{equation*}
which appears in sections \ref{sec10} and \ref{ssec11.1}. \\ \\
We can extend the foliation $(u,\ub)$ to a neighborhood of $\Si_0\setminus K$ in $J^-(\Si_0\setminus K)$\footnote{Recall that $J^-(\Si_0\setminus K)$ denotes the causal past of $\Si_0\setminus K$.} such that it is well defined on $\Si_0\setminus K$. The curvature flux on $\Si_0\setminus K$ is defined by:
\begin{equation}
\Rk:=\int_{\Sigma_0\setminus K}\sum_{l=0}^1 {r}^{s}\left(|{\mathfrak{d}}^l\alpha|^2+|{\mathfrak{d}}^l\beta|^2+|{\mathfrak{d}}^l (\chr,\sigma)|^2+|{\mathfrak{d}}^l\unb|^2+|{\mathfrak{d}}^l \aa|^2\right)+\sup_{\Sigma_0\setminus K} |{r}^3\overline{\rho}|^2,
\end{equation}
with
\begin{equation*}
    {\mathfrak{d}}:=(r\nabla,r\nabla_3,r\nabla_4).
\end{equation*}
\subsubsection{\texorpdfstring{$\mo^*(\Cb_*)$}{} norms (Ricci coefficients on the last slice)}\label{Ostar}
We define the following norms on the last slice $\Cb_*$ in the foliation $(u,\ub)$ of $\kk$. We denote
\begin{equation*}
    \mo^{*}(\unc_*):=\mo^*_{1}+\mo^{*}_{0}+\sup_{\unc_*}\left|r^2\left(\overline{\tr\chi}-\frac{2}{r}\right)\right|+\mo_2^*(\trchbc),
\end{equation*}
where
\begin{align*}
    \mo_q^{*}&:=\mo_q^*(\trchc)+\mo_q^*(\hch)+\mo_q^*(\trchbc)+\mo_q^*(\hchb)+\mo_q^*(\ze),
\end{align*}
with
\begin{equation*}
    \mo_q^{*}(\Ga):=\sup_{p\in[2,4]}\sup_{\unc_*}\mo_q^{*p,S}(\Ga).
\end{equation*}
It remains to define $\mo_q^{*p,S}(\Ga)$:
\begin{align*}
    \mo_q^{*p,S}(\widehat\chi)&:=|r^{2-\frac{2}{p}}|u|^{\frac{s-3}{2}}(r\nab)^q\widehat\chi|_{p,S(u,\unu_*)},\\
    \mo_q^{*p,S}(\widehat\chib)&:=|r^{1-\frac{2}{p}}|u|^{\frac{s-1}{2}}(r\nab)^q\widehat\chib|_{p,S(u,\unu_*)},\\
    \mo_q^{*p,S}(\trchc)&:=|r^{2-\frac{2}{p}}|u|^{\frac{s-3}{2}}(r\nab)^q\widecheck{\tr\chi}|_{p,S(u,\unu_*)},\\
    \mo_q^{*p,S}(\trchbc)&:=|r^{2-\frac{2}{p}}|u|^{\frac{s-3}{2}}(r\nab)^q\widecheck{\tr\chib}|_{p,S(u,\unu_*)},\\
    \mo_q^{*p,S}(\ze)&:=|r^{2-\frac{2}{p}}|u|^{\frac{s-3}{2}}(r\nab)^q\ze|_{p,S(u,\unu_*)}.
\end{align*}
\subsubsection{\texorpdfstring{$\mo_{(0)}$}{} and \texorpdfstring{$\Rfk_{(0)}$}{} norms (Initial data on the foliation of \texorpdfstring{$\kk_{(0)}$)}{}}\label{initialO0}
In this section, all the norms are defined by the initial layer foliation $(u_{(0)},\ub)$ in $\kk_{(0)}$. We define
\begin{align*}
    \OO_{(0)}&:=\sum_{\Ga_{(0)}}\mo_{(0)}(\Ga_{(0)})+\sup_{\kk_{(0)}}\left|r_{(0)}^2 \left(\trch_{(0)}-\frac{2}{r_{(0)}}\right)\right|+\sup_{\kk_{(0)}}\left|r_{(0)}^2\left(\trchb_{(0)}+\frac{2}{r_{(0)}}\right)\right|\\
&+\sup_{\kk_{(0)}}\left|r_{(0)}^2\om_{(0)}\right|+\sup_{\kk_{(0)}}\left|r_{(0)}^2\omb_{(0)}\right|+\sup_{\kk_{(0)}}\left|r_{(0)}\left(\Om_{(0)}-\frac{1}{2}\right)\right|,
\end{align*}
where
\begin{equation*}
    \mo_{(0)}(\Ga_{(0)}):=\sum_{q=0}^2\sup_{\kk_{(0)}}\sup_{p\in [2,4]}\left|r_{(0)}^{\frac{s+1}{2}-\frac{2}{p}}(r_{(0)}\nab)^q\Ga_{(0)}\right|_{p,S_{(0)}(u_{(0)},\ub)},
\end{equation*}
with
\begin{equation*}
    \Ga_{(0)}\in\{\trchc_{(0)},\trchbc_{(0)},\hch_{(0)},\hchb_{(0)},\eta_{(0)},\etab_{(0)},\omc_{(0)},\ombc_{(0)}\}.
\end{equation*}
We define the curvature flux on the initial hypersurface:
\begin{align*}
    \Rfk_{(0)}^2:=\sum_{l=0}^2\int_{\Si_0\setminus K}r_{(0)}^s\left|\dk_{(0)}^{l}\left(\a_{(0)},\b_{(0)},\rhoc_{(0)},\si_{(0)},\bb_{(0)},\aa_{(0)}\right)\right|^2
    +\sup_{\Si_0\setminus K}|r_{(0)}^3{\rho_{(0)}}|,
\end{align*}
where
\begin{equation}
    \dk_{(0)}:=\{r_{(0)}\nab_3,r_{(0)}\nab_4, r_{(0)}\nab\}.
\end{equation}
\subsubsection{\texorpdfstring{$\osc$}{}--norms (Oscillation control)}
We have two foliations $(u,\ub)$ and $(u_{(0)},\ub)$ in the initial layer $\kk_{(0)}$ which will be compared using $\osc$--norms. We denote $(f,\la)$ the change of frame from $(u,\ub)$ to $(u_{(0)},\ub)$, defined in Lemma \ref{lemchange} in section \ref{nullframetrans}. We define
\begin{align*}
    \osc:=\osc(f)+\osc(\la)+\osc(r),
\end{align*}
where
\begin{align*}
    \osc(f)=\sup_{\kk_{(0)}}\left|r^\frac{s-1}{2}\dk^{\leq 1} f\right|,\qquad \osc(\la)=\sup_{\kk_{(0)}}r|\ovla|,\qquad\osc(r):=\sup_{\kk_{(0)}}|r_{(0)}-r|,
\end{align*}
with 
\begin{align*}
    \ovla:=\la-1,\qquad \dk:=\{r\nab_3,r\nab_4,r\nab\}.
\end{align*}
See Definition \ref{generalchange} in section \ref{ssec12.1} for a generalized definition of $\osc$--norms.
\subsection{Main theorem}\label{ssec8.3}
The goal of this paper is to prove the following theorem, which provides a new proof of the seminal result obtained by Klainerman and Nicol\`o in \cite{kn}.
\begin{thm}[Main Theorem]\label{th8.1}
Consider an initial data set $(\Si_0,g,k)$ $s$--asymptotically flat in the sense of Definition \ref{def6.3} with $s>3$. Assume that we have the following control of the initial layer region $\kk_{(0)}$ defined in section \ref{sssec7.1.1}
\begin{equation}
    \mo_{(0)}\leq \ep_0,\qquad \Rfk_{(0)}\leq \ep_0,
\end{equation}
where $\mo_{(0)}$, $\mathfrak{R}_{(0)}$ are defined in section \ref{initialO0} and $\ep_0>0$ is small enough.\\ \\
Then, the initial layer $\kk_{(0)}$ has a unique development $(\M,\bg)$ in its future domain of dependence with the following properties:
\begin{enumerate}
    \item $(\M,\bg)$ can be foliated by a double null foliation $(u,\ub)$. Moreover, the outgoing cones $C_u$ are complete for all $u\leq u_0$.
    \item The norms $\mo$ and $\mr$ defined in section \ref{ssec8.1} for $s\in [4,6]$, and in sections \ref{secc1} and \ref{ssecd1} respectively for $s\in (3,4)$ and $s>6$ satisfy
\begin{equation}\label{finalesti}
    \mo\lesssim\epsilon_0,\qquad\mr\lesssim\epsilon_0.
\end{equation}
\end{enumerate}
\end{thm}
\begin{rk}
Theorem \ref{th8.1} contains also a number of important conclusions following from \eqref{finalesti}: peeling properties, complete future infinity, Bondi mass formula and so on, see Chapter 8 in \cite{kn}.
\end{rk}
\begin{rk}
Theorem \ref{th8.1} is proved in section \ref{ssec8.6} for $s\in [4,6]$, and extended to $s\in(3,4)$ and $s>6$ in Appendices \ref{secc} and \ref{secd}.
\end{rk}
The proof of Theorem \ref{th8.1} is given in section \ref{ssec8.6}. It hinges on a sequence of basic theorems stated in section \ref{ssec8.4}, concerning estimates for $\mo$ and $\mr$ norms. \\ \\
We choose $\ep_0$ and $\ep$ small enough such that
\begin{equation*}
    \ep\ll \de_0<1, \qquad \ep:=\ep_0^{\frac{2}{3}},
\end{equation*}
where we recall that $\de_0$ is the height of the initial layer, see \eqref{defde0}. Here, $A\ll B$ means that $CA<B$ where $C$ is the largest universal constant among all the constants involved in the proof via $\lesssim$.
\subsection{Main intermediate results}\label{ssec8.4}
The following lemma provides comparisons between $r_{(0)},r,w,u_{(0)},\ub,u$.
\begin{lem}\label{equivalence}
Under the assumptions
\begin{align}\label{equiass}
    \mo_{(0)}\leq \ep_0,\qquad \mo\leq\ep,\qquad \osc\leq\ep,
\end{align}
we have in the initial layer region $\kk_{(0)}$
\begin{align}\label{equivalenceused}
|r_{(0)}-w|\les\ep_0\log r_{(0)}, \qquad |u_{(0)}-u|\les\ep\log r_{(0)},
\end{align}
and in the bootstrap region $\kk$
\begin{equation}\label{equivalencerubu}
 \left|r-\frac{\ub-u}{2}\right|\les \ep\log r.
\end{equation}
\end{lem}
\begin{proof}
See Appendix \ref{secb}.
\end{proof}
\begin{rk}
Recall that we have $|u|\leq \ub$, and $u<0$ in $\kk$. Together with \eqref{equivalencerubu}, this yields
\begin{equation}\label{equiuubr}
   |u|\leq \ub\simeq r \quad \mbox{ in }\kk.
\end{equation}
In the sequel, we will use \eqref{equivalenceused}, \eqref{equivalencerubu} and \eqref{equiuubr} frequently without explicitly mentioning them.
\end{rk}
The following theorems are important intermediate steps in the proof of Theorem \ref{th8.1}.
\newtheorem*{m0}{Theorem M0}
\begin{m0}
Under the assumptions
\begin{align*}
    \mo_{(0)}\leq \ep_0,\qquad \mathfrak{R}_{(0)}\leq \ep_0, \qquad \mo\leq\ep,\qquad \osc\leq\ep,
\end{align*}
we have
\begin{equation}
    \mathfrak{R}_0\les \ep_0,\qquad \uobr(\Si_0\setminus K)\les\ep_0.
\end{equation}
If in addition we assume that
\begin{equation*}
    \mobr\les\ep_0,
\end{equation*}
then, we have
\begin{equation}\label{oscm0control}
    \osc\les\ep_0,\qquad \uo(\Si_0\setminus K)\les \ep_0.
\end{equation}
\end{m0}
Theorem M0 is proved in section \ref{ssec11.1}. The proof is based on null frame transformation formulae introduced by Klainerman and Szeftel in \cite{kl-sz1}, see Proposition \ref{transformation} in section \ref{nullframetrans}.
\begin{rk}
The second part of Theorem M0 implies that we need an estimate for $\mobr$ to control $\osc$. To this end, we first estimate the quantities in $\uobr(\Si_0\setminus K)$, which are nonlinearly dependent on $\osc$. Next, we apply them as initial data to estimate part of the Ricci coefficients, i.e. $\mobr$, see \eqref{eq3.17} in Theorem M3. Then, we obtain the oscillation control, which allows us to estimate all the initial data, i.e. the control \eqref{oscm0control} for $\uo(\Si_0\setminus K)$.
\end{rk}
\newtheorem*{m1}{Theorem M1}
\begin{m1}
Assume that
\begin{equation}
    \mo_{(0)}\leq\ep_0,\qquad \mo\leq\epsilon,\qquad \ur_0^S\leq\epsilon.
\end{equation}
Then, we have
\begin{equation}
    \mr\les \mathfrak{R}_0.
\end{equation}
\end{m1}
Theorem M1 is proved in section \ref{sec9}. The proof is based on the $r^p$--weighted method introduced by Dafermos and Rodnianski in \cite{Da-Ro}.
\newtheorem*{m2}{Theorem M2}
\begin{m2}
Assume that
\begin{equation}
    \mo^*(\unc_*)\leq\epsilon,\qquad \ur_0^S\leq \Delta_0,\qquad\mo_{(0)}(\Sigma_0\cap\unc_*)\leq \mathcal{I}_0.
\end{equation}
Then, we have
\begin{equation}
    \mo^*(\unc_*) \lesssim \Delta_0+\mathcal{I}_0+\epsilon^2.
\end{equation}
\end{m2}
Theorem M2 is proved in section \ref{ssec11.5}. The proof is done by integrating the transport equations along the null generator $\Lb$ of the last slice $\Cb_*$ and applying elliptic estimates on $2$--spheres of the geodesic foliation of $\Cb_*$.
\newtheorem*{m3}{Theorem M3}
\begin{m3}
Assume that
\begin{align}
        \mo\leq\epsilon,\qquad \mr_0^S+\ur_0^S\leq\Delta_0,\qquad\mo^*(\unc_*)\leq \mathcal{I}_*, \qquad\uobr(\Si_0\setminus K)\leq \IIbr,
\end{align}
then, we have
\begin{equation}\label{eq3.17}
    \mobr\les \IIbr+\mathcal{I}_*+\Delta_0+\epsilon^2.
\end{equation}
If we assume in addition that
\begin{equation}
    \uo(\Si_0\setminus K)\leq\II_0,
\end{equation}
then, we have
\begin{equation}
    \mo\lesssim \mathcal{I}_0+\mathcal{I}_*+\Delta_0+\epsilon^2.
\end{equation}
\end{m3}
Theorem M3 is proved in section \ref{sec10}. The proof is done by integrating the transport equations along the outgoing and incoming null cones and applying elliptic estimate on $2$--spheres of the double null foliation of the spacetime $\kk$.
\newtheorem*{m4}{Theorem M4}
\begin{m4}
We consider the spacetime $\kk$ and its double null foliation $(u,\unu)$, which satisfy the hypotheses
\begin{equation}
\mo\les\epsilon_0,\qquad \mr\les\epsilon_0,\qquad \osc\les\ep_0.
\end{equation}
Then, we can extend the spacetime $\kk=V(u_0,\unu_*)$ and the double null foliation $(u,\unu)$ to a new spacetime $\wideparen{\kk}=\wideparen{V}(u_0,\unu_*+\nu)$, where $\nu$ is sufficiently small, and an associated double null foliation $(\upa,\unu)$. Moreover, the new foliation $(\upa,\ub)$ is geodesic on the new last slice $\Cb_{**}:=\Cb_{\ub_*+\nu}$ and
the new norms satisfy
\begin{align}
\wideparen{\mo}\les\ep_0,\qquad\wideparen{\mr}\les\ep_0,\qquad \wideparen{\osc}\les\ep_0.
\end{align}
\end{m4}
Theorem M4 is proved in section \ref{ssec12.1}. The proof is based on local existence, and the null frame transformation formulas of Proposition \ref{transformation}. We also need to reapply Theorems M1, M2 and M3 in the extended spacetime region $\wideparen{\kk}$.
\begin{rk}
When compared to \cite{kn}, we have the following similarities and differences:
\begin{enumerate}
    \item We treat the general decay $s>3$ in Theorem \ref{th8.1} while the main result of \cite{kn}\footnote{We also reobtain the strong peeling properties of \cite{kncqg} that correspond to $s>7$.} correspond to the particular case $s=4$ in Theorem \ref{th8.1}.
    \item In \cite{kn}, third order derivatives of Ricci coefficients and second order derivatives of curvature are estimated. In this paper, we only estimate first order derivative of both Ricci coefficients and curvature respectively in Theorems M0-M2-M3 and M1.
    \item To estimate the curvature components in Theorem M1, instead of using vectorfield method as in \cite{kn}, we use the $r^p$--weighted estimates introduced by Dafermos and Rodnianski in \cite{Da-Ro}.
    \item On the last slice $\Cb_*$, a canonical foliation is used in \cite{kn} while we use a geodesic foliation in Theorem M2.
\item The estimates of Ricci coefficients and the extension argument, i.e. the proof of Theorems M3 and M4, are similar to \cite{kn}.
\end{enumerate}
\end{rk}
\subsection{Proof of the main theorem}\label{ssec8.6}
We now use Theorems M0-M4 to prove Theorem \ref{th8.1}.
\begin{df}\label{bootstrap}
Let $\aleph(\unu_*)$ the set of spacetimes $\kk$ associated with a double null foliation $(u,\unu)$ which satisfy the following properties:
\begin{enumerate}
    \item $\kk=V(u_0,\unu_*)$.
    \item The foliation on $\unc_*$ is geodesic.
    \item We have the following bootstrap bounds:
\begin{align}
    \mo\leq\epsilon,\qquad\mr\leq\ep,\qquad \osc\leq\ep.\label{B2}
\end{align}
\end{enumerate}
\end{df}
\begin{df}\label{defboot}
We denote $\mathcal{U}$ the set of values $\unu_*$ such that $\aleph(\unu_*)\ne\emptyset$.
\end{df}
Applying the local existence result of Theorem 10.2.1 in \cite{Ch-Kl} and the assumption $\mo_{(0)}\leq\ep_0$, we deduce that \eqref{B2} holds if $\ub_*$ is sufficiently small. So, we have $\mathcal{U}\ne\emptyset$.\\ \\
Define $\unu_*$ to be the supremum of the set $\mathcal{U}$. We want to prove $\unu_*=+\infty$. We assume by contradiction that $\unu_*$ is finite. In particular we may assume $\unu_*\in\mathcal{U}$. We consider the region $\kk=V(u_0,\unu_*)$. Recall that we have
\begin{equation}
    \mo_{(0)}\leq\ep_0,\qquad \mathfrak{R}_{(0)}\leq\ep_0,
\end{equation}
according to the assumptions of Theorem \ref{th8.1}. Applying Theorems M0 and M1, we obtain
\begin{equation}\label{mrest}
    \mr\les \mathfrak{R}_0\les\ep_0,\qquad \uobr(\Si_0\setminus K)\les \ep_0.
\end{equation}
Then, applying Theorem M2 and the hypothesis $\mo\leq\epsilon$, we obtain\footnote{By definition, $\mo\leq\ep$ implies $\mo^*(\Cb_*)\leq\ep$.}
\begin{equation*}
    \mo^{*}(\unc_*)\les\ep_0.
\end{equation*}
In view of the above, the quantities in Theorem M3 satisfy the following conditions:
\begin{equation*}
\mathcal{I}_*\les\ep_0,\qquad \Delta_0\les \ep_0,\qquad \IIbr\les\ep_0.
\end{equation*}
Applying the first part of Theorem M3, we obtain
\begin{equation}
    \mobr\les\ep_0.
\end{equation}
Then, we can apply the second part of Theorem M0 to deduce
\begin{equation}\label{M0final}
    \osc\les\ep_0,\qquad\quad\uo(\Si_0\setminus K)\les \ep_0.
\end{equation}
Note that \eqref{M0final} implies $\II_0\leq\ep_0$ where $\II_0$ is defined in Theorem M3. Thus, we may apply the second part of Theorem M3 to obtain
\begin{equation}\label{moest}
    \mo\les\ep_0.
\end{equation}
Notice that \eqref{mrest}, \eqref{M0final} and \eqref{moest} implies that we can apply Theorem M4 to extend $\kk$ to $\wideparen{\kk}:=\wideparen{V}(u_0,\unu_*+\nu)$ for a $\nu$ sufficiently small. We denote $\wideparen{\mo}$ and $\wideparen{\mr}$ the norms associated to the new foliation $(\upa,\ub)$. We have
\begin{equation*}
    \wideparen{\mo}\les\ep_0,\qquad \wideparen{\mr}\les\ep_0,\qquad\wideparen{\osc}\les\ep_0,
\end{equation*}
as a consequence of Theorem M4. We deduce that $\wideparen{V}(u_0,\unu_*+\nu)$ with the double null foliation $(\upa,\ub)$ satisfies all the properties in Definition \ref{bootstrap}, and so $\aleph(\unu_*+\nu)\ne\emptyset$, which is a contradiction. Thus, we have $\unu_*=+\infty$, which implies property 1 of Theorem \ref{th8.1}. Moreover, we have
\begin{equation}
    \mo\les \ep_0,\qquad \mr\les\ep_0,
\end{equation}
in the whole exterior region, which implies property 2 of Theorem \ref{th8.1}. This concludes the proof of Theorem \ref{th8.1}.
\section{Curvature estimates (Theorem M1)}\label{sec9}
In this section, we prove Theorem M1 by the $r^p$--weighted estimate method introduced in \cite{Da-Ro} and applied to Bianchi equations in \cite{holzegel} and \cite{kl-sz}. For convenience, we recall the statement below.
\begin{m1}
Assume that
\begin{equation}
    \mo_{(0)}\leq\ep_0,\qquad \mo\leq\epsilon,\qquad \ur_0^S\leq\epsilon,\qquad \osc\leq\ep.
\end{equation}
Then, we have
\begin{equation}
    \mr\les \mathfrak{R}_0. 
\end{equation}
\end{m1}
\begin{figure}
  \centering
  \includegraphics[width=1\textwidth]{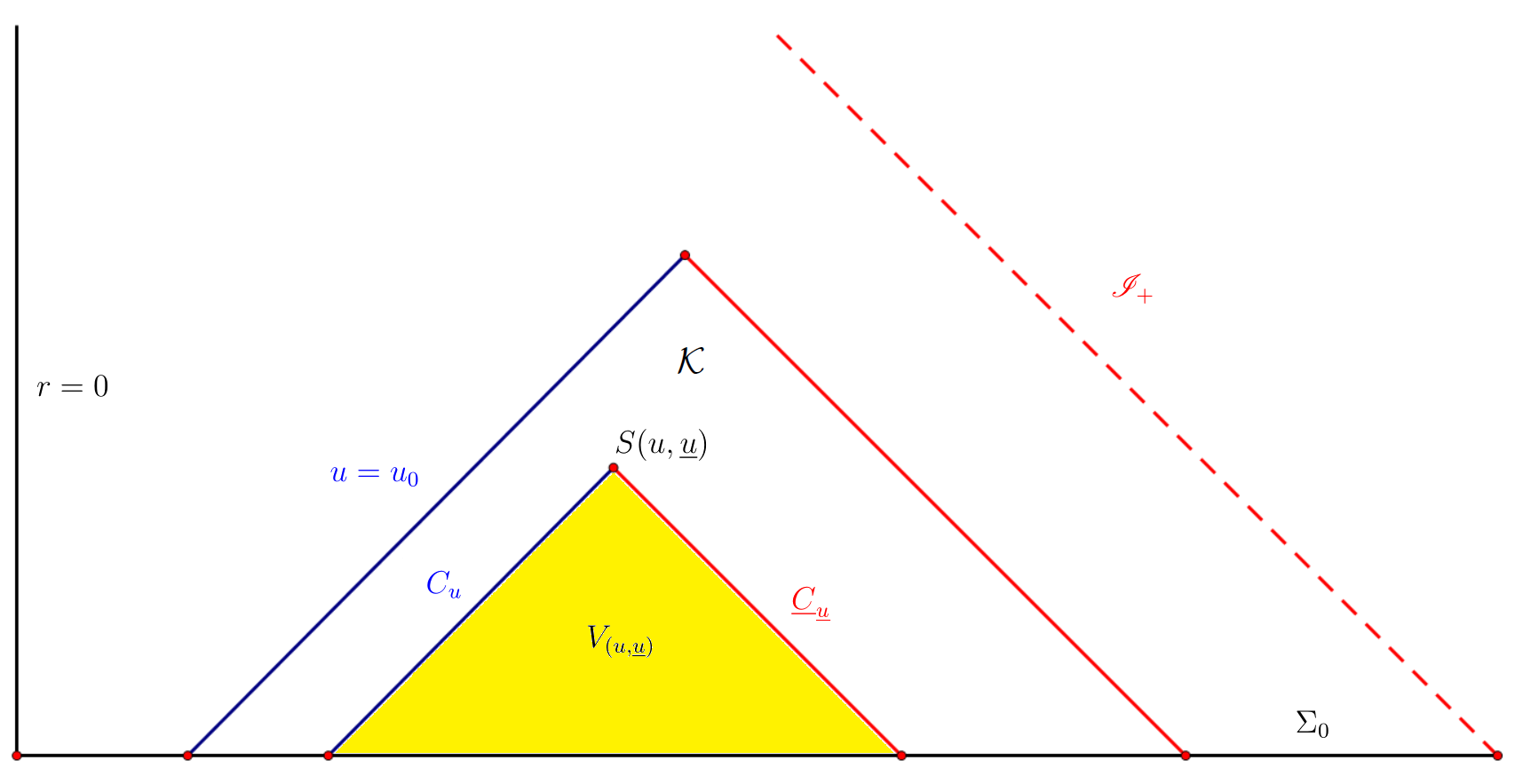}
  \caption{Domain of integration $V(u,\ub)$}\label{fig4}
\end{figure}
\begin{lem}\label{decayGagGabGaa}
We have the following estimates:
\begin{equation}\label{estGagba}
    |\Gag|\lesssim \frac{\epsilon}{r^2|u|^{\frac{s-3}{2}}},\qquad\quad |\Gab|\lesssim \frac{\epsilon}{r|u|^{\frac{s-1}{2}}},\qquad\quad |\Gaa|\lesssim\frac{\epsilon}{r^2},
\end{equation}
and
\begin{equation}\label{estnabGagb}
    |r^{2-\frac{2}{p}}|u|^\frac{s-3}{2}\Gag^{(1)}|_{p,S}\les\ep,\qquad |r^{1-\frac{2}{p}}|u|^\frac{s-1}{2}\Gab^{(1)}|_{p,S}\les\ep,\qquad p\in [2,4].
\end{equation}
\end{lem}
\begin{proof}
    It follows directly from the assumption $\mo\leq\ep$, Definition \ref{gammag} and the Sobolev inequality of Proposition \ref{standardsobolev}.
\end{proof}
\subsection{Estimates for general Bianchi pairs}
The following lemma provides the general structure of Bianchi pairs. It will be used repeatedly in this section. See Lemma 8.24 in \cite{kl-sz} for an analog version under the assumption of axial polarization.
\begin{lem}\label{keypoint}
Let $k=1,2$ and $a_{(1)}$, $a_{(2)}$ real numbers. Then, we have the following properties.
\begin{enumerate}
    \item If $\psi_{(1)},h_{(1)}\in\sk_k$ and $\psi_{(2)},h_{(2)}\in \sk_{k-1}$ satisfying
    \begin{align}
    \begin{split}\label{bianchi1}
        \nab_3(\psi_{(1)})+a_{(1)}\trchb\,\psi_{(1)}&=-kd_k^*(\psi_{(2)})+h_{(1)},\\
        \nab_4(\psi_{(2)})+a_{(2)}\trch\,\psi_{(2)}&=d_k(\psi_{(1)})+h_{(2)}.
    \end{split}
    \end{align}
Then, the pair $(\psi_{(1)},\psi_{(2)})$ satisfies for any real number $p$
\begin{align}
\begin{split}\label{div}
       &\bdiv(r^p |\psi_{(1)}|^2e_3)+k\bdiv(r^p|\psi_{(2)}|^2e_4)\\
       +&\left(2a_{(1)}-1-\frac{p}{2}\right)r^{p}\trchb|\psi_{(1)}|^2+k\left(2a_{(2)}-1-\frac{p}{2}\right)r^{p}\trch|\psi_{(2)}|^2\\
       =&2k r^p\sdiv(\psi_{(1)}\cdot\psi_{(2)})
       +2r^p\psi_{(1)}\cdot h_{(1)}+2kr^p\psi_{(2)}\cdot h_{(2)}-2r^p\omb|\psi_{(1)}|^2\\
       -&2kr^p\om|\psi_{(2)}|^2+pr^{p-1}\left(e_3(r)-\frac{r}{2}\trchb\right)|\psi_{(1)}|^2+kpr^{p-1}\left(e_4(r)-\frac{r}{2}\tr\chi\right)|\psi_{(2)}|^2.
\end{split}
\end{align}
    \item If $\psi_{(1)},h_{(1)}\in\sk_{k-1}$ and $\psi_{(2)},h_{(2)}\in\sk_k$ satisfying
    \begin{align}
        \begin{split}\label{bianchi2}
        \nab_3(\psi_{(1)})+a_{(1)}\trchb\,\psi_{(1)}&=d_k(\psi_{(2)})+h_{(1)},\\
        \nab_4(\psi_{(2)})+a_{(2)}\trch\,\psi_{(2)}&=-kd_k^*(\psi_{(1)})+h_{(2)}.
        \end{split}
    \end{align}
Then, the pair $(\psi_{(1)},\psi_{(2)})$ satisfies for any real number $p$
\begin{align}
\begin{split}\label{div2}
       &k\bdiv(r^p |\psi_{(1)}|^2e_3)+\bdiv(r^p|\psi_{(2)}|^2e_4)\\
       +&k\left(2a_{(1)}-1-\frac{p}{2}\right)r^{p}\trchb|\psi_{(1)}|^2+\left(2a_{(2)}-1-\frac{p}{2}\right)r^{p}\trch|\psi_{(2)}|^2\\
       =&2 r^p \sdiv(\psi_{(1)}\cdot\psi_{(2)})
       +2kr^p\psi_{(1)}\cdot h_{(1)}+2r^p\psi_{(2)}\cdot h_{(2)}-2k r^p\omb|\psi_{(1)}|^2\\
       -&2r^p\om|\psi_{(2)}|^2+k pr^{p-1}\left(e_3(r)-\frac{r}{2}\trchb\right)|\psi_{(1)}|^2+pr^{p-1}\left(e_4(r)-\frac{r}{2}\trch\right)|\psi_{(2)}|^2.
\end{split}
\end{align}
\end{enumerate}
\end{lem}
\begin{rk}
    Note that the Bianchi equations can be written as systems of equations of the type \eqref{bianchi1} and \eqref{bianchi2}. In particular
    \begin{itemize}
        \item the Bianchi pair $(\a,\b)$ satisfies \eqref{bianchi1} with $k=2$, $a_{(1)}=\frac{1}{2}$, $a_{(2)}=2$,
        \item the Bianchi pair $(\b,(\rhoc,-\si))$ satisfies \eqref{bianchi1} with $k=1$, $a_{(1)}=1$, $a_{(2)}=\frac{3}{2}$,
        \item the Bianchi pair $((\rhoc,\si),\bb)$ satisfies \eqref{bianchi2} with $k=1$, $a_{(1)}=\frac{3}{2}$, $a_{(2)}=1$,
        \item the Bianchi pair $(\bb,\aa)$ satisfies \eqref{bianchi2} with $k=2$, $a_{(1)}=2$, $a_{(2)}=\frac{1}{2}$.
    \end{itemize}
\end{rk}
\begin{proof}
By a direct computation, we have
\begin{align*}
    \bdiv e_4=\D_\ga e_{{4}}^\ga=-\frac{1}{2}\g(\D_4e_4,e_3)-\frac{1}{2}\g(\D_3e_4,e_4)+\g^{AB} \g(\D_A e_4,e_B)=-2\om+\trch,
\end{align*}
and similarly
\begin{align*}
    \bdiv e_3=-2\omb+\trchb.
\end{align*}
For $\psi_{(1)}$ and $\psi_{(2)}$ satisfying \eqref{bianchi1}, we compute
\begin{align*}
    &\bdiv(r^p |\psi_{(1)}|^2 e_3)+k\bdiv(r^p |\psi_{(2)}|^2 e_4)\\
    =&\nab_3(r^p |\psi_{(1)}|^2)+ \bdiv(e_3)r^p|\psi_{(1)}|^2+k\nab_4(r^p |\psi_{(2)}|^2)+k\bdiv(e_4)r^p|\psi_{(2)}|^2\\
    =&pr^{p-1}e_3(r)|\psi_{(1)}|^2+2r^p\psi_{(1)}\cdot\nab_3\psi_{(1)}+r^{p}(\trchb-2\omb) |\psi_{(1)}|^2\\
    +&kpr^{p-1}e_4(r)|\psi_{(2)}|^2+2kr^p\psi_{(2)}\cdot\nab_4\psi_{(2)}+kr^p(\trch-2\om)|\psi_{(2)}|^2\\
    =&pr^{p-1}e_3(r)|\psi_{(1)}|^2+2r^p\psi_{(1)}\cdot\left(-a_{(1)}\trchb\psi_{(1)}-kd_k^*\psi_{(2)}+h_{(1)}\right)+r^p\trchb|\psi_{(1)}|^2-2r^{p}\omb|\psi_{(1)}|^2\\
    +&kpr^{p-1}e_4(r)|\psi_{(2)}|^2+2kr^p\psi_{(2)}\cdot\left(-a_{(2)}\trch\psi_{(2)}+d_k\psi_{(1)}+h_{(2)}\right)+k r^p\trch|\psi_{(2)}|^2-2k r^p\om|\psi_{(2)}|^2\\
    =&pr^{p-1}\left(e_3(r)-\frac{r}{2}\trchb\right)|\psi_{(1)}|^2+2r^p\psi_{(1)}\cdot\left(-kd_k^*\psi_{(2)}+h_{(1)}\right)+\left(1-2a_{(1)}+\frac{p}{2}\right)r^p\trchb|\psi_{(1)}|^2\\
    +&kpr^{p-1}\left(e_4(r)-\frac{r}{2}\trch\right)|\psi_{(2)}|^2+2kr^p\psi_{(2)}\cdot\left(d_k\psi_{(1)}+h_{(2)}\right)+\left(1-2a_{(2)}+\frac{p}{2}\right)kr^p\trch|\psi_{(2)}|^2\\
    -&2r^{p}\omb|\psi_{(1)}|^2-2k r^p\om|\psi_{(2)}|^2\\
    =&2k\sdiv(\psi_{(1)}\cdot\psi_{(2)})+pr^{p-1}\left(e_3(r)-\frac{r}{2}\trchb\right)|\psi_{(1)}|^2 +k pr^{p-1}\left(e_4(r)-\frac{r}{2}\trch\right)|\psi_{(2)}|^2\\
    +&2r^p\psi_{(1)}\cdot h_{(1)}+2kr^p\psi_{(2)}\cdot h_{(2)}+\left(1-2a_{(1)}+\frac{p}{2}\right)r^p\trchb|\psi_{(1)}|^2+\left(1-2a_{(2)}+\frac{p}{2}\right)kr^p\trch|\psi_{(2)}|^2\\
    -&2kr^p\om|\psi_{(2)}|^2-2r^{p}\omb|\psi_{(1)}|^2.
\end{align*}
Then, we obtain
\begin{align*}
    &\bdiv(r^p|\psi_{(1)}|^2e_3)+k\bdiv(r^p|\psi_{(2)}|^2e_4)+\left(2a_{(1)}-1-\frac{p}{2}\right)r^{p}\trchb|\psi_{(1)}|^2+\left(2a_{(2)}-1-\frac{p}{2}\right)k r^{p}\trch|\psi_{(2)}|^2\\
       &=2k r^p\sdiv(\psi_{(1)}\cdot\psi_{(2)})
       +2r^p\psi_{(1)}\cdot h_{(1)}+2kr^p\psi_{(2)}\cdot h_{(2)}-2r^p\omb|\psi_{(1)}|^2-2kr^p\om|\psi_{(2)}|^2\\
       &+pr^{p-1}\left(e_3(r)-\frac{r}{2}\trchb\right)|\psi_{(1)}|^2+k pr^{p-1}\left(e_4(r)-\frac{r}{2}\trch\right)|\psi_{(2)}|^2,
\end{align*}
which implies \eqref{div}. The proof of \eqref{div2} is similar and left to the reader. This concludes the proof of Lemma \ref{keypoint}.
\end{proof}
\begin{prop}\label{keyintegral}
We denote
\begin{equation}
   V:=V(u,\unu),\qquad C_u^V:=C_u\cap V,\qquad\unc_\unu^V:=\unc_\unu\cap V.
\end{equation}
For $\psi_{(1)}$, $\psi_{(2)}$ and $h_{(1)}$, $h_{(2)}$ satisfying \eqref{bianchi1} or \eqref{bianchi2}, we have the following properties for $q=0,1$ and all $(u,\ub)\in\kk$.
\begin{itemize}
\item In the case of $2+p-4a_{(1)}>0$ and $4a_{(2)}-2-p>0$, we have
\begin{align}
\begin{split}\label{caseone}
&\int_{\cuv}r^p |\psi_{(1)}^{(q)}|^2+\int_\ucuv r^p|\psi_{(2)}^{(q)}|^2 +\int_{V}r^{p-1}|\psi_{(1)}^{(q)}|^2+r^{p-1}|\psi_{(2)}^{(q)}|^2 \\ 
\les &\int_{\Sigma_0 \cap V} r^p |\psi_{(1)}^{(q)}|^2 + r^p |\psi_{(2)}^{(q)}|^2+\int_{V} r^p|\psi_{(1)}^{(q)}||h_{(1)}^{(q)}|+r^p|\psi_{(2)}^{(q)}||h_{(2)}^{(q)}|.
\end{split}
\end{align}
\item In the case of $2+p-4a_{(1)}>0$ and $4a_{(2)}-2-p=0$, we have
\begin{align}
\begin{split}\label{casetwo}
&\int_{\cuv}r^p |\psi_{(1)}^{(q)}|^2+\int_\ucuv r^p|\psi_{(2)}^{(q)}|^2 +\int_{V} r^{p-1}|\psi_{(1)}^{(q)}|^2\\
\les &\int_{\Sigma_0 \cap V} r^p |\psi_{(1)}^{(q)}|^2+r^p |\psi_{(2)}^{(q)}|^2+\int_{V} r^p|\psi_{(1)}^{(q)}||h_{(1)}^{(q)}|+r^p|\psi_{(2)}^{(q)}||h_{(2)}^{(q)}|.
\end{split}
\end{align}
\item In the case of $2+p-4a_{(1)}\leq 0$ and $4a_{(2)}-2-p\geq 0$, we have
\begin{align}
\begin{split}\label{casethree}
&\int_{\cuv}r^p|\psi_{(1)}^{(q)}|^2+\int_\ucuv r^p|\psi_{(2)}^{(q)}|^2\\ 
\les&\int_{\Sigma_0\cap V}r^p|\psi_{(1)}^{(q)}|^2+r^p|\psi_{(2)}^{(q)}|^2+\int_{V}r^{p-1}|\psi_{(1)}^{(q)}|^2+r^p|\psi_{(1)}^{(q)}||h_{(1)}^{(q)}|+r^p|\psi_{(2)}^{(q)}||h_{(2)}^{(q)}|.
\end{split}
\end{align}
\item In the case of $2+p-4a_{(1)}> 0$ and $4a_{(2)}-2-p\leq 0$, we have
\begin{align}
\begin{split}\label{casefour}
&\int_{\cuv}r^p|\psi_{(1)}^{(q)}|^2+\int_\ucuv r^p|\psi_{(2)}^{(q)}|^2+\int_{V}r^{p-1}|\psi_{(1)}^{(q)}|^2\\
\les&\int_{\Sigma_0\cap V}r^p|\psi_{(1)}^{(q)}|^2+r^p|\psi_{(2)}^{(q)}|^2+\int_{V}r^{p-1}|\psi_{(2)}^{(q)}|^2+r^p|\psi_{(1)}^{(q)}||h_{(1)}^{(q)}|+r^p|\psi_{(2)}^{(q)}||h_{(2)}^{(q)}|.
\end{split}
\end{align}
\end{itemize}
\end{prop}
\begin{rk}
In the sequel, For a sum of terms of the form $\Ga\cdot R^{(1)}$, we ignore the terms have same or even better decay. For example, we write
\begin{align*}
    \Gaa\cdot \b^{(1)}+\Gaa\cdot\a^{(1)} =\Gaa\cdot\b^{(1)},
\end{align*}
since $\a^{(1)}$ decays better than $\b^{(1)}$.
\end{rk}
\begin{proof}
Recall that $\om,\,\omb\in\Gaa$. Applying Lemma \ref{dint}, we obtain 
\begin{align*}
    e_3(r)-\frac{r}{2}\trchb=\frac{\ov{\Om\trchb}}{2\Om}r-\frac{r}{2}\trchb=\frac{r}{2\Om}(\ov{\Om\trchb}-\Om\trchb)\in r\Gag,
\end{align*}
and similarly
\begin{align*}
    e_4(r)-\frac{r}{2}\trch\in r\Gag.
\end{align*}
Applying \eqref{estGagba}, we infer
\begin{align}
\begin{split}\label{Gaapsi}
&\int_V r^p|\Gaa|\left(|\psi_{(1)}|^2+|\psi_{(2)}|^2\right)\\
&\les\ep\int_V r^{p-2}\left(|\psi_{(1)}|^2+|\psi_{(2)}|^2\right)\\
&\les\ep \int_{u_0(\ub)}^u\left(|u|^{-2}\int_\cuv r^{p}|\psi_{(1)}|^2\right) du+\ep\int_{|u_0|}^\ub\left(|\ub|^{-2}\int_\ucuv r^{p}|\psi_{(2)}|^2\right)d\ub\\
&\les\ep\sup_u\left(\int_\cuv r^{p}|\psi_{(1)}|^2\right)\int_{u_0(\ub)}^u|u|^{-2}du+\ep\sup_{\ub}\left(\int_\ucuv r^{p}|\psi_{(2)}|^2\right)\int_{|u_0|}^{\ub}|\ub|^{-2}d\ub\\
&\les\ep\sup_u\left(\int_\cuv r^{p}|\psi_{(1)}|^2\right)+\ep\sup_{\ub}\left(\int_\ucuv r^{p}|\psi_{(2)}|^2\right).
\end{split}
\end{align}
Integrating \eqref{div} or \eqref{div2}, reminding that $\trch-\frac{2}{r}\in\Gaa$, $\trchb+\frac{2}{r}\in\Gaa$, and that $\Gag$ decays better than $\Gaa$, we obtain
\begin{align*}
&\int_{\cuv}r^p |\psi_{(1)}|^2+\int_\ucuv r^p|\psi_{(2)}|^2 +\int_{V} (2+p-4a_{(1)})r^{p-1}|\psi_{(1)}|^2+(4a_{(2)}-2-p)r^{p-1}|\psi_{(2)}|^2 \\ 
\les &\int_{\Sigma_0 \cap V} r^p |\psi_{(1)}|^2+r^p|\psi_{(2)}|^2+\int_{V} r^p|\psi_{(1)}||h_{(1)}|+r^p|\psi_{(2)}||h_{(2)}|+r^p|\Gaa||\psi_{(1)}|^2+r^p|\Gaa||\psi_{(2)}|^2. 
\end{align*}
Taking the supremum of $u$ and $\ub$ and applying \eqref{Gaapsi}, we obtain for $\ep$ small enough
\begin{align}
\begin{split}\label{psipsipsi}
&\sup_u\int_{\cuv}r^p |\psi_{(1)}|^2+\sup_\ub\int_\ucuv r^p|\psi_{(2)}|^2\\ +&\int_{V}(2+p-4a_{(1)})r^{p-1}|\psi_{(1)}|^2+(4a_{(2)}-2-p)r^{p-1}|\psi_{(2)}|^2 \\ 
\les &\int_{\Sigma_0 \cap V} r^p |\psi_{(1)}|^2+r^p|\psi_{(2)}|^2+\int_{V}r^p|\psi_{(1)}||h_{(1)}|+r^p|\psi_{(2)}||h_{(2)}|,
\end{split}
\end{align}
which implies \eqref{caseone}--\eqref{casefour} hold in correspond cases when $q=0$. \\ \\
Next, we consider the case $q=1$. Assume that $(\psi_{(1)},\psi_{(2)})$ satisfies \eqref{bianchi1}.\footnote{The case of $(\psi_{(1)},\psi_{(2)})$ satisfies \eqref{bianchi2} is similar.} We multiply \eqref{bianchi1} by $\Om$ and differentiate it by $\dkb$ to obtain
\begin{align*}
&[\dkb,\Om\nab_3]\psi_{(1)}+\Om\nab_3(\dkb\psi_{(1)})+a_{(1)}r\nab(\Om\trchb)\cdot\psi_{(1)}+a_{(1)}\Om\trchb\,\dkb\psi_{(1)}\\
=&-k (r\nab\Om) d_k^*(\psi_{(2)})-k\Om d_k^*(\dkb\psi_{(2)})+(r\nab\Om) h_{(1)}+\Om \dkb h_{(1)},\\
&[\dkb,\Om\nab_4]\psi_{(2)}+\Om\nab_4(\dkb\psi_{(2)})+a_{(2)}r\nab(\Om\trch)\cdot\psi_{(2)}+a_{(2)}\Om\trch\,\dkb\psi_{(2)}\\
=&(r\nab\Om)d_k(\psi_{(1)})+\Om d_k(\dkb\psi_{(1)})+(r\nab\Om)\cdot h_{(2)}+\Om\dkb h_{(2)}.
\end{align*}
Applying Proposition \ref{commutation}, we infer
\begin{align}
\begin{split}\label{psi(1)psi(2)}
\nab_3(\dkb\psi_{(1)})+a_{(1)}\trchb\,\dkb\psi_{(1)}&=-kd_k^*(\dkb\psi_{(2)})+h_{(1)}^{(1)}+(\Gab\cdot\psi_{(1)})^{(1)}+(\Gag\cdot\psi_{(2)})^{(1)},\\
\nab_4(\dkb\psi_{(2)})+a_{(2)}\trch\,\dkb\psi_{(2)}&=d_k^*(\dkb\psi_{(1)})+h_{(2)}^{(1)}+\left(\Gag\cdot(\psi_{(1)},\psi_{(2)})\right)^{(1)}.
\end{split}
\end{align}
Integrating \eqref{psi(1)psi(2)}, and proceeding similarly as in \eqref{Gaapsi}, we deduce
\begin{align}
\begin{split}\label{psipsipsipsi}
&\sup_u\int_{\cuv}r^p |\dkb\psi_{(1)}|^2+\sup_\ub\int_\ucuv r^p|\dkb\psi_{(2)}|^2 \\
+&\int_{V}(2+p-4a_{(1)})r^{p-1}|\dkb\psi_{(1)}|^2+(4a_{(2)}-2-p)r^{p-1}|\dkb\psi_{(2)}|^2\\
\les &\int_{\Sigma_0 \cap V} r^p |\psi_{(1)}^{(1)}|^2+r^p|\psi_{(2)}^{(1)}|^2+\int_{V} r^p|\psi_{(1)}^{(1)}||h_{(1)}^{(1)}|+r^p|\psi_{(2)}^{(1)}||h_{(2)}^{(1)}|.
\end{split}
\end{align}
Combining \eqref{psipsipsi} and \eqref{psipsipsipsi}, this concludes the proof of Proposition \ref{keyintegral}.
\end{proof}
The following lemma allows us to obtain $|u|$--decay of curvature along $\Si_0\cap V(u,\ub_*)$. 
\begin{lem}\label{gainu}
We have the following estimate for $p\leq s$:
\begin{align*}
    \int_{\Si_0\cap V(u,\ub)}r^p\left(|\a^{(1)}|^2+|\b^{(1)}|^2+|(\rhoc^{(1)},\si^{(1)})|^2+|\bb^{(1)}|^2+|\aa^{(1)}|^2\right)\les \frac{\Rk}{|u|^{s-p}}.
\end{align*}
\end{lem}
\begin{proof}
For fixed $u$ and $\ub$, we have from Lemma \ref{equivalence}
\begin{align*}
    &\int_{\Si_0\cap V(u,\ub)}r^p\left(|\a^{(1)}|^2+|\b^{(1)}|^2+|(\rhoc^{(1)},\si^{(1)})|^2+|\bb^{(1)}|^2+|\aa^{(1)}|^2\right)\\
    \les&\int_{\Si_0\cap\{2|u'|\geq|u|\}}{r'}^p\left(|\a^{(1)}|^2+|\b^{(1)}|^2+|(\rhoc^{(1)},\si^{(1)})|^2+|\bb^{(1)}|^2+|\aa^{(1)}|^2\right)\\
    \les&\int_{\Si_0\cap\{2|u'|\geq|u|\}}|u'|^{p-s}{r'}^s\left(|\a^{(1)}|^2+|\b^{(1)}|^2+|(\rhoc^{(1)},\si^{(1)})|^2+|\bb^{(1)}|^2+|\aa^{(1)}|^2\right)\\
    \les&|u|^{p-s}\int_{\Si_0\cap\{2|u'|\geq|u|\}}{r'}^s\left(|\a^{(1)}|^2+|\b^{(1)}|^2+|(\rhoc^{(1)},\si^{(1)})|^2+|\bb^{(1)}|^2+|\aa^{(1)}|^2\right)\\
    \les&\frac{\Rk}{|u|^{s-p}}.
\end{align*}
This concludes the proof of Lemma \ref{gainu}.
\end{proof}
\subsection{Estimates for the Bianchi pair \texorpdfstring{$(\alpha,\beta)$}{}}\label{ssec9.1}
\begin{prop}\label{estab}
We have the following estimate:
\begin{equation}
\mr_0[\alpha]^2+\ur_0[\beta]^2+\mr_1[\alpha]^2+\ur_1[\beta]^2\les\Rk+\ep\mr^2.\label{abs}
\end{equation}
\end{prop}
\begin{proof}
We recall the following equations of Corollary \ref{prop7.6}:
\begin{align}
\begin{split}\label{Bianchiequationab}
    \nabla_4\beta+2\tr\chi\beta&=\sld_2\alpha+h[\beta_4],\\
\nabla_3\alpha+\frac{1}{2}\tr\unchi\alpha&=-2\sld_2^*\beta+h[\alpha_3],    
\end{split}
\end{align}
which correspond to $\psi_{(1)}=\a$, $\psi_{(2)}=\b$, $a_{(1)}=\frac{1}{2}$, $a_{(2)}=2$, $h_{(1)}=h[\a_3]$, $h_{(2)}=h[\b_4]$ and $k=2$ in \eqref{bianchi1}. Taking $p=s$ and recalling that $s\in[4,6]$, we have
\begin{equation}\label{pless6}
    2+p-4a_{(1)}=s>0,\qquad 4a_{(2)}-2-p=6-s\geq 0.
\end{equation}
We apply \eqref{caseone} to obtain
\begin{align*}
    &\int_{\cuv} r^s |\alpha^{(1)}|^2 + \int_\ucuv r^s |\beta^{(1)}|^2 + \int_{V}r^{s-1} |\alpha^{(1)}|^2 \\
    \les&\int_{\Si_0 \cap V} r^s |\a^{(1)}|^2 + r^s |\b^{(1)}|^2+\int_{V} r^s|\a^{(1)}||h[\a_3]^{(1)}|+r^s|\b^{(1)}||h[\beta_4]^{(1)}|.
\end{align*}
It remains to estimate $|\a^{(1)}||h[\alpha_3]^{(1)}|$ and $|\b^{(1)}||h[\beta_4]^{(1)}|$. Recalling from Corollary \ref{prop7.6} that\footnote{We used the fact that $\si$, $\b$ decay better than $\rho$.}
\begin{align}\label{ha3hb4}
    h[\a_3]=4\omb\a+\Gag\cdot\rho,\qquad h[\b_4]=(\Gaa,\Gag)\cdot\b,
\end{align}
we have\footnote{We ignore the terms which decay better.}
\begin{align*}
|\a^{(1)}||h[\a_3]^{(1)}|+|\b^{(1)}||h[\b_4]^{(1)}|\les|\Gab^{(1)}||\a^{(1)}||\a|+|\Gaa||\b^{(1)}|^2+|\Gag||\a^{(1)}||\rho^{(1)}|+|\Gag^{(1)}||\a||\rho|.
\end{align*}
Applying \eqref{estGagba} to estimate $\Gab^{(1)}$, we have
\begin{align}
\begin{split}\label{Gabaa}
\int_V r^s |\Gab^{(1)}||\a^{(1)}||\a|&\les\int_{u_0(\ub)}^{u}\left(\int_\cuv r^s|\a^{(1)}|^2\right)^{\frac{1}{2}}\left(\int_\cuv r^s|\Gab^{(1)}|^2|\a|^2\right)^{\frac{1}{2}} du\\
&\les\mr\int_{u_0(\ub)}^{u}\left(\int_{|u|}^\ub\left(\int_{S(u,\ub)}r^s|\Gab^{(1)}|^2|\a|^2\right)d\ub\right)^{\frac{1}{2}}du\\
&\les\mr\int_{u_0(\ub)}^{u}\left(\int_{|u|}^\ub r^{-3}|u|^{1-s}|r^{\frac{1}{2}}|u|^{\frac{s-1}{2}}\Gab^{(1)}|^2_{4,S}|r^{\frac{s+2}{2}}\a|^2_{4,S}d\ub\right)^{\frac{1}{2}}du\\
&\les\ep\mr^2\int_{u_0(\ub)}^{u}\left(\int_{|u|}^\ub r^{-3}|u|^{1-s} d\ub\right)^{\frac{1}{2}}du\\
&\les\ep\mr^2\int_{u_0(\ub)}^{u}|u|^{-1}|u|^{\frac{1-s}{2}} du\les\ep\mr^2.    
\end{split}
\end{align}
Applying \eqref{estGagba} to estimate $\Gaa$, we infer
\begin{align}
\begin{split}\label{Gagab2}
\int_V r^s|\Gaa||\b^{(1)}|^2&\les\ep\int_V r^{s-2}|\b^{(1)}|^2\les\ep\int_{|u_0|}^\ub |\ub|^{-2}d\ub\left(\int_\ucuv r^s|\b^{(1)}|^2\right)\les \ep\mr^2.
\end{split}
\end{align}
Applying \eqref{estGagba} to estimate $\Gag$, we infer
\begin{align}
\begin{split}\label{Gagrhoa}
\int_V r^s|\Gag||\rho^{(1)}||\a^{(1)}|&\les\int_{u_0(\ub)}^{u}\frac{\ep}{|u|^{\frac{s-3}{2}}}\left(\int_\cuv r^s|\a^{(1)}|^2\right)^{\frac{1}{2}}\left(\int_\cuv r^{s-4}|\rho^{(1)}|^2\right)^{\frac{1}{2}} du\\
&\les\int_{u_0(\ub)}^{u}\frac{\ep\mr}{|u|^{\frac{s-3}{2}}}\left(\int_{|u|}^\ub r^{s-8}  d\ub\int_S |r^2\rho^{(1)}|^2\right)^{\frac{1}{2}} du\\
&\les\int_{u_0(\ub)}^{u}\frac{\ep\mr^2}{|u|^{\frac{s-3}{2}}}\left(\int_{|u|}^\ub r^{s-8}  d\ub\right)^{\frac{1}{2}} du\\
&\les\int_{u_0(\ub)}^{u}\frac{\ep\mr^2}{|u|^{\frac{s-3}{2}}}\frac{1}{|u|^{\frac{7-s}{2}}} du\les\ep\mr^2.
\end{split}
\end{align}
We also have
\begin{align}
\begin{split}\label{Gagrhoa'}
\int_V r^s|\Gag^{(1)}||\rho||\a|&\les\int_{u_0(\ub)}^{u}\left(\int_\cuv r^s|\a|^2\right)^{\frac{1}{2}}\left(\int_\cuv r^s|\Gag^{(1)}|^2|\rho|^2\right)^{\frac{1}{2}} du\\
&\les\mr\int_{u_0(\ub)}^{u}\left(\int_{|u|}^\ub\left(\int_{S(u,\ub)}r^s|\Gag^{(1)}|^2|\rho|^2 \right)d\ub\right)^{\frac{1}{2}}du\\
&\les\mr\int_{u_0(\ub)}^{u}\left(\int_{|u|}^\ub r^{s-8}|u|^{3-s}|r^{\frac{3}{2}}|u|^{\frac{s-3}{2}}\Gag^{(1)}|^2_{4,S}|r^{\frac{5}{2}}\rho|^2_{4,S}d\ub\right)^{\frac{1}{2}}du\\
&\les\ep\mr^2\int_{u_0(\ub)}^{u}\left(\int_{|u|}^\ub r^{s-8}|u|^{3-s} d\ub\right)^{\frac{1}{2}}du\\
&\les\ep\mr^2\int_{u_0(\ub)}^{u}|u|^{\frac{s-7}{2}}|u|^{\frac{3-s}{2}} du\les\ep\mr^2.    
\end{split}
\end{align}
Thus, we obtain
\begin{equation}\label{bulkterm}
\int_{\cuv} r^s |\a^{(1)}|^2 + \int_\ucuv r^s |\b^{(1)}|^2 + \int_{V}r^{s-1} |\a^{(1)}|^2\les \Rk+\ep\mr^2.
\end{equation}
This concludes the proof of Proposition \ref{estab}.
\end{proof}
\subsection{Estimates for the Bianchi pair \texorpdfstring{$(\b,(\rhoc,-\si))$}{}}\label{ssec9.2}
\begin{prop}\label{estbr}
We have the following estimate:
\begin{equation}\label{eqbr}
\mr_0[\b]^2+\ur_0[(\rhoc,\si)]^2+\mr_1[\b]^2+\ur_1[(\chr,\si)]^2\les\Rk+\ep\mr^2.
\end{equation}
\end{prop}
\begin{proof}
We recall the following Bianchi equations
\begin{align*}
\nabla_4\widecheck\rho+\frac{3}{2}\tr\chi\widecheck\rho&=\sdiv\beta+h[\chr_4],\\
\nabla_4\sigma+\frac{3}{2}\tr\chi\sigma&=-\sdiv{^*\beta}+h[\sigma_4],\\
\nabla_3\beta+\tr\unchi\beta&=\nabla\widecheck\rho+{^*\nabla}\sigma+h[\beta_3],
\end{align*}
which can be written in the form
\begin{align}
\begin{split}
\nabla_4(\chr,-\sigma)+\frac{3}{2}\tr\chi(\chr,-\sigma)&=\sld_1\beta+h[\chr_4,\sigma_4],\\
\nabla_3\beta+\tr\unchi\beta&=-\sld_1^*(\chr,-\sigma)+h[\beta_3],
\end{split}
\end{align}
where $h[\chr_4,\sigma_4]=(h[\chr_4],h[\sigma_4])$. Applying \eqref{casetwo} with $\psi_{(1)}=\b$, $\psi_{(2)}=(\rhoc,\si)$, $a_{(1)}=1$, $a_{(2)}=\frac{3}{2}$, $h_{(1)}=h[\b_3]$, $h_{(2)}=h[\rhoc_4,\si_4]$ and $p=4$, we obtain from Lemma \ref{gainu}
\begin{align}
\begin{split}\label{brint}
&\int_{\cuv} r^4|\b^{(1)}|^2+\int_{\ucuv}r^{4}|(\rhoc^{(1)},\si^{(1)})|^2 \\
\les &\int_{\Sigma_0\cap V}\left(r^4|\beta^{(1)}|^2+r^{4}|(\rhoc^{(1)},\si^{(1)})|^2\right)\\
&+\int_V r^{4}|\beta^{(1)}\cdot h[\beta_3]^{(1)}|+\int_V r^{4} |(\chr^{(1)},\si^{(1)})\cdot h[\rhoc_4,\si_4]^{(1)}|\\
\les &|u|^{4-s}\Rk+\int_V r^{4}|\b^{(1)}\cdot h[\b_3]^{(1)}|+\int_V r^{4} |(\chr^{(1)},\si^{(1)})\cdot h[\rhoc_4,\si_4]^{(1)}|.    
\end{split}
\end{align}
Recall from Corollary \ref{prop7.6} that
\begin{equation*}
    h[\b_3]=\Gaa\cdot\b+\Gag\cdot\bb,\qquad h[\rhoc_4,\si_4]=\Gab\cdot\a+\Gag\cdot\rhoc.
\end{equation*}
Hence, we have
\begin{equation*}
    h[\b_3]^{(1)}=\Gaa\cdot\b^{(1)}+\Gag\cdot\bb^{(1)}+\Gag^{(1)}\cdot\bb,\qquad h[\rhoc_4,\si_4]^{(1)}=\Gab\cdot\a^{(1)}+\Gag\cdot\rhoc^{(1)}+\Gab^{(1)}\cdot\a+\Gag^{(1)}\cdot\rhoc.
\end{equation*}
Hence, we infer
\begin{align}\label{brhocsi}
\begin{split}
&\int_V r^{4}|\b^{(1)}||h[\b_3]^{(1)}|+r^4|(\rhoc,\si)^{(1)}||h[\rhoc_4,\si_4]^{(1)}|\\
\les&\int_V r^4|\Gaa||\b^{(1)}|^2+r^4|\Gag||\b^{(1)}||\bb^{(1)}|+r^4|\Gag^{(1)}||\b||\bb|.
\end{split}
\end{align}
For the first term in \eqref{brhocsi} we have
\begin{align*}
    \int_V r^4|\Gaa||\b^{(1)}|^2&\les\ep\int_{u_0(\ub)}^{u} |u|^{-2}du\left(\int_\cuv r^4|\b^{(1)}|^2\right)\les \ep \mr^2 |u|^{4-s}.
\end{align*}
For the second term in \eqref{brhocsi}, we infer
\begin{align*}
    \int_V r^4|\Gag||\b^{(1)}||\bb^{(1)}|&\les\ep\int_{u_0(\ub)}^u \frac{1}{|u|^{\frac{s-3}{2}}}du\left(\int_\cuv r^4|\b^{(1)}|^2\right)^{\frac{1}{2}}\left(\int_\cuv |\bb^{(1)}|^2\right)^{\frac{1}{2}}\\
    &\les\ep\int_{u_0(\ub)}^u\frac{du}{|u|^{\frac{s-3}{2}}}\frac{\mr}{|u|^{\frac{s-4}{2}}}\frac{\mr}{|u|^{\frac{s}{2}}}\les\frac{\ep\mr^2}{|u|^{s-4}}.
\end{align*}
For the third term in \eqref{brhocsi}, we have
\begin{align*}
\int_V r^4|\Gag^{(1)}||\b||\bb|&\les\int_V r^4\frac{\ep}{|u|^\frac{s-1}{2}}|\b|^2 +\int_V r^4 |u|^\frac{s-1}{2}\ep^{-1}|\Gag^{(1)}|^2|\bb|^2\\
&\les\ep\int_{u_0(\ub)}^u \frac{du}{|u|^{\frac{s-1}{2}}}\int_\cuv r^4|\b|^2 +\int_{|u|}^\ub d\ub \int_\ucuv r^4|u|^\frac{s-1}{2}\ep^{-1}|\Gag^{(1)}|^2|\bb|^2\\
&\les\ep\int_{u_0(\ub)}^u \frac{du}{|u|^{\frac{s-1}{2}}}\frac{\mr^2}{|u|^{s-4}} +\int_{|u|}^\ub d\ub \int_{u_0(\ub)}^u  r^{-2}|u|^\frac{s-1}{2}\ep^{-1}|r^\frac{3}{2}\Gag^{(1)}|^2_{4,S}|r^\frac{3}{2}\bb|^2_{4,S}du\\
&\les \frac{\ep\mr^2}{|u|^\frac{s-3}{2}|u|^{s-4}}+\int_{|u|}^\ub d\ub \int_{u_0(\ub)}^u  r^{-2}|u|^\frac{s-1}{2}\frac{\ep\mr^2}{|u|^{2s-4}}du\\
&\les \frac{\ep\mr^2}{|u|^{s-4}}+\frac{\ep\mr^2}{|u|^{\frac{3}{2}s-\frac{5}{2}}}\int_{|u|}^\ub \frac{d\ub}{r^2} \\
&\les \frac{\ep\mr^2}{|u|^{s-4}}.
\end{align*}
Injecting all the estimates above into \eqref{brhocsi}, we deduce
\begin{equation*}
\int_V r^{4}|\b^{(1)}||h[\b_3]^{(1)}|+r^4|(\rhoc^{(1)},\si^{(1)})||h[\rhoc_4,\si_4]^{(1)}|\les\ep\mr^2|u|^{4-s}.
\end{equation*}
Combining with \eqref{brint}, we obtain \eqref{eqbr}. This concludes the proof of Proposition \ref{estbr}.
\end{proof}
\subsection{Estimates for the Bianchi pair \texorpdfstring{$((\chr,\sigma),\unb)$}{}}\label{ssec9.3}
\begin{prop}\label{estrb}
We have the following estimates:
\begin{equation}
\mr_0[(\chr,\sigma)]^2+\ur_0[\unb]^2+\mr_1[(\chr,\sigma)]^2+\ur_1[\unb]^2\lesssim \Rk+\epsilon\mr^{2}.\label{rb00}
\end{equation}
\end{prop}
\begin{proof}
We recall the following Bianchi equations:
\begin{align*}
\nabla_4\unb+\tr\chi\,\bb&=\sld_1^*(\rhoc,\si)+h[\bb_4],\\
\nab_3(\rhoc,\sigma)+\frac{3}{2}\trchb(\rhoc,\sigma)&=-\sld_1\unb+h[\chr_3,\sigma_3],
\end{align*}
where
\begin{align}
\begin{split}\label{rhosibb}
    h[\rhoc_3,\si_3]:=(h[\rhoc_3],h[\si_3])=\Gag\cdot\aa,\qquad h[\bb_4]=\Gaa\cdot\bb. 
\end{split}
\end{align}
We apply \eqref{casethree} with $\psi_{(1)}=(\rhoc,\si)$, $\psi_{(2)}=\bb$, $a_{(1)}=1$, $a_{(2)}=\frac{3}{2}$, $h_{(1)}=h[\rhoc_3,\si_3]$, $h_{(2)}=h[\bb_4]$ and $p=2$, we obtain from Lemma \ref{gainu} and Proposition \ref{estbr} that
\begin{align}
\begin{split}\label{rbint}
&\int_{\cuv} r^2|(\rhoc^{(1)},\si^{(1)})|^2+\int_{\ucuv}r^2|\bb^{(1)}|^2 \\
\les &\int_{\Si_0\cap V}\left(r^2(\rhoc^{(1)},\si^{(1)})|^2+r^2|\bb^{(1)}|^2\right)+\int_V r|(\rhoc^{(1)},\si^{(1)})|^2\\
&+\int_V r^{2} |(\chr^{(1)},\si^{(1)})\cdot h[\rhoc_3,\si_3]^{(1)}|+r^2|\bb^{(1)}\cdot h[\bb_4]^{(1)}|\\
\les &|u|^{2-s}\Rk+\int_{|u|}^{\ub}|\ub|^{-3}d\ub\int_\ucuv  r^4|(\rhoc^{(1)},\si^{(1)})|^2\\
&+\int_V r^{2} |(\chr^{(1)},\si^{(1)})\cdot h[\rhoc_3,\si_3]^{(1)}|+r^2|\bb^{(1)}\cdot h[\bb_4]^{(1)}|\\
\les &|u|^{2-s}\Rk+(\Rk+\ep\mr^2)\int_{|u|}^{\ub}|u|^{4-s}|\ub|^{-3}d\ub\\
&+\int_V r^{2} |(\chr^{(1)},\si^{(1)})\cdot h[\rhoc_3,\si_3]^{(1)}|+r^2|\bb^{(1)}\cdot h[\bb_4]^{(1)}|\\
\les &|u|^{2-s}(\Rk+\ep\mr^2)+\int_V r^{2} |(\chr^{(1)},\si^{(1)})\cdot h[\rhoc_3,\si_3]^{(1)}|+r^2|\bb^{(1)}\cdot h[\bb_4]^{(1)}|.    
\end{split}
\end{align}
Applying \eqref{rhosibb} we obtain
\begin{align}\label{rhosibbh}
\begin{split}
    &\int_V r^{2} |(\chr^{(1)},\si^{(1)})\cdot h[\rhoc_3,\si_3]^{(1)}|+r^2|\bb^{(1)}\cdot h[\bb_4]^{(1)}|\\
    \les&\int_V r^2|\Gag||\aa^{(1)}||\rhoc^{(1)}|+r^2|\Gaa||\bb^{(1)}|^2+|\Gag^{(1)}||\bb|^2.
\end{split}
\end{align}
For the first term in \eqref{rhosibbh}, we infer
\begin{align*}
    \int_V r^2 |\Gag||\aa^{(1)}||\rhoc^{(1)}|&\les\ep\int_{|u|}^\ub d\ub \int_\ucuv |\aa^{(1)}||\rhoc^{(1)}|\\
    &\les\ep\int_{|u|}^\ub \frac{d\ub}{|\ub|^2} \left(\int_\ucuv |\aa^{(1)}|^2\right)^\frac{1}{2}\left(\int_\ucuv r^4|\rhoc^{(1)}|^2\right)^\frac{1}{2}\\
    &\les\ep\int_{|u|}^\ub \frac{d\ub}{|\ub|^2}\frac{\mr}{|u|^{\frac{s}{2}}}\frac{\mr}{|u|^{\frac{s-4}{2}}}\les \ep\mr^2|u|^{2-s}.
\end{align*}
For the second term in \eqref{rhosibbh}, we have
\begin{align*}
\int_V r^2|\Gaa||\bb^{(1)}|^2\les\ep\int_{|u|}^\ub |\ub|^{-2}d\ub\int_\ucuv\ r^2|\bb^{(1)}|^2\les\ep\mr^2 |u|^{2-s}\int_{|u|}^\ub |\ub|^{-2}d\ub\les\ep\mr^2 |u|^{2-s}.
\end{align*}
For the third term in \eqref{rhosibbh}, we infer
\begin{align*}
    \int_V r^2|r\nab\Gag||\bb|^2&\les\int_{|u|}^\ub d\ub\left(\int_\ucuv r^2|\bb|^2 \right)^\frac{1}{2}\left(\int_\ucuv r^2|\Gag^{(1)}|^2|\bb|^2 \right)^\frac{1}{2}\\
    &\les\int_{|u|}^\ub d\ub\frac{\mr}{|u|^{\frac{s-2}{2}}}\left(\int_{u_0(\ub)}^u \frac{du}{|\ub|^4}|r^{\frac{3}{2}}\Gag^{(1)}|^2_{4,S}|r^{\frac{3}{2}}\bb|^2_{4,S} \right)^\frac{1}{2}\\
    &\les\mr\int_{|u|}^\ub d\ub\frac{1}{|u|^{\frac{s-2}{2}}}\left(\int_{u_0(\ub)}^u \frac{du}{|\ub|^4}\frac{\ep^2}{|u|^{s-3}}\frac{\mr^2}{|u|^{s-1}}\right)^\frac{1}{2}\\
    &\les\ep\mr^2\int_{|u|}^\ub d\ub\frac{1}{|u|^{\frac{s-2}{2}}}\frac{1}{|\ub|^2}\frac{1}{|u|^{s-\frac{5}{2}}}\les \ep\mr^2|u|^{2-s}.
\end{align*}
Thus, we obtain
\begin{align*}
\int_V r^{2} |(\chr^{(1)},\si^{(1)})\cdot h[\rhoc_3,\si_3]^{(1)}|+r^2|\bb^{(1)}\cdot h[\bb_4]^{(1)}|\les \ep\mr^2|u|^{2-s}.
\end{align*}
Combining with \eqref{rbint}, we obtain \eqref{rb00}. This concludes the proof of Proposition \ref{estrb}.
\end{proof}
\subsection{Estimates for the Bianchi pair \texorpdfstring{$(\unb,\una)$}{}}\label{ssec9.4}
\begin{prop}\label{estba}
We have the following estimates:
\begin{equation}\label{bbaa00}
    \mr_0[\unb]^2+\ur_0[\una]^2+\mr_1[\unb]^2+\ur_1[\una]^2\les\Rk+\ep\mr^2.
\end{equation}
\end{prop}
\begin{proof}
We recall the following Bianchi equations:
\begin{align}
\begin{split}
   \nabla_4\una+\frac{1}{2}\tr\chi\,\una&=-\nabla\widehat{\otimes}\unb+h[\una_4],\\
\nabla_3\unb+2\tr\unchi\,\unb&=-\sdiv\una+h[\unb_3],\label{babian}
\end{split}
\end{align}
where
\begin{align}\label{bbaa}
    h[\aa_4]=\Gaa\cdot\aa,\qquad h[\bb_3]=\Gaa\cdot\bb+\Gag\cdot\aa.
\end{align}
Applying \eqref{casethree} with $\psi_{(1)}=\bb$, $\psi_{(2)}=\aa$, $a_{(1)}=2$, $a_{(2)}=\frac{1}{2}$, $h_{(1)}=h[\bb_3]$, $h_{(2)}=h[\aa_4]$ and $p=0$, we obtain from Lemma \ref{gainu}
\begin{align}
\begin{split}
&\int_{\cuv}|\bb^{(1)}|^2+\int_{\ucuv} |\aa^{(1)}|^2 \\
\les&\int_{\Si_0\cap V}|\unb^{(1)}|^2+|\una^{(1)}|^2+\int_V r^{-1}|\unb^{(1)}|^2+\int_V |\unb^{(1)}||h[\unb_3]^{(1)}|+|\una^{(1)}||h[\una_4]^{(1)}|\\
\les& |u|^{-s}\Rk+\int_{|u_0|}^\ub r^{-3}d\ub\int_\ucuv r^2|\bb^{(1)}|^2+\int_V |\unb^{(1)}||h[\unb_3]^{(1)}|+|\una^{(1)}||h[\una_4]^{(1)}|\\
\les&|u|^{-s}\Rk+(\Rk+\ep\mr^2)\int_{|u_0|}^\ub r^{-3}d\ub|u|^{2-s}+\int_V |\unb^{(1)}||h[\unb_3]^{(1)}|+|\una^{(1)}||h[\una_4]^{(1)}|\\
\les&|u|^{-s}(\Rk+\ep\mr^2)+\int_V |\unb^{(1)}||h[\unb_3]^{(1)}|+|\una^{(1)}||h[\una_4]^{(1)}|.\label{bap}
\end{split}
\end{align}
Applying \eqref{bbaa}, we infer\footnote{Recall that the $\Gaa$ before $\aa$ only contains $\om$.}
\begin{align}\label{bbaah}
\int_V |\unb^{(1)}||h[\unb_3]^{(1)}|+|\una^{(1)}||h[\una_4]^{(1)}|\les\int_V |\Gaa||\aa^{(1)}|^2+|\Gag^{(1)}||\aa|^2.
\end{align}
For the first term in \eqref{bbaah}, we have
\begin{align*}
    \int_V|\Gaa||\aa^{(1)}|^2\les\ep\int_{|u|}^\ub |\ub|^{-2}d\ub \int_\ucuv|\aa|^2\les\ep\mr^2|u|^{-s}.
\end{align*}
For the second term in \eqref{bbaah}, we infer
\begin{align*}
    \int_V|\Gag^{(1)}||\aa|^2&\les\int_{|u|}^\ub d\ub\left(\int_\ucuv |\aa|^2\right)^\frac{1}{2}\left(\int_\ucuv|\Gag^{(1)}|^2|\aa|^2\right)^\frac{1}{2}\\
    &\les\mr\int_{|u|}^\ub d\ub |u|^{-\frac{s}{2}}\left(\int_{u_0(\ub)}^{u}\frac{du}{|\ub|^4}|r^{\frac{3}{2}}\Gag^{(1)}|_{4,S}^2
    |r^{\frac{1}{2}}\aa|_{4,S}^2\right)^\frac{1}{2}\\
    &\les\ep\mr^2\int_{|u|}^\ub d\ub |\ub|^{-2} |u|^{-\frac{s}{2}}\left(\int_{u_0(\ub)}^{u}\frac{du}{|u|^{2s-2}}\right)^\frac{1}{2}\\
    &\les\ep\mr^2\int_{|u|}^\ub d\ub |u|^{-\frac{s}{2}}|\ub|^{-2} |u|^{\frac{3-2s}{2}}\les \ep\mr^2|u|^{-s}.
\end{align*}
Thus, we obtain
\begin{equation*}
\int_V |\unb^{(1)}||h[\unb_3]^{(1)}|+|\una^{(1)}||h[\una_4]^{(1)}|\les\ep\mr^2|u|^{-s}.
\end{equation*}
Combining with \eqref{bap}, we obtain \eqref{bbaa00}. This concludes the proof of Proposition \ref{estba}.
\end{proof}
\subsection{Estimate for \texorpdfstring{$\nab_4\a$}{}}
\begin{prop}\label{esta4}
We have the following estimates:
\begin{equation}\label{a4}
    \mr_1[\a_4]^2\les\Rk+\ep\mr^2.
\end{equation}
\end{prop}
\begin{proof}
We recall from \eqref{teu}
\begin{align*}
    \nab_3\ac&=-2d_2^*\as+\frac{4\a}{r}+\Gaa\cdot\b^{(1)}+\Gag^{(1)}\cdot\b,\\
    \nab_4\as+\frac{5}{2}\trch \,\as&=d_2\ac+\Gaa\cdot\b^{(1)}+\Gag^{(1)}\cdot\b.
\end{align*}
Applying \eqref{caseone} with $\psi_{(1)}=\ac$, $\psi_{(2)}=\as$, $a_{(1)}=0$, $a_{(2)}=\frac{5}{2}$, $h_{(1)}=\frac{4\a}{r}+\Gaa\cdot\rg+\Gag^{(1)}\c\b$, $h_{(2)}=\Gaa\cdot\rg+\Gag^{(1)}\c\b$, $k=2$ and $p=s$, we obtain
\begin{align}
\begin{split}\label{esta4eq}
       &\int_{C_u}r^{s}|\ac|^2 +\int_{\Cb_\ub}r^{s}{|\as|^2}+\int_V r^{s-1}\big(|\ac|^2+|\as|^2\big)\\
    \les &\Rk+\int_V r^{s-1}|\ac||\a|+r^{s}|(\ac,\as)||\Gaa||\rg|+r^{s}|(\ac,\as)||\Gag^{(1)}||\b|.
\end{split}
\end{align}
First, we have for all $\de>0$
\begin{align*}
    \int_V r^{s-1}|\ac||\a|\leq \de \int_V r^{s-1}|\ac|^2 +\frac{1}{4\de}\int_V r^{s-1}|\a|^2.
\end{align*}
Combining with \eqref{bulkterm}, we deduce from \eqref{esta4eq}, for $\de$ small enough
\begin{align*}
    &\int_{C_u}r^{s}|\ac|^2 +\int_{\Cb_\ub}r^{s}{|\as|^2}+\int_V r^{s-1}\big(|\ac|^2+|\as|^2\big)\\
    \les& \Rk+\ep\mr^2+\int_V r^{s}|(\ac,\as)||\Gaa||\rg|+r^{s}|(\ac,\as)||\Gag^{(1)}||\b|.
\end{align*}
We have
\begin{align*}
    \int_V r^{s}|(\ac,\as)||\Gaa||\rg|&\les\ep\int_{V}r^{s-2}|\ac,\as||\rg|\\
    &\les \ep\int_{-\ub}^u du \frac{1}{|u|^\frac{8-s}{2}}\left(\int_{\cuv} r^{s}|\ac,\as|^2 \right)^\frac{1}{2}\left(\int_{\cuv} r^{4}|\rg|^2 \right)^\frac{1}{2}\\
    &\les\int_{-\ub}^u du \frac{\ep\mr^2}{|u|^{\frac{8-s}{2}}|u|^\frac{s-4}{2}}\\
    &\les\ep\mr^2.
\end{align*}
We also have
\begin{align*}
\int_V r^{s} |(\ac,\as)||\Gag^{(1)}||\b|&\les \int_{V}r^{s}|\ac||\Gag^{(1)}||\b|+\int_V r^{s}|\as||\Gag^{(1)}||\b|\\
        &\les\int_{-\ub}^u du \left(\int_{\cuv} r^{s}|\ac,\as|^2 \right)^\frac{1}{2}\left(\int_{\cuv} r^{s}|\Gag^{(1)}|^2|\b|^2 \right)^\frac{1}{2}\\
        &\les\mr\int_{-\ub}^u du\left(\int_{|u|}^\ub d\ub\, r^{s}\int_S |\Gag^{(1)}|^2|\b|^2 \right)^\frac{1}{2}\\
        &\les\mr\int_{-\ub}^u du\left(\int_{|u|}^\ub d\ub\, r^{s-10}|r^\frac{3}{2}\Gag^{(1)}|^2_{4,S}|r^\frac{7}{2}\b|_{4,S}^2 \right)^\frac{1}{2}\\
        &\les\ep\mr^2\int_{-\ub}^u du\left(\int_{|u|}^\ub r^{s-10}|u|^{8-2s} d\ub\right)^\frac{1}{2}\\
        &\les \ep\mr^2\int_{-\ub}^u du\, |u|^\frac{s-9}{2} |u|^{4-s}\\
        &\les \ep\mr^2.
    \end{align*}
Combining the above estimates, we obtain
    \begin{align*}
        \int_{\cuv}r^s|\ac|^2 +\int_{\ucuv}r^s|\as|^2+\int_V r^{s-1}\left(|\ac|^2 +|\as|^2\right)\les \Rk+\ep\mr^2.
    \end{align*}
    Recalling that
\begin{align*}
    \ac=r^{-4}\nab_4(r^5\a)=r\nab_4\a+5e_4(r)\a,
\end{align*}
we deduce from Proposition \ref{estab}
\begin{align*}
    \int_{\cuv} r^{s+2}|\nab_4\a|^2\les\Rk+\ep\mr^2.
\end{align*}
This concludes the proof of Proposition \ref{esta4}.
\end{proof}
\subsection{Estimate for \texorpdfstring{$\nab_3\aa$}{}}
\begin{prop}\label{estaa3}
We have the following estimates:
\begin{equation}\label{aa3}
    \mr_1[\aa_3]^2\les\mr(\Rk+\ep\mr^2)^\frac{1}{2}+\Rk+\ep\mr^2.
\end{equation}
\end{prop}
\begin{proof}
We recall from \eqref{teuaa}
\begin{align*}
\nab_4\aac&=-2d_2^*\aas+\frac{4\aa}{r}+(\Gaa,\Gab)\cdot\aa^{(1)}+\Gab^{(1)}\cdot\aa,\\
\nab_3\aas+\frac{5}{2}\trchb\,\aas&=d_2\aac+(\Gaa,\Gab)\cdot\aa^{(1)}+\Gab\c\aac+\Gab^{(1)}\cdot\aa.
\end{align*}
    Applying \eqref{casethree} with $\psi_{(1)}=\aas$, $\psi_{(2)}=\aac$, $a_{(1)}=\frac{5}{2}$, $a_{(2)}=0$, $h_{(1)}=\Gab^{(1)}\cdot\aa+\Gab\c\aac+(\Gaa,\Gab)\cdot\aa^{(1)}$, $h_{(2)}=\frac{4\aa}{r}+\Gab^{(1)}\cdot\aa+(\Gaa,\Gab)\cdot\aa^{(1)}$, $k=2$ and $p=-2$, we obtain
    \begin{align*}
        &\int_{\cuv}r^{-2}|\aas|^2+\int_{\ucuv}r^{-2}|\aac|^2\\
        \les&\int_{\Si_0\cap V} r^{-2}(|\aas|^2+|\aac|^2)+\int_V  r^{-3}|\aas|^2+r^{-3}|\aac||\aa^{(1)}|+r^{-2}|\Gab^{(1)}||\aa|(|\aac|+|\aas|).
    \end{align*}
    First, we have from Lemma \ref{gainu}
    \begin{align*}
        \int_{\Si_0\cap V} r^{-2}(|\aas|^2+|\aac|^2)\les\frac{\Rk}{|u|^{s+2}}.
    \end{align*}
    Notice that we have from Proposition \ref{estba}
    \begin{align*}
        \int_V r^{-3}|\aas|^2 \les \int_{|u|}^{\ub}r^{-3} d\ub\int_{\ucuv}|\aas|^2\les \frac{\Rk+\ep\mr^2}{|u|^{s+2}},
    \end{align*}
    and
    \begin{align*}
        \int_V r^{-3}|\aac||\aa^{(1)}|&\les\int_{|u|}^\ub r^{-2}\left(\int_{\ucuv}r^{-2}|\aac|^2\right)^\frac{1}{2}\left(\int_\ucuv|\aa^{(1)}|^2\right)^\frac{1}{2}d\ub\\ 
        &\les\int_{|u|}^\ub r^{-2}\left(\frac{\mr^2}{|u|^{s+2}}\right)^\frac{1}{2}\left(\frac{\Rk+\ep\mr^2}{|u|^s}\right)^\frac{1}{2}d\ub \\
        &\les\int_{|u|}^\ub\frac{\mr(\Rk+\ep\mr^2)^\frac{1}{2}}{r^2|u|^{s+1}}d\ub\\
        &\les\frac{\mr(\Rk+\ep\mr^2)^\frac{1}{2}}{|u|^{s+2}}.
    \end{align*}
    We also have
    \begin{align*}
        \int_V r^{-2}|\Gab^{(1)}||\aa|(|\aac|+|\aas|)&\les\int_{|u|}^\ub r^{-2} d\ub \int_\ucuv (|\aac|+|\aa^{(1)}|)|\aa||\Gab^{(1)}|\\
        &\les \int_{|u|}^\ub d\ub \left(\int_\ucuv |\aac|^2+|\aa^{(1)}|^2\right)^\frac{1}{2} \left(\int_{\ucuv}|\aa|^2|\Gab^{(1)}|^2\right)^\frac{1}{2}\\
        &\les\frac{\mr}{|u|^\frac{s}{2}}\int_{|u|}^\ub r^{-3}d\ub \left(\int_{|u|}^\ub du|r^\frac{1}{2}\aa|^2_{4,S}|r^\frac{1}{2}\Gab^{(1)}|^2_{4,S}\right)^\frac{1}{2}\\
        &\les \frac{\mr}{|u|^\frac{s}{2}}\int_{|u|}^\ub r^{-3}d\ub\left(\int_{|u|}^\ub du\frac{\ep^2\mr^2}{|u|^{2s}}\right)^\frac{1}{2}\\
        &\les \frac{\ep\mr^2}{|u|^\frac{3s+3}{2}}.
    \end{align*}
    Combining the above estimates, we deduce
    \begin{align*}
        \int_{\cuv}r^{-2}|\aas|^2+\int_{\ucuv}r^{-2}|\aac|^2\les \frac{\mr(\Rk+\ep\mr^2)^\frac{1}{2}+\Rk+\ep\mr^2}{|u|^{s+2}}.
    \end{align*}
    Combining with Proposition \ref{estba}, we infer
    \begin{align*}
\int_{\ucuv}|\nab_3\aa|^2\les\int_{\ucuv}r^{-10}|\nab_3(r^5\aa)|^2+r^{-2}|\aa|^2\les \int_{\ucuv}r^{-2}|\aac|^2+r^{-2}|\aa|^2\les\frac{\mr(\Rk+\ep\mr^2)^\frac{1}{2}+\Rk+\ep\mr^2}{|u|^{s+2}}.
    \end{align*}
    This concludes the proof of Proposition \ref{estaa3}.
\end{proof}
\subsection{Estimate for \texorpdfstring{$\ov{\rho}$}{}}
\begin{prop}\label{barrho}
We have the following estimate for $p\in[2,4]$:
\begin{equation}
    |r^{3}\ov{\rho}|_{\infty,S}(u,\ub)\les\mathfrak{R}_0+\ep\mr.
\end{equation}
\end{prop}
\begin{proof}
We recall the Bianchi equation
\begin{equation*}
    \nab_3\ov{\rho}+\frac{3}{2}\trchb\,\ov{\rho}=h[\ov{\rho}_3].
\end{equation*}
Applying Lemma \ref{evolutionlemma}, we infer for $p\in[2,4]$
\begin{align*}
    |r^{3-\frac{2}{p}}\ov{\rho}|_{p,S}(u,\ub)&\les \mathfrak{R}_0+\int_{u_0(\ub)}^u|r^{3-\frac{2}{p}}h[\ov{\rho}_3]|_{p,S}du\\
    &\les\mathfrak{R}_0+\int_{u_0(\ub)}^u|r^{3-\frac{2}{p}}\Gag\cdot(\aa,\ov{\rho})|_{p,S}du\\
    &\les\mathfrak{R}_0+\int_{u_0(\ub)}^u\left(\frac{\ep}{|u|^{\frac{s-3}{2}}}\frac{\mr}{|u|^{\frac{s+1}{2}}}+\frac{\ep\mr}{r^2|u|^{\frac{s-3}{2}}}\right)du\\
    &\les\mathfrak{R}_0+\ep\mr.
\end{align*}
Since $\ov{\rho}$ is constant on $S(u,\ub)$, this concludes the proof of Proposition \ref{barrho}.
\end{proof}
\subsection{End of the proof of Theorem M1}
\begin{prop}\label{prop9.1}
We have the following estimate:
\begin{equation}
    \mr_0^{S}+\ur_0^{S}\lesssim \mr_0+\mr_1+\ur_0+\ur_1.
\end{equation}
\end{prop}
\begin{proof}
It is sufficient to prove that
\begin{equation*}
\mr_0^{S}\lesssim\mr_0+\mr_1,\qquad\ur_0^{S}\lesssim\ur_0+\ur_1,
\end{equation*}
which is a direct consequence of Propositions \ref{prop7.6} and \ref{prop7.8}.
\end{proof}
Finally, we deduce from Proposition \ref{estab}-\ref{prop9.1} that
\begin{equation*}
\mr^2\les\mr(\Rk+\ep\mr^2)^\frac{1}{2}+\Rk+\ep\mr^2.
\end{equation*}
Applying Cauchy-Schwarz inequality, we obtain for $\ep$ small enough,
\begin{equation*}
    \mr \les \mathfrak{R}_0,
\end{equation*}
this concludes the proof of Theorem M1.
\section{Ricci coefficients estimates (Theorem M3)}\label{sec10}
The goal of this section is to prove Theorem M3, which we recall below for convenience.
\begin{m3}
Assume that
\begin{align}
        \mo\leq\ep,\qquad \mr_0^S+\ur_0^S\leq\Delta_0,\qquad\mo^*(\unc_*)\leq \mathcal{I}_*, \qquad\uobr(\Si_0\setminus K)\leq\IIbr,
\end{align}
then, we have
\begin{equation}\label{old5.2}
\mobr \les\IIbr+\II_*+\De_0+\ep^2.
\end{equation}
If we assume in addition that
\begin{equation}
    \uo(\Si_0\setminus K)\leq\II_0,
\end{equation}
then, we have
\begin{equation}\label{old5.4}
    \mo\lesssim \mathcal{I}_0+\mathcal{I}_*+\Delta_0+\epsilon^2.
\end{equation}
\end{m3}
In this section, we always assume $p\in [2,4]$.
\subsection{Estimates for \texorpdfstring{$\mo_0(\Om\om)$}{} and \texorpdfstring{$\mo_0(\Om\omb)$}{}}\label{ssec10.3}
\begin{prop}\label{prop10.3}
We have the following estimates:
\begin{align}
    \begin{split}
        |r^{2-\frac{2}{p}}\Omega\omega|_{p,S}(u,\ub) &\les \IIbr+\Delta_0+\epsilon^2\\
        |r^{2-\frac{2}{p}}\Omega\uome|_{p,S}(u,\ub) &\les\mathcal{I}_*+\Delta_0+\epsilon^2.
    \end{split}
\end{align}
\end{prop}
\begin{proof}
By the null structure equations \eqref{omom}, we have
\begin{align}
\D_3(\Omega\omega)=\widehat{F}+\frac{1}{2}\Omega\rho, \qquad \D_4(\Omega\uome)=\widehat{\uf}+\frac{1}{2}\Omega\rho,\label{34omega}
\end{align}
where
\begin{align*}
\widehat{F}&:=-\frac{1}{2}\Omega\zeta\cdot(\eta+\etab)-\frac{1}{2}\Omega(\ue\cdot\eta-2\zeta^2)=\Gag\cdot\Gag,\\
\widehat{\uf}&:=\frac{1}{2}\Omega\zeta\cdot(\eta+\etab)-\frac{1}{2}\Omega(\ue\cdot\eta-2\zeta^2)=\Gag\cdot\Gag.
\end{align*}
Apply Lemma \ref{evolutionlemma}, we obtain
\begin{align*}
|r^{-\frac{2}{p}}\Omega\uome|_{p,S}(u,\unu) &\lesssim |r^{-\frac{2}{p}}\Omega\uome|_{p,S}(u,\unu_*)+\int_{\unu}^{\unu_*} |r^{-\frac{2}{p}}\widehat{\underline{F}}|_{p,S} + \int_{\unu}^{\unu_*} |r^{-\frac{2}{p}}\Omega\rho|_{p,S},\\
|r^{-\frac{2}{p}}\Omega\omega|_{p,S}(u,\unu) &\lesssim |r^{-\frac{2}{p}}\Omega\omega|_{p,S}(u_0(\unu),\unu)+\int_{u_0(\unu)}^{u} |r^{-\frac{2}{p}}\widehat{F}|_{p,S} +\int_{u_0(\unu)}^{u} |r^{-\frac{2}{p}}\Omega\rho|_{p,S}.
\end{align*}
Notice that we have
\begin{equation*}
    |r^{-\frac{2}{p}}(\widehat{F},\widehat\uf)|_{p,S} \les\frac{\epsilon^2}{r^4|u|^{s-3}} ,\qquad |r^{-\frac{2}{p}}\Omega\rho|_{p,S} \lesssim r^{-3}|r^{3-\frac{2}{p}}\rho|_{p,S}\lesssim \frac{\Delta_0}{r^{3}}.
\end{equation*}
Hence, we obtain
\begin{align*}
|r^{2-\frac{2}{p}}\Omega\uome|_{p,S}(u,\unu) &\lesssim |r^{2-\frac{2}{p}}\Omega\uome|_{p,S}(u,\unu_*)+\ep^2+\Delta_0\lesssim \mathcal{I}_*+ \Delta_0+\epsilon^2, \\
|r^{2-\frac{2}{p}}\Omega\omega|_{p,S}(u,\unu) &\lesssim |r^{2-\frac{2}{p}}\Omega \omega|_{p,S}(u_0(\unu),\unu)+\epsilon^2+\Delta_0\les\IIbr+\Delta_0+\epsilon^2.
\end{align*}
which concludes the proof.
\end{proof}
\subsection{Estimates for \texorpdfstring{$\mo_{0,1}(\widecheck{\Om\om})$}{} and \texorpdfstring{$\mo_{0,1}(\widecheck{\Om\omb})$}{}}\label{ssec10.4}
\begin{prop}\label{prop10.4}
We have the following estimates:
\begin{align}
\begin{split}\label{estombom}
|r^{2+\frac{s-3}{6}-\frac{2}{p}}|u|^{\frac{s-3}{3}}(r\nab)^q\widecheck{\Omega\omega}|_{p,S}(u,\ub)&\les \IIbr+\Delta_0+\epsilon^2,\qquad q=0,1,\\
|r^{1-\frac{2}{p}}|u|^{\frac{s-1}{2}}(r\nab)^q\widecheck{\Om\omb}|_{p,S}(u,\ub) &\les \mathcal{I}_*+\Delta_0+\epsilon^2,\qquad q=0,1.
\end{split}
\end{align}
\end{prop}
\begin{proof}
First, we derive an evolution equation for $\nab(\Om\om)$. Applying \eqref{omom} and Corollary \ref{commcor}, we obtain
\begin{align*}
    \Om\nab_3\nab(\Om\om)&=[\Om\nab_3,\nab](\Om\om)+\nab(\Om\nab_3(\Om\om))\\
    &=-\Om\chib\cdot\nab(\Om\om)+\nab(\Om\nab_3(\Om\om))\\
    &=-\Om\chib\cdot\nab(\Om\om)+\nab\left(\frac{1}{2}\Om^2\rho+\Gag\cdot\Gag\right)\\
    &=-\frac{1}{2}\Om\trchb\nab(\Om\om)+\frac{1}{2}\Om^2\nab\rho+(\Gaa,\Gab)\cdot\nab\Gag,
\end{align*}
which implies
\begin{equation}
    \nab_3\nab(\Om\om)+\frac{1}{2}\trchb\nab(\Om\om)=\frac{1}{2}\nab(\Om\rho)+r^{-1}(\Gaa,\Gab)\cdot\Gag^{(1)}.
\end{equation}
Since we cannot control $\nab\rho$ in $L^p(S)$, we use a renormalization method. We define $\om^\dagger$ by
\begin{align}\label{omdagger}
    \nab_3\om^\dagger=\frac{1}{2}\Om\sigma \quad \mbox{in }\kk,\qquad \om^\dagger=0\quad \mbox{on }S(u_0(\ub),\ub).
\end{align}
Applying Corollary \ref{commcor}, we have
\begin{align*}
    \Om\nab_3{^*\nabla\om^\dagger}={^*\nab}(\Om\nab_3\om^\dagger)+[\Om\nab_3,{^*\nab}]\om^\dagger=\frac{1}{2}{^*\nab(\Om^2\si)}-\Om\chib\cdot{^*\nab\om^\dagger},
\end{align*}
which implies
\begin{equation}
    \nab_3{^*\nab\om^\dagger}+\frac{1}{2}\trchb({^*\nab\om^\dagger})=\frac{1}{2}{^*\nab}(\Om\si)+\Gab\cdot(\si,{^*\nab}\om^\dagger).
\end{equation}
We recall the following Bianchi equation from Proposition \ref{prop7.4}:
\begin{align*}
\nab_3 \b +\trchb\b &=\nab\rho+{^*\nab}\si+2\omb\b+2\hch\cdot\bb+3(\eta\rho+{^*\eta} \si)\\
&=\nab \rho+ {^*\nab}\si +\Gaa\cdot\b+\Gag\cdot \bb,
\end{align*}
which implies\footnote{Notice that $\b\nab_3\Om=-2\Om\omb\b=\Gaa\cdot\b$ and $(\rho,\si)\nab\Om=\Gaa\cdot\nab\Gag$ decay better than the R.H.S. of \eqref{Ombeta}. }
\begin{equation}\label{Ombeta}
    \nab_3(\Om\b)+\trchb (\Om\b)=\nab(\Om\rho)+{^*\nab}(\Om\si) +r^{-1}\Gaa\cdot\Gag^{(1)}+r^{-1}\Gag\cdot\Gab^{(1)}.
\end{equation}
We introduce a new quantity:
\begin{equation}
{\varkappa}:=\nab(\Om\om)+{^*\nabla\om^\dagger}-\frac{1}{2}\Om\b.
\end{equation}
Thus, we have the following equation:
\begin{align}
    \begin{split}\label{old5.56}
        \nab_3{\varkappa}&=\nab_3\nab(\Om\om)+\nab_3{^*\nab}\om^\dagger-\frac{1}{2}\nab_3(\Om\b)\\
        &=-\frac{1}{2}\trchb\nab(\Om\om)-\frac{1}{2}\trchb({^*\nab}\om^\dagger)+\frac{1}{2}\Om\trchb\,\b+r^{-1}(\Gaa,\Gab)\cdot\Gag^{(1)}+\Gab\cdot{^*\nab}\om^\dagger\\
        &=-\frac{1}{2}\trchb\,\varkappa+\frac{1}{4}\Om\trchb\,\b+r^{-1}(\Gaa,\Gab)\cdot\Gag^{(1)}+\Gab\cdot{^*\nab}\om^\dagger.
    \end{split}
\end{align}
Denoting $\langle\om\rangle:=(-\Om\om,\om^\dagger)$, then\footnote{Recall that $d_1^*(f,g)=-\nabla f+{^*\nabla}g.$}
\begin{align}
    \begin{split}
    d_1^*\langle\om\rangle&={\varkappa}+\frac{1}{2}\Om\b.\label{hoho}
    \end{split}
\end{align}
Applying Proposition \ref{prop7.3} to \eqref{hoho}, we obtain
\begin{align}
|r^{1-\frac{2}{p}}\nab(\Om\om)|_{p,S}+|r^{1-\frac{2}{p}}\nab\om^\dagger|_{p,S}\lesssim|r^{1-\frac{2}{p}}{\varkappa}|_{p,S}+|r^{1-\frac{2}{p}}\b|_{p,S}.\label{varkappa}
\end{align}
Also, applying Lemma \ref{evolutionlemma} to \eqref{old5.56}, we have
\begin{align*}
    &\;\;\;\,\,\,|r^{1-\frac{2}{p}}{\varkappa}|_{p,S}(u,\ub)\\
    &\les|r^{1-\frac{2}{p}}{\varkappa}|_{p,S}(u_0(\ub),\ub)+\int_{u_0(\ub)}^{u}|r^{-\frac{2}{p}}\b|_{p,S}+|r^{-\frac{2}{p}}(\Gaa,\Gab)\cdot\Gag^{(1)}|_{p,S}+ |r^{1-\frac{2}{p}}\Gab\cdot{^*\nab}\om^\dagger|_{p,S}\\
    &\les\frac{\IIbr+\De_0}{r^{\frac{s+1}{2}}}+\int_{u_0(\ub)}^{u}\frac{\De_0}{r^\frac{7}{2}|u|^\frac{s-4}{2}}+\frac{\ep^2}{r^4|u|^{\frac{s-3}{2}}}+\frac{\ep^2}{r^3|u|^{s-2}}+\frac{\ep}{r|u|^{\frac{s-1}{2}}}|r^{1-\frac{2}{p}}({^*\nab}\om^\dagger)|_{p,S}\\
    &\les \frac{\IIbr+\De_0+\ep^2}{r^{2+\frac{s-3}{6}}|u|^{\frac{s-3}{3}}}+\int_{u_0(\ub)}^{u}\frac{\ep}{r|u|^{\frac{s-1}{2}}}|r^{1-\frac{2}{p}}{\varkappa}|_{p,S},
\end{align*}
where at the last step we used for all $s>3$
\begin{align}\label{sdayu3}
    2+\frac{s-3}{6}\leq \min \left\{3, \frac{s+1}{2}\right\}.
\end{align}
Applying Gronwall inequality, we infer
\begin{align*}
    |r^{3+\frac{s-3}{6}-\frac{2}{p}}|u|^{\frac{s-3}{3}}{\varkappa}|_{p,S}(u,\unu) \lesssim\IIbr+\Delta_0+\epsilon^2.
\end{align*}
Injecting it into \eqref{varkappa} and recall that
\begin{align*}
   |r^{3+\frac{s-3}{6}-\frac{2}{p}}|u|^\frac{s-3}{3}\b|_{p,S}\les |r^{\frac{7}{2}-\frac{2}{p}}|u|^\frac{s-4}{2}\b|_{p,S}\les \De_0,
\end{align*}
we infer
\begin{align}
\begin{split}
    |r^{3+\frac{s-3}{6}-\frac{2}{p}}|u|^{\frac{s-3}{3}}\nab(\Om\om)|_{p,S}(u,\unu) \lesssim\IIbr+\Delta_0+\epsilon^2.
\end{split}
\end{align}
Applying Poincar\'e inequality, we conclude the first estimate of \eqref{estombom}. The second estimate is similar and left to the reader. This concludes the proof of Proposition \ref{prop10.4}.
\end{proof}
\subsection{Estimates for \texorpdfstring{$\mo_{0,1}(\widecheck{\Om})$}{} and \texorpdfstring{$r\left|\Om-\frac{1}{2}\right|$}{}}\label{ssecOmc}
\begin{prop}\label{propOmc}
We have the following estimates:
\begin{align}\label{estOmc}
    \left|r^{1-\frac{2}{p}}|u|^{\frac{s-3}{2}}(r\nab)^q\Omc\right|_{p,S}(u,\unu) &\les \IIbr+\mathcal{I}_*+ \Delta_0+\epsilon^2,\qquad q=0,1,\\
     \sup r\left|\Om-\frac{1}{2}\right| &\les \IIbr+\mathcal{I}_*+ \Delta_0+\epsilon^2.\label{Om1/2}
\end{align}
\end{prop}
\begin{proof}
We recall that
\begin{equation*}
    \Om\nab_4\Om=\Om^2\nab_4\log\Om=-2\Om^2\om.
\end{equation*}
Differentiating by $r\nab$ and applying Proposition \ref{commutation}, we deduce
\begin{align*}
    \Om\nab_4 (r\nab\Om)&=[\Om\nab_4,r\nab]\Om-2\Om r\nab(\Om\om)-2r\nab\Om (\Om\om)\\
    &=O(r\nab(\Om\om))+r\Gaa\cdot\Gag.
\end{align*}
Applying Lemma \ref{evolutionlemma} and Proposition \ref{prop10.4}, we infer
\begin{align*}
    |r^{1-\frac{2}{p}}\nab\Om|_{p,S}(u,\ub)&\les|r^{1-\frac{2}{p}}\nab\Om|_{p,S}(u,\ub_*)+\int_{\ub}^{\ub_*} \left(|r^{1-\frac{2}{p}}\nab({\Om\om})|_{p,S} +\frac{\ep^2}{r^3|u|^{\frac{s-3}{2}}}\right)\\
    &\les \frac{\IIbr+\mathcal{I}_*+ \Delta_0+\epsilon^2}{r|u|^{\frac{s-3}{2}}},
\end{align*}
where we used the fact that $\Om=\frac{1}{2}$ on $\Cb_*$. Combining with Poincar\'e inequality, we deduce \eqref{estOmc}. On the other hand, we have
\begin{align*}
    \Om\nab_4\left(\Om-\frac{1}{2}\right)=-2\Om^2\om.
\end{align*}
Recall that Propositions \ref{prop10.3}, \ref{prop10.4} and \ref{standardsobolev} implies
\begin{align*}
    |\Om \om|_{\infty,S}\les \frac{\IIbr+\II_*+\De_0+\ep^2}{r^2}.
\end{align*}
Applying Lemma \ref{evolutionlemma}, we obtain
\begin{align*}
    \left|\Om-\frac{1}{2}\right| \les \int_{\ub}^{\ub_*} |\Om \om|_{\infty,S} \les \frac{\IIbr+\II_*+\De_0+\ep^2}{r},
\end{align*}
which implies \eqref{Om1/2}. This concludes the proof of Proposition \ref{propOmc}.
\end{proof}
\subsection{Estimate for \texorpdfstring{$\mo_{0,1}(\widecheck{\Om\trch})$}{} }\label{ssec10.5}
\begin{prop}\label{prop10.5}
We have the following estimate:
\begin{align}
    \begin{split}\label{omtrch}
    \left|r^{2-\frac{2}{p}}|u|^{\frac{s-3}{2}}(r\nab)^q\widecheck{\Om\trch}\right|_{p,S}(u,\ub) &\les\IIbr+\mathcal{I}_*+ \Delta_0+\epsilon^2,\qquad q=0,1.
    \end{split}
\end{align}
\end{prop}
\begin{proof}
We recall the following equation from Proposition \ref{nulles}:
\begin{equation}
    \Om\nab_4(\Om\tr\chi)+\frac{1}{2}\Om\trch(\Om\trch)=-4\Om\trch(\Om\om)-\Om^2|\hch|^2 .\label{4trchi}
\end{equation}
We derive an evolution equation for $r\nab(\Omega\tr\chi)$. Differentiating \eqref{4trchi} by $r\nab$ and applying Proposition \ref{commutation} to obtain
\begin{align}
\begin{split}
&\Om\nab_4(r\nab(\Om\trch))+\Om\trch(r\nab\Om\trch)\\
=&[\Om\nab_4,r\nab](\Om\trch)-4r\nab(\Om\trch)\Om\om-4\Om\trch\, r\nab(\Om\om)+\Gag\cdot \Gag^{(1)}\\
=&\Gaa\cdot \Gag^{(1)}+O(\nab(\Om\om)).\label{ovchi}
\end{split}
\end{align}
Applying Lemma \ref{evolutionlemma} and Proposition \ref{prop10.4}, we obtain
\begin{align*}
|r^{3-\frac{2}{p}}\nab(\Om\trch)|_{p,S}(u,\ub) &\les |r^{3-\frac{2}{p}}\nab(\Om\trch)|_{p,S}(u,\ub_*)+\int_{\ub}^{\ub_*}|r^{2-\frac{2}{p}}\nab(\Om\om)|_{p,S}+|r^{2-\frac{2}{p}}\Gag^{(1)}\cdot\Gaa|_{p,S}\\
&\les\frac{\II_*}{|u|^{\frac{s-3}{2}}}+\int_\ub^{\ub_*}\frac{\De_0+\IIbr+\ep^2}{r^{1+\frac{s-3}{6}}|u|^{\frac{s-3}{3}}}\\
&\les \frac{\IIbr+\II_*+\De_0+\ep^2}{|u|^{\frac{s-3}{2}}}.
\end{align*}
Applying Poincar\'e inequality, we obtain \eqref{omtrch}. This concludes the proof of Proposition \ref{prop10.5}.
\end{proof}
\subsection{Estimate for \texorpdfstring{$\left|r^2\left(\overline{\Omega\tr\chi}-\frac{1}{r}\right)\right|$}{}}
\begin{prop}\label{prop10.9}
We have the following estimate:
\begin{align}
\sup\left|r^2\left(\overline{\Omega\tr\chi}-\frac{1}{r}\right)\right|\lesssim\mathcal{I}_0+\mathcal{I}_*+\Delta_0+\epsilon^2.
\end{align}
\end{prop}
\begin{proof}
Applying Lemma \ref{dint}, we obtain
\begin{equation}
    \frac{d}{d\unu}\frac{1}{r}=-\frac{1}{r^2}\frac{\pa r}{\pa \unu} =-\frac{1}{2r} \overline{\Omega\tr\chi} = -\frac{1}{2r}\Omega\tr\chi+\frac{\widecheck{\Om\trch}}{2r}.
\end{equation}
Recalling \eqref{4trchi} and denoting 
\begin{equation}\label{defW}
    W:=\overline{\Omega\tr\chi}-\frac{1}{r},
\end{equation}
we deduce
\begin{align}
\begin{split}\label{dubW}
    \frac{d}{d\unu}W+\frac{1}{2}\Omega\tr\chi W &=\frac{1}{2}W(\widecheck{\Om\trch})+\frac{1}{2}\overline{(\widecheck{\Om\trch})^2}-\overline{4\Omega\tr\chi(\Omega\omega)+\Omega^2|\hch|^2}\\
    &=-4\ov{\Om\trch\Om\om} +W\cdot\Gag+\Gag\cdot\Gag.
\end{split}
\end{align}
Applying Proposition \ref{prop10.3} and $W\in\Gaa$, we infer
\begin{align*}
|r^{1-\frac{2}{p}}\Om\trch\,\Om\om|_{p,S}(u,\unu)&\les \frac{\mathcal{I}_0 + \mathcal{I}_*+ \Delta_0+\epsilon^2}{r^2},\\
|r^{1-\frac{2}{p}}W\cdot\Gag|_{p,S}(u,\unu)+|r^{1-\frac{2}{p}}\Gag\cdot\Gag|_{p,S}(u,\unu)&\les \frac{\mathcal{I}_0 + \mathcal{I}_*+ \Delta_0+\epsilon^2}{r^3|u|^{\frac{s-3}{2}}} .
\end{align*}
Using Lemma \ref{evolutionlemma}, we have
\begin{align*}
    |r^{1-\frac{2}{p}}W|_{p,S}(u,\unu)\les \frac{\mathcal{I}_0 + \mathcal{I}_*+ \Delta_0+\epsilon^2}{r}.
\end{align*}
Recalling that $\nab W=0$, applying Proposition \ref{standardsobolev}, we deduce
\begin{align*}
    |r^2 W|_{\infty,S}\les \mathcal{I}_0 + \mathcal{I}_*+ \Delta_0+\epsilon^2.
\end{align*}
In view of \eqref{defW}, this concludes the proof of Proposition \ref{prop10.9}.
\end{proof}
\subsection{Estimate for \texorpdfstring{$\mo_{0,1}(\widecheck{\Om\trchb})$}{}}\label{ssec10.6}
\begin{prop}\label{prop10.6}
We have the following estimate:
\begin{equation}\label{esttrchbc}
    \left|r^{2-\frac{2}{p}}|u|^{\frac{s-3}{2}}(r\nab)^q\widecheck{\Om\trchb}\right|_{p,S}(u,\unu) \les \IIbr+\mathcal{I}_*+ \Delta_0+\epsilon^2,\qquad q=0,1.
\end{equation}
\end{prop}
\begin{proof}
We recall the following equation from Proposition \ref{nulles}:
\begin{equation}\label{nab3Om-1}
    \Om\nab_3(\Om\trchb)+\frac{1}{2}\Om\trchb(\Om\trchb)=-4\Om\trchb(\Om\omb)-\Om^2|\hchb|^2.
\end{equation}
As in \eqref{ovchi}, we derive an evolution equation for $r\nab(\Om\trchb)$. Differentiating \eqref{nab3Om-1} by $r\nab$ and applying Proposition \ref{commutation} to obtain
\begin{align}
\begin{split}
&\Om\nab_3(r\nab(\Om\trchb))+\Om\trchb(r\nab\Om\trchb)\\
=&[\Om\nab_3,r\nab](\Om\trchb)-4r\nab(\Om\trchb)\Om\omb-4\Om\trchb\, r\nab(\Om\omb)+\Gab\cdot\Gab^{(1)}\\
=&\Gaa\cdot \Gag^{(1)}+\Gab\cdot \Gab^{(1)}+O(\nab(\Om\omb)).\label{ovchib}
\end{split}
\end{align}
Applying Lemma \ref{evolutionlemma}, we obtain
\begin{align*}
    |r^{3-\frac{2}{p}}\nab{(\Om\trchb)}|_{p,S}(u,\ub)&\les|r^{3-\frac{2}{p}}\nab{(\Om\trchb)}|_{p,S}(u_0(\ub),\ub)\\
    &+\int_{u_0(\ub)}^u\left(|r^{2-\frac{2}{p}}\nab{(\Om\omb)}|_{p,S}+|r^{2-\frac{2}{p}}\Gab\cdot\Gab^{(1)}|+|r^{2-\frac{2}{p}}\Gaa\cdot \Gag^{(1)}|_{p,S}\right)\\
    &\les |r^{3-\frac{2}{p}}\nab{(\Om^{-1}\trchb)}|_{p,S}(u_0(\ub),\ub)+|r^{2-\frac{2}{p}}\nab\Om|_{p,S}\\
    &+\int_{u_0(\ub)}^u\left(|r^{2-\frac{2}{p}}\nab{(\Om\omb)}|_{p,S}+\frac{\ep^2}{r^2|u|^\frac{s-3}{2}}+\frac{\ep^2}{|u|^{s-1}}\right)\\
    &\les \frac{\IIbr+\II_*+\De_0+\ep^2}{|u|^\frac{s-3}{2}},
\end{align*}
where we used Propositions \ref{prop10.4}, \ref{propOmc} and the fact that $\uobr\leq\IIbr$ at the last step. Combining with Poincar\'e inequality, this concludes the proof of Proposition \ref{prop10.6}.
\end{proof}
\subsection{Estimates for \texorpdfstring{$\mo_{0,1}(\eta)$}{} and \texorpdfstring{$\mo_{0,1}(\etab)$}{}}\label{ssec10.7}
The following proposition plays an essential role in section \ref{sec10}.
\begin{prop}\label{prop10.7}
We have the following estimates:
\begin{align}
    \begin{split}\label{esteta}
        |r^{2-\frac{2}{p}}|u|^{\frac{s-3}{2}}(r\nab)^q\eta|_{p,S}(u,\unu)&\lesssim \IIbr+\mathcal{I}_*+\Delta_0+\epsilon^2,\qquad q=0,1,\\
        |r^{2-\frac{2}{p}}|u|^{\frac{s-3}{2}}(r\nab)^q\ue|_{p,S}(u,\unu)&\lesssim \mathcal{I}_0+\mathcal{I}_*+\Delta_0+\epsilon^2,\qquad q=0,1.
    \end{split}
\end{align}
\end{prop}
\begin{proof}
We recall the following null structure equations from Proposition \ref{nulles}:
\begin{align*}
    \nabla_4\eta&= -\chi\cdot\eta + \chi\cdot\ue-\beta, \\
    \nabla_3\ue&= -\unchi\cdot\ue +\unchi\cdot\eta+\unb.
\end{align*}
We introduce the mass aspect functions
\begin{align}
    \begin{split}
        \mu&:=-\sdiv \eta +\frac{1}{2} \widehat\chi\cdot\widehat\unchi-\rho, \\
        \uu&:= -\sdiv\ue+\frac{1}{2}\widehat\unchi\cdot\widehat\chi-\rho.\label{defmu}
    \end{split}
\end{align}
By a direct computation, we obtain\footnote{See Lemma 4.3.1 of \cite{kn}.}
\begin{align}
    \begin{split}\label{computationlongmu}
\nab_4\mu +\tr\chi\,\mu&=G+\frac{1}{2}\tr\chi\,\uu, \\
\nab_3\uu +\tr\unchi\,\uu&=\underline{G}+\frac{1}{2}\tr\unchi \,\mu,
    \end{split}
\end{align}
where
\begin{align*}
    G=\trch\,\rho+\Gab\cdot\nab\Gag,\qquad
    \underline{G}=\trchb\,\rho+\Gab\cdot\nab\Gab.
\end{align*} 
We introduce the following modifications of $\mu$ and $\uu$:
\begin{equation}\label{mumc}
    [\mu]:=\mu+\frac{1}{4}\trch\,\trchb,\qquad [\mub]:=\mub+\frac{1}{4}\trch\,\trchb.
\end{equation}
We also define
\begin{align}
\begin{split}
    \widecheck{[\mu]}&:=[\mu]-\ov{[\mu]}=-\sdiv \eta +\frac{1}{2} (\widehat\chi\cdot\widehat\unchi-\overline{\widehat\chi\cdot\widehat\unchi})+\frac{1}{4}\widecheck{\trch\trchb}-\rhoc, \\
    \widecheck{[\mub]}&:=[\mub]-\ov{[\mub]}=-\sdiv \ue + \frac{1}{2} (\widehat\chi\cdot\widehat\unchi-\overline{\widehat\chi\cdot\widehat\unchi})+\frac{1}{4}\widecheck{\trch\trchb}-\rhoc,\label{tildemu}
\end{split}
\end{align}
and remark that $\widecheck{[\mu]},\widecheck{[\mub]}\in r^{-1}\Gag^{(1)}$.
To simplify the notations, we denote
\begin{equation*}
    I:=\frac{1}{2} (\widehat\chi\cdot\widehat\unchi-\overline{\widehat\chi\cdot\widehat\unchi})+\frac{1}{4}\widecheck{\trch\trchb}.
\end{equation*}
By the torsion equation \eqref{torsion} and \eqref{tildemu}, we obtain
\begin{align}
\begin{split}
\sdiv\eta&=-\widecheck{[\mu]}+I-\rhoc,\\
\curl\eta&=\sigma+ \frac{1}{2}\widehat\unchi\wedge\widehat\chi,\label{4.3.37}
\end{split}
\end{align}
and
\begin{align}
\begin{split}
    \sdiv\etab&=-\widecheck{[\mub]}+I-\rhoc,\\
    \curl\etab&= -\sigma- \frac{1}{2}\widehat\unchi\wedge\widehat\chi. \label{4.3.38}
\end{split}
\end{align}
According to \eqref{defmu} and applying Proposition \ref{nulles}, we infer
\begin{align*}
    \nab_4[\mu]&=\nabla_4\left(\mu+\frac{1}{4}\tr\chi\tr\unchi\right)\\
    &=\nab_4\mu+\frac{1}{4}\Om^{-1}\trch\nabla_4(\Om\trchb)+\frac{1}{4}\nab_4(\Om^{-1}\tr\chi)\Om\tr\unchi \\
    &=-\tr\chi\mu+G+\frac{1}{2}\tr\chi\left(\frac{1}{2}\widehat\unchi\cdot\widehat\chi-\rho-\sdiv\ue\right)\\
    &+\frac{1}{4}\tr\chi\left[-\frac{1}{2}\tr\chi\tr\unchi-\widehat\chi\cdot\widehat\unchi+2\sdiv\ue+2|\etab|^2+2\rho\right]+\frac{1}{4}\trchb\left[-\frac{1}{2}(\tr\chi)^2-|\widehat\chi|^2\right]\\
    &=-\trch [\mu]+G+\frac{1}{2}\trch|\etab|^2-\frac{1}{4}\trchb|\hch|^2.
\end{align*}
Thus, we obtain
\begin{align}\label{keymuc}
    \nab_4[\mu]+\tr\chi[\mu]=\trch\,\rho+\Gab\cdot\nab\Gag.
\end{align}
By Corollary \ref{dav}, we infer
\begin{align*}
\Om\nab_4\overline{[\mu]}=\ov{\widecheck{\Om\trch}\widecheck{[\mu]}}+\ov{\Om\nab_4[\mu]}=-\ov{\Om\trch [\mu]}+\ov{\Om\trch\,\rho}+\Gab\cdot\nab\Gag.
\end{align*}
Then, we obtain
\begin{align}\label{keymucc}
    \nab_4\widecheck{[\mu]}+\trch\widecheck{[\mu]}=r^{-1}O(\rhoc)+\Gab\cdot\nab\Gag.
\end{align}
Applying Lemma \ref{evolutionlemma}, we deduce
\begin{align}
\begin{split}
    |r^{2-\frac{2}{p}}\widecheck{[\mu]}|_{p,S}(u,\unu)\les& |r^{2-\frac{2}{p}}\widecheck{[\mu]}|_{p,S}(u,\unu_*)+\int_{\unu}^{\unu_*}|r^{1-\frac{2}{p}}\rhoc|_{p,S}+\int_{\unu}^{\unu_*}|r^{2-\frac{2}{p}}\Gab\cdot\nab\Gag|_{p,S}\\
    \les&\frac{\II_*+\De_0}{r|u|^{\frac{s-3}{2}}}+\int_{\unu}^{\unu_*}\frac{\De_0}{r^2|u|^{\frac{s-3}{2}}}+\int_{\unu}^{\unu_*}\frac{\ep^2}{r^2|u|^{s-2}}\\
    \les&\frac{\II_*+\De_0+\ep^2}{r|u|^{\frac{s-3}{2}}}.
\end{split}
\end{align}
Applying Proposition \ref{prop7.3} to \eqref{4.3.37}, recalling that 
\begin{align*}
I=\frac{1}{4}\widecheck{\trch\trchb}+\Gag\cdot\Gab=r^{-1}O(\trchc)+r^{-1}O(\trchbc)+\Gag\cdot\Gab,    
\end{align*}
we have for $q=0,1$
\begin{align*}
|r^{2-\frac{2}{p}}(r\nab)^q\eta|_{p,S}(u,\ub)&\les|r^{3-\frac{2}{p}}\widecheck{[\mu]}|_{p,S}(u,\ub)+|r^{2-\frac{2}{p}}(\widecheck{\trch},\widecheck{\trchb})|_{p,S}(u,\ub)\\
&+|r^{2-\frac{2}{p}}\Gag\cdot\Gab|_{p,S}(u,\ub)+|r^{3-\frac{2}{p}}(\rhoc,\si)|_{p,S}(u,\ub)\\
&\les\frac{\IIbr+\mathcal{I}_*+\Delta_0+\epsilon^2}{|u|^{\frac{s-3}{2}}},
\end{align*}
where we used Propositions \ref{propOmc}, \ref{prop10.5} and \ref{prop10.6}. This concludes the first estimate of \eqref{esteta}. The second estimate is similar and left to the reader. This concludes the proof of Proposition \ref{prop10.7}.
\end{proof}
Remark that Propositions \ref{prop10.5} and \ref{prop10.7} implies \eqref{old5.2}, which concludes the first part of Theorem M3. In the sequel, we focus on the second part and hence we assume $\uo(\Si_0\setminus K)\leq\II_0$.
\subsection{Estimates for \texorpdfstring{$\mo_{0,1}(\widehat\chi)$}{} and \texorpdfstring{$\mo_{0,1}(\hchb)$}{}}\label{ssec10.1}
\begin{prop}\label{prop10.1}
We have the following estimate:
\begin{align}
    \begin{split}\label{hchhchbest}
        |r^{2-\frac{2}{p}}|u|^{\frac{s-3}{2}}(r\nab)^q\hch|_{p,S}(u,\unu) &\les\II_0+\II_* + \De_0+\ep^2,\qquad q=0,1,\\
        |r^{1-\frac{2}{p}}|u|^{\frac{s-1}{2}}(r\nab)^q\hchb|_{p,S}(u,\unu) &\les\II_0+\II_*+\De_0+\ep^2,\qquad q=0,1.
    \end{split}
\end{align}
\end{prop}
\begin{proof}
We recall the following Codazzi equations from Proposition \ref{nulles}
\begin{align*}
    \sdiv \hch &=\frac{1}{2}\nab\trch-\ze\cdot \left(\hch-\frac{1}{2}\trch\right)-\b,\\
    \sdiv \hchb&=\frac{1}{2}\nab\trchb+\ze\cdot \left(\hchb-\frac{1}{2}\trchb\right)+\bb.
\end{align*}
Applying Propositions \ref{propOmc}, \ref{prop10.5}, \ref{prop10.7} and the fact that $\ze=\frac{\eta-\etab}{2}$, we obtain
\begin{align*}
    |r^{3-\frac{2}{p}}\sdiv\hch|_{p,S}&\les |r^{3-\frac{2}{p}}\nab\trch|_{p,S}+|r^{2-\frac{2}{p}}\ze|_{p,S}+|r^{3-\frac{2}{p}}\b|_{p,S}+|r^{3-\frac{2}{p}}\Gag\cdot\Gag|_{p,S}\\
    &\les \frac{\II_0+\II_*+\De_0+\ep^2}{|u|^\frac{s-3}{2}}+\frac{\ep^2}{r|u|^{s-3}}\\
    &\les \frac{\II_0+\II_*+\De_0+\ep^2}{|u|^\frac{s-3}{2}},
\end{align*}
and respectively
\begin{align*}
    |r^{2-\frac{2}{p}}\sdiv\hchb|_{p,S}&\les |r^{2-\frac{2}{p}}\nab\trch|_{p,S}+|r^{1-\frac{2}{p}}\ze|_{p,S}+|r^{2-\frac{2}{p}}\bb|_{p,S}+|r^{2-\frac{2}{p}}\Gag\cdot\Gab|_{p,S}\\
    &\les \frac{\II_0+\II_*+\De_0+\ep^2}{|u|^\frac{s-1}{2}}+\frac{\ep^2}{r|u|^{s-2}}\\
    &\les \frac{\II_0+\II_*+\De_0+\ep^2}{|u|^\frac{s-1}{2}}.
\end{align*}
Combining with Proposition \ref{prop7.3}, these conclude the proof of Proposition \ref{prop10.1}.
\end{proof}
\subsection{Estimate for \texorpdfstring{$\mo_{2}(\Omc)$}{}}\label{ssec10.8}
\begin{prop}\label{prop10.8}
We have the following estimate for $q=0,1,2$:
\begin{align}
\begin{split}\label{estimationOm}   
|r^{1+q-\frac{2}{p}}|u|^\frac{s-3}{2}\nab^q\Omc|_{p,S}(u,\ub)&\les\mathcal{I}_0+ \mathcal{I}_*+ \Delta_0+\epsilon^2.
\end{split}
\end{align}
\end{prop}
\begin{proof}
Recall that $\eta+\etab=2\nabla\log\Omega$, which implies
\begin{equation*}
    \Delta \log\Omega =\frac{1}{2}\sdiv (\eta+\ue).
\end{equation*}
Applying Propositions \ref{prop7.3} and \ref{prop10.7}, we obtain
\begin{align}
\begin{split}
|r^{{1}-\frac{2}{p}}(r\nab)^q({\log\Omega})|_{p,S}\les |r^{3-\frac{2}{p}}\sdiv(\eta+\ue)|_{p,S} \lesssim\frac{\mathcal{I}_0+ \mathcal{I}_*+ \Delta_0+\epsilon^2}{|u|^{\frac{s-3}{2}}},\qquad q=0,1,2. \label{4.3.64}
\end{split}
\end{align}
Combining with \eqref{estOmc}, this concludes the proof of Proposition \ref{prop10.8}.
\end{proof}
\subsection{Estimate for \texorpdfstring{$\left|r^2\left(\overline{\Omega\tr\unchi}+\frac{1}{r}\right)\right|$}{}}
\begin{prop}\label{prop10.10}
We have the following estimate:
\begin{align}
\sup\left|r^2\left(\overline{\Omega\tr\unchi}+\frac{1}{r}\right)\right|\lesssim \mathcal{I}_0+ \mathcal{I}_*+\De_0+\ep^2.
\end{align}
\end{prop}
\begin{proof}
We introduce the following notation:
\begin{align*}
    \underline{W} := \overline{\Omega\tr\unchi}+\frac{1}{r}.
\end{align*}
As in \eqref{dubW}, we have
\begin{equation}
\Om\nab_3\underline{W}+\frac{1}{2}\Om\tr\chib\,\underline{W} =-4\overline{{\Om\trchb(\Omega\omb)}}+\Gag\cdot\underline{W}+\Gab\cdot\Gab,
\end{equation}
Using Lemma \ref{evolutionlemma} and $\underline{W}\in\Gaa$, we have
\begin{align*}
    \left|r^{1-\frac{2}{p}}\underline{W}\right|_{p,S}(u,\ub)&\les\left|r^{1-\frac{2}{p}}\underline{W}\right|_{p,S}(u_0(\ub),\ub)+ \int_{u_0(\ub)}^u \frac{\II_0+\II_*+\De_0+\ep^2}{r^2}+\frac{\II_0+\II_*+\De_0+\ep^2}{r|u|^{s-1}}\\
    &\les\frac{\mathcal{I}_0+\mathcal{I}_*+\Delta_0+\epsilon^2}{r}.
\end{align*}
Noticing that $\nab\underline{W}=0$, taking $p=4$ and applying Proposition \ref{standardsobolev}, we conclude the proof of Proposition \ref{prop10.10}.
\end{proof}
In view of Propositions \ref{prop10.3}-\ref{prop10.10}, we obtain \eqref{old5.4}. This concludes the proof of Theorem M3.
\section{Initialization and extension (Theorems M0, M2 and M4)}\label{sec11}
\subsection{Preliminaries}
\subsubsection{Null frame transformations}\label{nullframetrans}
\begin{lem}\label{lemchange}
Let two null frames $(e_3,e_4,e_1,e_2)$ and $(e'_3,e'_4,e'_1,e'_2)$ associated to double null foliations $(u,\ub)$ and $(u',\ub)$, and assume that they have the same generator $\unl$ for the incoming direction. Then, a null frame transformation from the null frame $(e_3,e_4,e_1,e_2)$ to $(e'_3,e'_4,e'_1,e'_2)$ can be written in the form:
\begin{align}
    \begin{split}\label{change}
        e'_4&=\lambda\left(e_4+f_Be_B+\frac{1}{4}|f|^2e_3\right),\\
        e'_3&=\lambda^{-1}e_3,\\
        e'_A&=e_A+\frac{1}{2}f_A e_3,
    \end{split}
\end{align}
where 
\begin{align}\label{lambdaOmOm'}
    \la=\frac{\Om}{\Om'}
\end{align}
is a scalar function and $f$ is a $1$--form. Moreover, the inverse transform of \eqref{change} is given by
\begin{align}
\begin{split}
e_4&=\la^{-1}e'_4-f^A e_A'+\frac{\la}{4}|f|^2 e_3',\\
e_3&=\la e'_3,\\
e_A&=e'_A-\frac{\la}{2}f_A e'_3. \label{change'}
\end{split}
\end{align} 
\end{lem}
\begin{proof}
Applying Lemma 3.1 in \cite{kl-sz1} with $\fb=0$, we obtain \eqref{change}. Notice that \eqref{change'} is a direct consequence of \eqref{change}. Recalling
\begin{align*}
    e'_3=2\Om' \Lb,\qquad e_3=2\Om \Lb,
\end{align*}
we obtain \eqref{lambdaOmOm'} immediately.
\end{proof}
\begin{prop}\label{transformation}
Under the null frame transform \eqref{change}, some of the Ricci coefficients transform as follows:
\begin{align*}
\la^{-1}\trch'&=\trch + \sdiv'f +\lot,\\
0&=\curl'f+\lot,\\
\la \chib'&=\chib,\\
\eta'&=\eta+\frac{1}{2}\la\nab_{e'_3}'f-\omb\, f+\lot,\\
\etab'&=\etab+\frac{1}{4}\trchb\, f+\lot,\\
\la^{-2}\xi'&=\xi+\frac{1}{2}\nab_{\la^{-1}e'_4}'f+\frac{1}{4}\trch\, f+\om f+\lot,\\
\la^{-1}\om'&=\om+\lot,
\end{align*}
where the terms denoted by $\lot$ have the following schematic structure
\begin{align*}
    \lot:= r^{-1}O(f^2)+\Gab\cdot f.
\end{align*}
The curvature components transform as follows:
\begin{align*}
\lambda^{-2}\a'&=\a+O(f)\b+O(f^2)\rho+O(f^3)(\si,\bb)+O(f^4)\aa,\\
\lambda^{-1}\b'&=\b+O(f)(\rho,\si)+O(f^2)\bb+O(f^3)\aa,\\
\rho'&=\rho+O(f)\bb+O(f^2)\aa,\\
\si'&=\si+O(f)\bb+O(f^2)\aa,\\
\lambda\unb'&=\bb+O(f)\aa,\\
\lambda^{2}\aa'&=\aa.
\end{align*}
\end{prop}
\begin{proof}
We only prove the transform formula of $\om$ since the others are direct consequences of Proposition 3.3 in \cite{kl-sz1} for $\fb=0$. Applying (3.13) in \cite{kl-sz1}, we have
\begin{equation*}
    \la^{-1}\om'=\om -\frac{1}{2}\la^{-1}e'_4(\log\la)+\err,
\end{equation*}
where
\begin{equation*}
    \err=r^{-1}O(f^2)+\Gab\cdot f.
\end{equation*}
Applying \eqref{lambdaOmOm'}, we infer
\begin{align*}
    \la^{-1}\om'&=\om-\frac{1}{2\la}e'_4(\log \Om-\log\Om')+\err\\
    &=\om-\frac{1}{2}\left(e_4+f^Be_B+\frac{1}{4}|f|^2e_3\right)\log\Om-\frac{1}{2\la}(2\om')\\
    &=2\om-\la^{-1}\om'+\err,
\end{align*}
which implies the transformation formula of $\om$.
\end{proof}
\subsubsection{Deformations of spheres}
\begin{lem}\label{deforsp}
Given two double null foliations $(u,\ub)$ and $(u',\ub)$ and we denote their leaves
\begin{equation*}
    S:=S(u,\ub),\qquad S':=S'(u',\ub).
\end{equation*}
Assume that $S\cap S'\ne\emptyset$ and the oscillation of $u$ on $S'$ satisfies
\begin{equation}\label{oscu}
    \sup_{S'}|u-\ov{u}'|\leq \de_1, \qquad\quad \ov{u}':=\frac{1}{|S'|} \int_{S'} u .
\end{equation}
Let $V$ be a tensor field satisfying
\begin{equation}\label{eqpruVf}
    \nab_{\pr_u} V = F,
\end{equation}
where for all $S$
\begin{equation}\label{estSF}
    |r^{-\frac{2}{p}}F|_{L^p(S)} \leq \de,\qquad p\in [1,\infty].
\end{equation}
Then, we have
\begin{equation}
    |r^{-\frac{2}{p}}V|_{L^p(S')}\les |r^{-\frac{2}{p}}V|_{L^p(S)} + \de_1\de.
\end{equation}
\end{lem}
\begin{proof}
The proof is largely analogous to Lemma 4.1.7 in \cite{kn}. \\ \\
Notice from \eqref{eqpruVf} that
\begin{align*}
        |\pr_u |V|^2|= |V\cdot (\nab_{\pr_u}F)|\leq |V| |F|,
\end{align*}
which implies
\begin{equation*}
|\pr_u|V||\leq |F|.
\end{equation*}
Taking $p\in S\cap S'$, we have from \eqref{oscu} that
\begin{equation}\label{2de1}
    \sup_{S'}|u-u(p)|\leq \sup_{S'}|u-\ov{u}'|+|u(p)-\ov{u}'|\leq 2\de_1.
\end{equation}
In a sphere coordinates $\phi=(\phi^1,\phi^2)$, we have
\begin{align}\label{Vu'Vu}
    |V|(u',\ub,\phi)-|V|(u,\ub,\phi)=\int_{0}^1 \pr_u(|V|)(u+t(u'-u),\ub,\phi)(u'-u)dt
\end{align}
Hence, we infer
\begin{align*}
|r^{-\frac{2}{p}}V|_{L^p(S')}^p &= \int_{\mathbb{S}^2}r^{-2}|V|^p (u',\ub,\phi)\sqrt{\det(g)|_{S'}}\,d\phi^1d\phi^2 \\
&\les\int_{\mathbb{S}^2}|V|^p(u',\ub,\phi) d\phi^1d\phi^2\\
&\les\int_{\mathbb{S}^2}|V|^p(u,\ub,\phi) d\phi^1d\phi^2+\int_{\mathbb{S}^2}\left(\int_{0}^1 |F|(u+t(u'-u),\ub,\phi)|u'-u|dt\right)^pd\phi^1d\phi^2\\
&\les\int_{\mathbb{S}^2}|r^{-\frac{2}{p}}V|^p\sqrt{\det(g)|_{S}}\,d\phi^1d\phi^2+\de_1^p\int_{\mathbb{S}^2}\left(\int_{0}^1 |F|(u+t(u'-u),\ub,\phi)dt\right)^pd\phi^1d\phi^2\\
&\les |r^{-\frac{2}{p}}V|_{L^p(S)}^p+\de_1^p |F(u+t(u'-u),\ub,\phi)|_{L^p(\mathbb{S}^2) L^1_t(0,1)}^p\\
&\les |r^{-\frac{2}{p}}V|_{L^p(S)}^p+\de_1^p |r^{-\frac{2}{p}}F|_{L^1_t(0,1)L^p(S(u+t(u'-u),\ub))}^p\\
&\les |r^{-\frac{2}{p}}V|_{L^p(S)}^p+\de_1^p\de^p,
\end{align*}
where we used \eqref{estSF}, \eqref{2de1}, \eqref{Vu'Vu} and General Minkowski inequality. This concludes the proof of Lemma \ref{deforsp}.
\end{proof}
\begin{rk}\label{notationprime}
    In the remainder of this paper, for any quantity $X$, we denote the quantity $X_{(0)}$ associated to the initial data layer region $\kk_{(0)}$ by
\begin{align*}
    X':=X_{(0)},
\end{align*}
to simplify the notations. For example:
\begin{align*}
    u':=u_{(0)},\qquad \eta':=\eta_{(0)},\qquad \mo':=\mo_{(0)},\qquad \mathfrak{R}':=\mathfrak{R}_{(0)},\qquad S':=S_{(0)},\qquad \kk':=\kk_{(0)}.
\end{align*}
\end{rk}
\subsection{The initial hypersurface (Theorem M0)}\label{ssec11.1}
The goal of this section is to prove Theorem M0 which we recall below for convenience.
\begin{m0}
Under the assumptions
\begin{align}\label{assumptionM0}
    \mo'\leq \ep_0,\qquad \mathfrak{R}'\leq\ep_0, \qquad \mo\leq\ep,\qquad \osc\leq\ep,
\end{align}
we have
\begin{equation}\label{estfrakR}
    \mathfrak{R}_0\les \ep_0, \qquad \uobr(\Si_0\setminus K)\les\ep_0.
\end{equation}
If in addition we assume that
\begin{equation}\label{additionassumption}
    \mobr\les\ep_0,
\end{equation}
then, we have
\begin{equation}\label{estoscuo}
    \osc\les\ep_0,\qquad \uo(\Si_0\setminus K)\les \ep_0.
\end{equation}
\end{m0}
Recall that we denoted by $(f,\la)$ the frame transformation from $(u,\ub)$ to $(u',\ub)$, and let $(f',\la')$ its inverse transformation. Notice from \eqref{change} and \eqref{change'} that we have
\begin{equation}\label{inversetrans}
    \la'=\la^{-1},\qquad f'=-\la f.
\end{equation}
\begin{rk}\label{GagGabsame}
In $\kk'$, we have $|u|\simeq r$. Hence, $\Gab$ has the same decay as $\Gag$ in $\kk'$. The assumption $\osc\leq\ep$ implies that $r\simeq r'$ and $f\in r\Gag$.
\end{rk}
We first prove \eqref{estfrakR} under the assumptions in \eqref{assumptionM0}.
\begin{lem}\label{a1}
Under the assumptions
\begin{align*}
    \mo'\leq \ep_0,\qquad \mathfrak{R}' \leq\ep_0, \qquad \mo\leq\ep,\qquad \osc\leq\ep,
\end{align*}
we have
\begin{equation*}
    \mathfrak{R}_0\les \ep_0.
\end{equation*}
\end{lem}
\begin{proof}
Recall that
\begin{align*}
\Rk=\int_{\Sigma_0\setminus K}\sum_{l=0}^1 {r}^{s}\left(|{\mathfrak{d}}^l\alpha|^2+|{\mathfrak{d}}^l\beta|^2+|{\mathfrak{d}}^l (\chr,\sigma)|^2+|{\mathfrak{d}}^l\unb|^2+|{\mathfrak{d}}^l \aa|^2\right)+\sup_{\Sigma_0\setminus K} |{r}^3\overline{\rho}|^2,
\end{align*}
Applying Proposition \ref{transformation} with \eqref{inversetrans}, we infer\footnote{Recall that $\dk^{\leq 1}f=O(\ep)r^{-\frac{s-1}{2}}$ and $\rho'=O(\ep)r^{-3}$.}
\begin{align*}
    \int_{\Si_0\setminus K} r^s|\dk^{\leq 1}\a|^2 &\les \int_{\Si_0\setminus K} r^s|\dk^{\leq 1}(\la^{-2}\a')|^2 + r^s \left|\dk^{\leq 1}(\la^{-1} f\cdot \b')\right|^2+r^s\left|\dk^{\leq 1}(f^2\cdot\rho')\right|^2+\lot\\
    &\les {\Rfk'}^2+\int_{\Si_0\setminus K}r^s|\dk^{\leq 1}(f^2)\rho'|^2+r^s|f^2\dk^{\leq 1}\rho'|^2\\
    &\les {\Rfk'}^2+\int_{\Si_0\setminus K}r^s \left|\frac{\ep^3}{r^{s+2}}\right|^2\\
    &\les \ep_0^2.
\end{align*}
Similarly, we have
\begin{align*}
    \int_{\Si_0\setminus K} r^s|\dk^{\leq 1}\b|^2 &\les \int_{\Si_0\setminus K} r^s|\dk^{\leq 1}(\la^{-1}\b')|^2 +r^s|\dk^{\leq 1}f\cdot(\rho',\si')|^2 \\
    &\les {\Rfk'}^2 + \int_{\Si_0\setminus K} r^s \left|\frac{\ep^2}{r^{\frac{s+5}{2}}} \right|^2\\
    &\les \ep_0^2.
\end{align*}
By the same method, we obtain:
\begin{align*}
    \int_{\Si_0\setminus K}r^s|\dk^{\leq 1}\rhoc|^2+r^s|\dk^{\leq 1}\si|^2+r^s|\dk^{\leq 1}\bb|^2+r^s|\dk^{\leq 1}\aa|^2\les\ep_0^2.
\end{align*}
Finally, we estimate
\begin{align*}
    |r^3 \ov{\rho}| \les |r^3\rho|_\infty \les |r^3\rho'|_\infty +|r^4\Gag\cdot\bb'|_\infty\les \ep_0.
\end{align*}
This concludes the proof of Lemma \ref{a1}.
\end{proof}
\begin{lem}\label{a1.5}
Under the assumptions
\begin{align*}
    \mo'\leq \ep_0,\qquad \mo\leq\ep,\qquad \osc\leq\ep,
\end{align*}
we have
\begin{equation*}
    \uobr(\Si_0\setminus K)\les \ep_0.
\end{equation*}
\end{lem}
\begin{proof}
We recall from Proposition \ref{transformation} that\footnote{Recall that $\Gab=\Gag$ in $\kk_{(0)}$ and $f\in r\Gag$.}
\begin{align*}
    {\Om'}^{-1}\trchb'&=\Om^{-1}\trchb+r\Gag\cdot \Gag,\\
    \Om'\om'&=\Om\om+r\Gag\cdot \Gag.
\end{align*}
Hence, we obtain
\begin{align*}
    |r^{2-\frac{2}{p}}\Om\om|_{p,S(u_0(\ub),\ub)}&\les |r^{2-\frac{2}{p}}\Om'\om'|_{p,S(u_0(\ub),\ub)}+ |r^{3-\frac{2}{p}}\Gag\cdot \Gag|_{p,S(u_0(\ub),\ub)}\\
    &\les |r^{2}\Om'\om'|_{\infty,S(u_0(\ub),\ub)}+ \frac{\ep^2}{r^{s-2}}\\
    &\les \ep_0.
\end{align*}
Notice from \eqref{change} and Proposition \ref{nulles} that
\begin{align*}
    re_A(\Om\om)&=re_A'(\Om'\om')-\frac{\la}{2} rf_A e'_3 (\Om'\om')+r\Gag\cdot \Gag^{(1)}\\
    &=re_A'(\Om'\om')+O(r\rho')f+r\Gag\cdot \Gag^{(1)}.
\end{align*}
We deduce
\begin{align*}
    |r^{1-\frac{2}{p}}\nab(\Om\om)|_{p,S(u_0(\ub),\ub)}&\les |r^{1-\frac{2}{p}}\nab'\Om'\om'|_{p,S(u_0(\ub),\ub)}+|r^{2-\frac{2}{p}}\Gag^{(1)}\cdot \Gaa|_{p,S(u_0(\ub),\ub)}\\
    &\les  |r\nab'\Om'\om'|_{\infty,S(u_0(\ub),\ub)}+\frac{\ep^2}{r^{\frac{s+1}{2}}}\\
    &\les \frac{\ep_0}{r^{\frac{s+1}{2}}}.
\end{align*}
Similarly, we have
\begin{align*}
    |r^{1-\frac{2}{p}}\nab(\Om^{-1}\trchb)|_{p,S(u_0(\ub),\ub)}\les \frac{\ep_0}{r^\frac{s+1}{2}}.
\end{align*}
Applying Poincar\'e inequality and recalling the definition of $\uobr(\Si_0\setminus K)$, we obtain
\begin{equation*}
    \uobr(\Si_0\setminus K)=\sum_{q=0}^1 \left(\uo_q(\Si_0\setminus K)(\widecheck{\Om^{-1}\trchb})+\uo_q(\Si_0\setminus K)(\widecheck{\Om\om})\right)+\uo_0(\Si_0\setminus K)(\Om\om)\les \ep_0.
\end{equation*}
This concludes the proof of Lemma \ref{a1.5}.
\end{proof}
Remark that Lemmas \ref{a1} and \ref{a1.5} imply \eqref{estfrakR}, which concludes the first part of Theorem M0. We now focus on the second part of Theorem M0 under the additional assumption \eqref{additionassumption}.
\begin{lem}[Oscillation lemma]\label{oscillation}
Under the assumptions:
\begin{equation}\label{mobrass}
    \OO_{(0)}\leq \ep_0,\qquad \mo\leq\ep,\qquad \osc\leq \ep,\qquad \mobr\les \ep_0,
\end{equation}
we have
\begin{equation}
    \osc\les\ep_0.
\end{equation}
\end{lem}
\begin{proof} {\bf Step 1.} Estimate for $\osc(f)$. \\ \\
We have from Proposition \ref{transformation} and Remark \ref{GagGabsame}
\begin{align*}
    \la^{-1}\trch'&=\trch+\sdiv' f+f\c(\Gag,r^{-1}f),\\
    0&=\curl'f+f\c(\Gag,r^{-1}f),
\end{align*}
which implies from \eqref{lambdaOmOm'} and the fact that $f\in r\Gag$\footnote{We have from \eqref{change} and Proposition \ref{transformation} that $\nab'f-\nab f=O(f)\c\nab_3 f=f\c\Gag$.}
\begin{align}
\begin{split}\label{ellipticf}
    \Om'\trch'&=\Om\trch+\Om\sdiv f+f\cdot\Gag,\\
    0&=\curl f+f\cdot \Gag.
\end{split}
\end{align}
Differentiating it by $e_A$ and applying \eqref{change}, we obtain
\begin{align*}
    e_A'(\Om'\trch')-\frac{\la}{2}f_A e'_3(\Om'\trch')=e_A(\Om\trch)+\Om e_A(\sdiv f)+\Gag\cdot \Gag^{(1)}.
\end{align*} 
Recall that we have from Proposition \ref{nulles}
\begin{align*}
    e_3'(\Om'\trch')&=-\frac{1}{2}\Om'\trch'\trchb'+r^{-1}\Gaa+r^{-1}\Gag^{(1)}+\Gag\cdot\Gag\\
    &=\frac{1}{r^2}+r^{-1}\Gaa+\Gag\cdot\Gag.
\end{align*}
Hence, we have
\begin{align*}
    \Om e_A(\sdiv f)+\frac{1}{2r^2}f_A=e'_A(\Om'\trch')-e_A(\Om\trch)+\Gaa\cdot\Gag^{(1)},
\end{align*}
which implies
\begin{align*}
    e_A(\sdiv f)+\frac{1}{r^2}f_A=\Om^{-1}(e'_A(\Om'\trch')-e_A(\Om\trch))+\Gaa\cdot\Gag^{(1)}.
\end{align*}
We also have from \eqref{ellipticf}
\begin{align*}
    e_A(\curl f)=\Gag\c\Gag^{(1)}.
\end{align*}
Hence, we infer
\begin{align}
    \begin{split}\label{ellipticsystemf}
        \left(d_1^*d_1-\frac{1}{r^2}\right)f&=-\nab\sdiv f+{^*\nab}\curl f-\frac{1}{r^2}f\\
        &=-\Om^{-1}\nab'(\Om'\trch')+\Om^{-1}\nab(\Om\trch)+\Gaa\c\Gag^{(1)}.
    \end{split}
\end{align}
Recall from \eqref{dddd} and Corollary \ref{prop7.6} that
\begin{align*}
    d_1^*d_1-\frac{1}{r^2}&=-\De_1+\mathbf{K}-\frac{1}{r^2}\\
    &=-\De_1-\frac{1}{4}\trch\,\trchb+\Gag\c\Gab-\rho-\frac{1}{r^2}\\
    &=-\De_1+r^{-1}\Gaa+r^{-1}\Gag^{(1)}.
\end{align*}
Injecting it into \eqref{ellipticsystemf}, we deduce
\begin{align*}
     -\De_1f=-\Om^{-1}\nab'(\Om'\trch')+\Om^{-1}\nab(\Om\trch)+\Gaa\c\Gag^{(1)}.
\end{align*}
We have from Proposition \ref{prop7.3}
\begin{align*}
    |r^{\frac{s-1}{2}-\frac{2}{p}}(r\nab)^{\leq 2} f|_{p,S}\les |r^{\frac{s+3}{2}-\frac{2}{p}}(\nab'\Om'\trch',\nab\Om\trch)|_{p,S}+|r^{\frac{s+3}{2}-\frac{2}{p}}\Gaa\c\Gag^{(1)}|_{p,S}\les \ep_0.
\end{align*}
Applying Proposition \ref{standardsobolev}, we deduce
\begin{equation}\label{nabfcontrol}
    |r^{\frac{s-1}{2}}(r\nab)^{\leq 1}f|_{\infty,S} \les\ep_0.
\end{equation}
We have from Proposition \ref{transformation} that
\begin{align}
\begin{split}\label{nab34f}
    \sup_{\kk'}|\nab_3 f| &\les \sup_{\kk'}|\eta'-\eta|+\sup_{\kk'}|f\cdot \Gaa|\les\frac{\ep_0}{r^{\frac{s+1}{2}}}+ \frac{\ep^2}{r^{\frac{s+3}{2}}}\les \frac{\ep_0}{r^{\frac{s+1}{2}}},\\
    \sup_{\kk'}|\nab_4 f| &\les \sup_{\kk'}|r^{-1}f| +\sup_{\kk'}|f\cdot \Gaa|\les\frac{\ep_0}{r^{\frac{s+1}{2}}}+ \frac{\ep^2}{r^{\frac{s+3}{2}}}\les  \frac{\ep_0}{r^{\frac{s+1}{2}}}.
\end{split}
\end{align}
Combining \eqref{nabfcontrol} and \eqref{nab34f}, we deduce
\begin{align*}
    \sup_{\kk'}(|f|+|\dk f|)\les \frac{\ep_0}{r^\frac{s-1}{2}},
\end{align*}
which implies
\begin{equation}\label{oscf}
    \osc(f)\les\ep_0.
\end{equation}
{\bf Step 2.} Estimate for $\osc(\la)$.\\ \\
We recall from \eqref{lambdaOmOm'} that
\begin{align*}
    |r\ovla|=|r(\la-1)|\les r|\Om'-\Om|\les r\left|\Om'-\frac{1}{2}+\frac{1}{2}-\Om\right|\les \ep_0,
\end{align*}
where we used \eqref{mobrass} in the last step. Hence, we obtain
\begin{equation}\label{oscla}
    \osc(\la)\les\ep_0.
\end{equation}
{\bf Step 3.} Estimate for $\osc(r)$.\\ \\
We have
\begin{align}
\begin{split}\label{1r1r'}
    \frac{1}{r}-\frac{1}{r'}&=\Om'\trch'-\frac{1}{r'}-\Om\trch+\frac{1}{r}+(\Om\trch-\Om'\trch')\\
    &=\widecheck{\Om'\trch'}-\widecheck{\Om\trch}+\ov{\Om'\trch'}-\frac{1}{r'}+\ov{\Om\trch}-\frac{1}{r}+(\Om\trch-\Om'\trch').
\end{split}
\end{align}
Applying \eqref{ellipticf} and \eqref{nabfcontrol}, we deduce
\begin{align}\label{difOmtrch}
    |\Om\trch-\Om'\trch'|\les |\sdiv f|+ |r\Gag\cdot\Gag|\les \frac{\ep_0}{r^{\frac{s+1}{2}}}.
\end{align}
Moreover, we have from \eqref{mobrass} that
\begin{align}\label{difOmtrch1r}
    r^\frac{s+1}{2}\left(|\widecheck{\Om'\trch'}|+|\widecheck{\Om\trch}|\right)+r^2\left(\left|\ov{\Om'\trch'}-\frac{1}{r'}\right|+\left|\ov{\Om\trch}-\frac{1}{r}\right|\right)\les \ep_0.
\end{align}
Hence, we obtain from \eqref{1r1r'}, \eqref{difOmtrch} and \eqref{difOmtrch1r} that
\begin{align*}
    \left|\frac{1}{r}-\frac{1}{r'}\right|\les \frac{\ep_0}{r^2},
\end{align*}
which implies
\begin{equation}\label{oscr}
    \osc(r)\les \ep_0.
\end{equation}
Combining \eqref{oscf}, \eqref{oscla} and \eqref{oscr}, we infer
\begin{equation}
    \osc\les\ep_0.
\end{equation}
This concludes the proof of Lemma \ref{oscillation}.
\end{proof}
\begin{proof}[Proof of Theorem M0]
We have from Lemma \ref{oscillation} that 
\begin{equation}\label{osccontroldone}
    \osc\les\ep_0.
\end{equation}
We recall from Proposition \ref{transformation} that
\begin{align*}
    \etab'=\etab+\frac{1}{4}\trchb f+r\Gag\cdot \Gag.
\end{align*}
Applying \eqref{change}, we have
\begin{align*}
    \nab'\etab'=\nab\etab+r^{-1}O(\nab f)+ r\Gag\cdot\Gag^{(1)}.
\end{align*}
Thus, we obtain from \eqref{assumptionM0} and \eqref{osccontroldone} that
\begin{align}\label{etabinfty}
    \sup_{\kk'}|\etab|\les \sup_{\kk'}|\etab'|+\sup_{\kk'}|r^{-1}f|+\sup_{\kk'}|r\Gag\cdot\Gag|\les \frac{\ep_0}{r^\frac{s+1}{2}},
\end{align}
and for $p\in [2,4]$
\begin{align}
\begin{split}\label{etabnab24}
        |r^{1-\frac{2}{p}}\nab\etab|_{p,S(u_0(\ub),\ub)}&\les |r^{1-\frac{2}{p}}\nab'\etab'|_{p,S(u_0(\ub),\ub)}+ |r^{-\frac{2}{p}}\nab f|_{p,S(u_0(\ub),\ub)}+|r^{2-\frac{2}{p}}\Gag\cdot\Gag^{(1)}|_{p,S(u_0(\ub),\ub)}\\
        &\les |r\nab'\etab'|_{\infty,S(u_0(\ub),\ub)}+ |\nab f|_{\infty,S(u_0(\ub),\ub)}+\frac{\ep^2}{r^{s-1}}\\
        &\les \frac{\ep_0}{r^\frac{s+1}{2}}.
\end{split}
\end{align}
Finally, we recall from Proposition \ref{transformation} that
\begin{align*}
    \Om\trchb=\frac{\Om^2}{\Om'}\trchb'.
\end{align*}
Hence, we infer
\begin{align*}
    \Om\trchb+\frac{1}{r}&=\frac{\Om^2}{\Om'}\trchb'+\frac{1}{r}\\
    &=\left(1-\frac{\Om^2}{{\Om'}^2}\right)\trchb'+\Om'\trchb'+\frac{1}{r'}+\frac{1}{r}-\frac{1}{r'}.
\end{align*}
Thus, we have from \eqref{assumptionM0} and \eqref{osccontroldone}
\begin{align}\label{Omtrchb1r}
    r^2\left|\Om\trchb+\frac{1}{r}\right|&\les r^2\left|\Om'\trchb'+\frac{1}{r'}\right|+r|\ovla|+|r'-r|\les \ep_0.
\end{align}
Combining \eqref{etabinfty}, \eqref{etabnab24}, \eqref{Omtrchb1r} and Lemma \ref{a1.5} and recalling the definition of $\uo(\Si_0\setminus K)$, we obtain
\begin{equation}
    \uo(\Si_0\setminus K)\les \ep_0.
\end{equation}
This concludes the proof of Theorem M0.
\end{proof}
\subsection{The last slice (Theorem M2)}\label{ssec11.5}
We now assume that the last slice $\unc_*$ is endowed with a geodesic foliation. Notice that $\Om=\frac{1}{2}$ and $\omb=0$ on $\Cb_*$. We write the null structure equations along the last slice in the following form:
\begin{align}
    \begin{split}\label{nab3geo}
        \nabla_3 \zeta +\trchb\,\zeta+2\hchb\cdot\ze&=-\bb,\\
        \nabla_3 \tr\unchi +\frac{1}{2}(\tr\unchi)^2 +|\widehat\unchi|^2&=0,\\
        \nabla_3\widehat\chi+\frac{1}{2}\tr\unchi\,\widehat\chi+\frac{1}{2}\tr\chi\,\widehat\unchi-\nabla\widehat\otimes\zeta-\zeta\widehat\otimes\zeta&=0,\\
        \nabla_3\tr\chi+\frac{1}{2}\tr\unchi\tr\chi+\widehat\unchi\cdot\widehat\chi-2\sdiv\zeta-2|\zeta|^2&=2\rho.
    \end{split}
\end{align}
We also have the Codazzi equations:
\begin{align}
    \begin{split}\label{nab3codazzi}
        \sdiv\widehat\chi&=\frac{1}{2}\nabla\tr\chi-\zeta\left(\widehat\chi-\frac{1}{2}\tr\chi\right)-\beta,\\
        \sdiv\widehat\unchi&=\frac{1}{2}\nabla\tr\unchi+\zeta\left(\widehat\unchi-\frac{1}{2}\tr\unchi\right)+\unb.
    \end{split}
\end{align}
For convenience, we recall the statement of Theorem M2.
\begin{m2}
Let $\Cb_*$ endowed with a geodesic foliation. We assume that
\begin{equation}
    \mo^*(\unc_*)\leq\ep,\qquad\ur_0^S\leq\De_0,\qquad\mo'\leq \mathcal{I}_0.
\end{equation}
Then, we have
\begin{equation}
    \mo^*(\unc_*)\lesssim \Delta_0+\mathcal{I}_0+\epsilon^2.
\end{equation}
\end{m2}
In the remainder of this section, we always assume $p\in [2,4]$ and we denote
\begin{equation}
    S:=S(u,\ub_*),\qquad S_*:=\Cb_*\cap \Si_0.
\end{equation}
\begin{proof}[Proof of Theorem M2]
Notice that the assumption $\OO'\leq \II_0$ controls the initial data on the last sphere $S_*:=\Cb_*\cap \Si_0$. The idea of the proof is to transport the estimates in the direction of $\nab_3$ along $\Cb_*$ using \eqref{nab3geo}. The proof is similar to section \ref{sec10}, and in fact easier since $\Om=\frac{1}{2}$ on $\Cb_*$. We only provide a sketch.\\ \\
{\bf Step 1.} Estimate for $\mo^*_{0,1,2}(\trchbc)$.\\ \\
The estimates for $\mo_0^*(\trchb)$ and $\mo_1^*(\trchb)$ follow directly from the results of section \ref{ssec10.6} in the particular case $\Om=\frac{1}{2}$ and $\omb=0$. \\ \\
We then estimate $\nab^2\trchb$. According to the assumption $\mo^*(\Cb_*)\leq\ep$, we have on $\Cb_*$:
\begin{equation}\label{eq6.44}
    r\nab\trchb\in\Gag ,\qquad (r\nab)^2\trchb\in \Gag^{(1)}.
\end{equation}
We recall from \eqref{nab3geo} that
\begin{align*}
    \nab_3 \trchb +\frac{1}{2}(\trchb)^2=-|\hchb|^2.
\end{align*}
Differentiating it by $r^2\De$ and applying Proposition \ref{commutation}, we deduce
\begin{align}\label{nab3nab2trchb}
\nab_3 (r^2\De\trchb)+\trchb(r^2\De\trchb)=-r^2\De|\hchb|^2+\Gab\cdot\Gab^{(1)}.
\end{align}
We recall from \eqref{nab3codazzi} that
\begin{equation}\label{cosch}
    \sdiv\hchb=\frac{1}{2}\nab\trchb-\frac{1}{2}\trchb\,\ze+\bb+\Gag\cdot\Gab=\bb+r^{-1}\Gag.
\end{equation}
Applying Proposition \ref{prop7.3}, we obtain\footnote{Here, $\bb^{(0)}$ denotes a quantity which has same or even better decay and regularity than $\bb$.}
\begin{align*}
    \nab\hchb=\bb^{(0)}+r^{-1}\Gag.
\end{align*}
Hence, we have
\begin{equation*}
    |\nab\hchb|^2=\bb^{(0)}\c\bb^{(0)}+r^{-2}\Gag\c\Gab^{(1)}.
\end{equation*}
Moreover, we have from \eqref{dddd} and \eqref{cosch}
\begin{align*}
    2\De_2\hchb\c\hchb=4((\mathbf{K}-d_2^*d_2)\hchb)\c\hchb=-4d_2^*d_2\hchb\c\hchb+r^{-2}\Gab\c\Gab^{(1)}=-4(d_2^*\bb)\c\hchb+r^{-2}\Gab\c\Gab^{(1)}.
\end{align*}
We then compute
\begin{align*}
    \De |\hchb|^2=2\nab^C \left((\nab_C\hchb_{AB})\hchb^{AB}\right)=2(\De_2\hchb)\c\hchb+2|\nab\hchb|^2=-4(d_2^*\bb)\c\hchb+\bb^{(0)}\c\bb^{(0)}+r^{-2}\Gab\c\Gab^{(1)}.
\end{align*}
Injecting it into \eqref{nab3nab2trchb}, we deduce
\begin{equation}\label{nab3r2nab2trchb}
\nab_3(r^2\De\trchb)+\trchb(r^2\De\trchb)=4r^2(d_2^*\bb)\c\hchb+r^2\bb^{(0)}\c\bb^{(0)}+\Gab\cdot\Gab^{(1)}.
\end{equation}
Applying Proposition \ref{prop7.3}, we have
\begin{align*}
    |d_2^*\bb|_{p,S}\les |d_1\bb|_{p,S},
\end{align*}
which implies that there exists a bounded elliptic operator $A:L^p(S)\to L^p(S)$ of order $0$ such that
\begin{align*}
    d_2^*\bb = A (d_1\bb).
\end{align*}
We have from Proposition \ref{comm}
\begin{align}
\begin{split}\label{Acomm}
    [\Om\nab_3,A]&=[\Om\nab_3,rd_2^*\circ (rd_1)^{-1}]=[\Om\nab_3,rd_2^*]\circ (rd_1)^{-1}+rd_2^*\circ [\Om\nab_3,(rd_1)^{-1}]\\
    &=\Gab\c rd_2^* \circ (rd_1)^{-1}+r\bb^{(0)}\c (rd_1)^{-1}-rd_2^*\circ(rd_1)^{-1}\circ[\Om\nab_3,rd_1]\circ(rd_1)^{-1}\\
    &=\Gab\c A+r\bb^{(0)}\c (rd_1)^{-1}-A\circ\left(\Gab\c rd_1+r\bb^{(0)}\right)\circ (rd_1)^{-1}\\
    &=\Gab+r\bb^{(0)}\c(rd_1)^{-1}.    
\end{split}
\end{align}
We recall from Corollary \ref{prop7.6} and Lemma \ref{dav} that
\begin{align*}
    \nab_3(\rhoc,\sic)+\frac{3}{2}\trchb(\rhoc,\sic)=-d_1\bb+\Gag\cdot\aa.
\end{align*}
Combining with \eqref{Acomm}, we deduce
\begin{align*}
    \nab_3(r^2 A(\rhoc,\sic))=r(\rhoc,\sic)^{(0)}-A(r^2 d_1\bb)+r^3\bb^{(0)}\c(\rhoc,\sic)^{(0)}+r^2\Gag\cdot\aa.
\end{align*}
Hence, we have
\begin{align*}
4\hchb\cdot(r^2d_2^*\bb)=-4\nab_3(r^2A(\rhoc,\sic)\hchb)+r^2\aa^{(0)}\c(\rhoc^{(0)},\sic^{(0)})+\Gab\cdot \Gab^{(1)}.
\end{align*}
Injecting it into \eqref{nab3r2nab2trchb}, we infer\footnote{Note that $\aa^{(0)}\c(\rhoc,\sic)^{(0)}$ has the same decay as $\bb^{(0)}\c\bb^{(0)}$.}
\begin{align}
\begin{split}\label{renornon}
\nab_3\left(r^2\De\trchb+4r^2A(\rhoc,\sic)\hchb\right)+\trchb\left(r^2\De\trchb+4r^2A(\rhoc,\sic)\hchb\right)=r^2\bb^{(0)}\c\bb^{(0)}+\Gab\cdot\Gab^{(1)}.    
\end{split}
\end{align}
Applying Lemma \ref{evolutionlemma}, we have
\begin{align*}
    &\;\;\;\,\,\,\left|r^{4-\frac{2}{p}}(\De\trchb+4A(\rhoc,\sic)\hchb)\right|_{p,S}\\
    &\les\left|r^{4-\frac{2}{p}}(\De\trchb+4A(\rhoc,\sic)\hchb)\right|_{p,S_*}+\int^{u}_{u_0(\ub_*)}|r^{4-\frac{2}{p}}\bb^{(0)}\c\bb^{(0)}|_{p,S}+|r^{2-\frac{2}{p}}\Gab\cdot \Gab^{(1)}|_{p,S}\\
    &\les\frac{\II_0+\De_0}{|u|^\frac{s-3}{2}}+\int^{u}_{u_0(\ub_*)}|r^{2-\frac{2}{8}}\bb^{(0)}|_{8,S}^2+\frac{\ep^2}{|u|^{s-1}}\\
    &\les\frac{\II_0+\De_0+\ep^2}{|u|^\frac{s-3}{2}}+\int_{u_0(\ub_*)}^u|r^{2-\frac{2}{2}}\bb^{(1)}|_{2,S}^2\\
    &\les\frac{\II_0+\De_0+\ep^2}{|u|^\frac{s-3}{2}},
\end{align*}
where we used the Sobolev inequality
\begin{align*}
    |r^{-\frac{2}{8}}\bb^{(0)}|_{8,S}\les |r^{-\frac{2}{2}}\bb^{(1)}|_{2,S}.
\end{align*}
Hence, we obtain from Proposition \ref{prop7.3}
\begin{equation}
    \left|r^{4-\frac{2}{p}}\nab^2\trchb\right|_{p,S}\les\left|r^{4-\frac{2}{p}}\De\trchb\right|_{p,S}\les\frac{\II_0+\De_0+\ep^2}{|u|^\frac{s-3}{2}}+\left|r^{4-\frac{2}{p}}A(\rhoc,\sic)\hchb\right|_{p,S}\les\frac{\II_0+\De_0+\ep^2}{|u|^\frac{s-3}{2}}.
\end{equation}
Hence, we obtain
\begin{equation}\label{esttrchbcstar}
    \mo^*_0(\trchbc)+\mo^*_1(\trchbc)+\mo^*_2(\trchbc)\les \II_0+\De_0+\ep^2.
\end{equation}
{\bf Step 2.} Estimate for $\widecheck{[\mub]}$. \\ \\
We introduce the mass aspect function $\mub$ as in \eqref{defmu}:
\begin{align}\label{defmubb}
    \mub:=\sdiv\ze+\frac{1}{2}\hchb\cdot\hch-\rho.
\end{align}
We recall from Lemma 4.3.1 in \cite{kn} that
\begin{align*}
    \nab_3\mub+\trchb\,\mub =\underline{G}+\frac{1}{2}\trchb \mub,
\end{align*}
where
\begin{align*}
    \underline{G}=\trchb \,\rho-\frac{1}{2}\hchb\cdot(\nab\hot\ze)-\ze\cdot\bb+\Gag\c\nab\trchb+r^{-1}\Gab\cdot\Gab.
\end{align*}
Applying \eqref{eq6.44}, we have
\begin{equation*}
    \underline{G}=\trchb\,\rho-\frac{1}{2}\hchb\cdot(\nab\hot\ze)-\ze\cdot\bb+r^{-1}\Gab\cdot\Gab.
\end{equation*}
As in \eqref{mumc}, we define
\begin{align}\label{mubkk}
    [\mub]:=\mub+\frac{1}{4}\trch\,\trchb,
\end{align}
and deduce the following analog of \eqref{keymuc}:
\begin{equation}\label{keymubc}
    \nab_3{[\mub]}+\trchb\,{[\mub]}=\trchb\,\rho-\frac{1}{2}\hchb\cdot(\nab\hot\ze)-\ze\cdot\bb+r^{-1}\Gab\cdot\Gab.
\end{equation}
Differentiating \eqref{keymubc} by $r\nab$ and applying Corollary \ref{commcor}, we obtain
\begin{align}
\begin{split}\label{mubnab3}
\nab_3(r\nab{[\mub]})+\trchb(r\nab{[\mub]})&=\trchb\,(r\nab\rho)-\frac{1}{2}\hchb\cdot (r\nab\nab\hot\ze)-\ze\cdot(r\nab\bb)\\
&+r^{-1}\Gaa\cdot\Gag^{(1)}+r^{-1}\Gab\cdot\Gab^{(1)}.
\end{split}
\end{align}
Recall from \eqref{nab3geo} and Corollary \ref{prop7.6} that, we have
\begin{align*}
    \nab_3\hch+\frac{1}{2}\trchb\,\hch+\frac{1}{2}\trch\,\hchb&=\nab\hot\ze+\Gag\cdot\Gag,\\
    \nab_3(\rhoc,\si)+\frac{3}{2}\trchb(\rhoc,\si)&=-d_1\bb+\Gag\cdot\aa.
\end{align*}
We apply a renormalization method similar to the one in \eqref{renornon}, and leave the details to the reader. We obtain
\begin{equation}\label{nab3Xi}
\nab_3\Xi+\trchb\,\Xi=O(\nab\rho)+r^{-1}\Gaa\cdot\Gag^{(1)}+r^{-1}\Gab\cdot\Gab^{(1)},
\end{equation}
where
\begin{equation}\label{defXi}
    \Xi:=r\nab[\mub]+\frac{1}{2}\hchb\cdot(r\nab\hch)+r\ze\cdot\rhoc.
\end{equation}
Next, we use a renormalization method similar to the one in section \ref{ssec10.4}. We define $\mub^\dagger$ by:
\begin{align}\label{mubdagger}
    \nab_3\mub^\dagger+\trchb \,\mub^\dagger=\trchb\,\si \quad \mbox{ on }\Cb_*,\qquad \mub^\dagger=0\quad \mbox{ on }S_*.
\end{align}
We assume that
\begin{equation}\label{bootmubdagger}
    \mub^\dagger \in r^{-1}\Gag,
\end{equation}
which will be improved in \eqref{mubmubdagger} and \eqref{mubtrans}. Differentiating \eqref{mubdagger} by $r{^*\nab}$ and applying Corollary \ref{commcor}, we infer
\begin{align}\label{mubdaggernab3}
\nab_3(r{^*\nab}\mub^\dagger)+\trchb(r{^*\nab}\mub^\dagger)=r\trchb{^*\nab}\si+r^{-1}\Gag\cdot\Gag^{(1)}.
\end{align}
Recall from Corollary \ref{prop7.6} that 
\begin{equation*}
    \nab_3\b+\trchb \b=\nab\rhoc+{^*\nab\si}+\Gaa\cdot\b+\Gag\cdot \bb,
\end{equation*}
which implies
\begin{equation}\label{rtrchbbnab3}
    \nab_3(r\trchb \b)+\trchb(r\trchb\,\b)=\nab_3(r\trchb)\b+\trchb(r\nab\rhoc+r{^*\nab\si})+r^{-1}\Gaa\cdot\Gag^{(1)}+r^{-1}\Gag\cdot\Gab^{(1)}. 
\end{equation}
Combining \eqref{nab3Xi}, \eqref{mubdaggernab3} and \eqref{rtrchbbnab3}, we obtain
\begin{align}
\begin{split}\label{keynor}
&\nab_3(\Xi+r{^*\nab}\mub^\dagger-r\trchb\b)+\trchb(\Xi+r{^*\nab}\mub^\dagger-r\trchb\b)\\
=&r^{-1}O(\b)+r^{-1}\Gaa\cdot\Gag^{(1)}+r^{-1}\Gab\cdot\Gab^{(1)}.
\end{split}
\end{align}
Applying Lemma \ref{evolutionlemma}, we deduce
\begin{equation}
    |r^{2-\frac{2}{p}}(\Xi+r{^*\nab}\mub^\dagger-r\trchb\b)|_{p,S}\les \frac{\II_0+\De_0+\ep^2}{r|u|^{\frac{s-3}{2}}}.
\end{equation}
Recall from \eqref{defXi} that
\begin{align*}
    \Xi-r\nab\widecheck{[\mub]}=\Gab\cdot\Gag^{(1)},
\end{align*}
we obtain
\begin{equation*}
    |r^{2-\frac{2}{p}}(r\nab\widecheck{[\mub]}+r{^*\nab}\mub^\dagger)|_{p,S}\les |r^{2-\frac{2}{p}}\b|_{p,S}+|r^{2-\frac{2}{p}}\Gab\cdot\Gag^{(1)}|_{p,S}+\frac{\II_0+\De_0+\ep^2}{r|u|^{\frac{s-3}{2}}} \les\frac{\II_0+\De_0+\ep^2}{r|u|^{\frac{s-3}{2}}}.
\end{equation*}
Then, we have from Proposition \ref{prop7.3} that
\begin{align}\label{mubmubdagger}
    \left|r^{4-\frac{2}{p}}\nab\left([\mub],\mub^\dagger\right)\right|_{p,S}\les\frac{\II_0+\De_0+\ep^2}{|u|^\frac{s-3}{2}}.
\end{align}
Finally, applying Lemma \ref{evolutionlemma} to \eqref{mubdagger}, we obtain
\begin{equation}\label{mubtrans}
    |r^{2-\frac{2}{p}}\mub^\dagger|_{p,S}\les |r^{2-\frac{2}{p}}\mub^\dagger|_{p,S_*}+\int_{u_0(\ub_*)}^u|r^{1-\frac{2}{p}}\si|_{p,S}\les \frac{\II_0+\De_0}{r|u|^{\frac{s-3}{2}}}.
\end{equation}
Remark that \eqref{mubmubdagger} and \eqref{mubtrans} improves \eqref{bootmubdagger}.\\ \\
{\bf Step 3.} Estimates for $\mo_{0,1}(\trchc)$ and $\mo_{0,1}(\ze)$. \\ \\
We recall from \eqref{nab3geo} that
\begin{align*}
    \nabla_3\tr\chi+\frac{1}{2}\tr\unchi\tr\chi-2\sdiv\zeta=2\rho+\Gab\cdot\Gag.
\end{align*}
Applying \eqref{defmubb} and \eqref{mubkk}, we obtain
\begin{align*}
    \nab_3\trch+\trchb\trch&=2\sdiv\ze+2\rho+\frac{1}{2}\trch\trchb+\Gab\cdot\Gag\\
    &=2\mub+4\rho+\frac{1}{2}\trch\trchb+\Gab\cdot\Gag\\
    &=2[\mub]+4\rho+\Gab\cdot\Gag.
\end{align*}
Differentiating it by $r\nab$, we deduce
\begin{align*}
    \nab_3(r\nab\trch)+\trchb(r\nab\trch)&=-\trch(r\nab\trchb)+2r\nab[\mub]+4r\nab\rho+\Gab\cdot\Gag^{(1)}.
\end{align*}
We use a renormalization argument as in Step 2. For this, we define $\chi^\dagger$ by:
\begin{align}\label{defchidagger}
    \nab_3 \chi^\dagger+\trchb\,\chi^\dagger =4\si,\quad\mbox{ on }\Cb_*,\qquad \chi^\dagger=0 ,\quad\mbox{ on }S_*.
\end{align}
We add the bootstrap assumption
\begin{equation}\label{bootchidagger}
    \chi^\dagger\in\Gag,
\end{equation}
which will be improved by \eqref{trchchidagger} and \eqref{chidaggerimproved}. Differentiating it by $r{^*\nab}$ and applying Corollary \ref{commcor}, we infer
\begin{align*}
    \nab_3(r{^*\nab}\chi^\dagger)=4r{^*\nab}\si+\Gag\cdot\Gag^{(1)}.
\end{align*}
We recall from Corollary \ref{prop7.6} that 
\begin{equation*}
    \nab_3\b+\trchb\, \b=\nab\rhoc+{^*\nab\si}+\Gaa\cdot\b+\Gag\cdot \bb,
\end{equation*}
which implies
\begin{equation*}
    \nab_3(r\b)+\trchb\,(r\b)=-e_3(r)\b+r\nab\rhoc+r{^*\nab\si}+\Gaa\cdot\Gag^{(1)}+\Gag\cdot\Gab^{(1)}.
\end{equation*}
Hence, we obtain the following analog of \eqref{keynor}:
\begin{align*}
    &\nab_3(r\nab\trch+r{^*\nab}\chi^\dagger-4r\b)+\trchb(r\nab\trch+r{^*\nab}\chi^\dagger-4r\b)\\
    =&-e_3(r)\b-\trch(r\nab\trchb)+2r\nab[\mub]+\Gab\cdot\Gag^{(1)}+\Gaa\cdot\Gag^{(1)}.
\end{align*}
Applying Lemma \ref{evolutionlemma}, \eqref{esttrchbcstar} and \eqref{mubmubdagger}, we deduce 
\begin{align*}
|r^{2-\frac{2}{p}}(r\nab\trch+r{^*\nab}\chi^\dagger-4r\b)|_{p,S}&\les |r^{2-\frac{2}{p}}(r\nab\trch+r{^*\nab}\chi^\dagger-4r\b)|_{p,S_*}\\
&+\int_{u_0(\ub_*)}^{u}|r^{2-\frac{2}{p}}\b|_{p,S}+|r^{2-\frac{2}{p}}\nab\trchb|_{p,S}+|r^{3-\frac{2}{p}}\nab[\mub]|_{p,S}\\
&+\int_{u_0(\ub_*)}^{u}  |r^{2-\frac{2}{p}}\Gab\cdot\Gag^{(1)}|_{p,S} + |r^{2-\frac{2}{p}}\Gaa\cdot\Gag^{(1)}|_{p,S}\\
&\les \frac{\II_0+\De_0}{|u|^\frac{s-3}{2}}+\int_{u_0(\ub_*)}^{u} \frac{\II_0+\De_0+\ep^2}{r|u|^\frac{s-3}{2}}+\frac{\ep^2}{r|u|^{s-2}}+\frac{\ep^2}{r^2|u|^\frac{s-3}{2}}\\
&\les \frac{\II_0+\De_0+\ep^2}{|u|^\frac{s-3}{2}}.
\end{align*}
Hence, we obtain from Proposition \ref{prop7.3} that
\begin{align}\label{trchchidagger}
    |r^{3-\frac{2}{p}}\nab(\trch,\chi^\dagger)|_{p,S}\les |r^{3-\frac{2}{p}}\b|_{p,S}+ \frac{\II_0+\De_0+\ep^2}{|u|^\frac{s-3}{2}}\les \frac{\II_0+\De_0+\ep^2}{|u|^\frac{s-3}{2}}.
\end{align}
Applying Lemma \ref{evolutionlemma} to \eqref{defchidagger}, we easily obtain
\begin{equation}\label{chidaggerimproved}
    |r^{2-\frac{2}{p}}\chi^\dagger|_{p,S}\les |r^{2-\frac{2}{p}}\chi^\dagger|_{p,S_*}+\int_{u_0(\ub_*)}^u |r^{2-\frac{2}{p}}\si|_{p,S}\les \frac{\II_0+\De_0}{|u|^\frac{s-3}{2}},
\end{equation}
which improved \eqref{bootchidagger}. By Poincar\'e inequality, we infer
\begin{equation}\label{trchcstar}
    \mo_0(\trchc)+\mo_1(\trchc)\les \II_0+\De_0+\ep^2.
\end{equation}
Next, we recall from \eqref{4.3.38} that
\begin{align*}
    \sdiv \ze&=\widecheck{[\mub]}+\rhoc+\frac{1}{4}\widecheck{\trch\,\trchb}+\Gab\cdot\Gag,\\
    \curl \ze &=\si +\Gab\cdot\Gag.
\end{align*}
Hence, we have from Proposition \ref{prop7.3} that
\begin{align*}
    |r^{2-\frac{2}{p}}(r\nab)^{\leq 1}\ze|_{p,S}&\les |r^{3-\frac{2}{p}}\widecheck{[\mub]}|_{p,S}+|r^{3-\frac{2}{p}}(\rhoc,\si)|_{p,S}+|r^{1-\frac{2}{p}}(\trchc,\trchbc)|_{p,S}\\
    &\les \frac{\II_0+\De_0+\ep^2}{|u|^\frac{s-3}{2}},
\end{align*}
where we used \eqref{esttrchbcstar}, \eqref{mubmubdagger} and \eqref{trchcstar} in the last step. Hence, we obtain
\begin{equation}\label{estzestar}
   \mo_0(\ze)+\mo_1(\ze)\les \II_0+\De_0+\ep^2.
\end{equation}
{\bf Step 4.} Estimates for $\mo_{0,1}(\hch)$ and $\mo_{0,1}(\hchb)$. \\ \\
Applying  Proposition \ref{prop7.3} to \eqref{nab3codazzi} and recalling \eqref{esttrchbcstar}, \eqref{trchcstar} and \eqref{estzestar}, we obtain
\begin{align*}
    |r^{2-\frac{2}{p}}(r\nab)^{\leq 1}\hch|_{p,S}\les |r^{3-\frac{2}{p}}\nab\trchc|_{p,S}+|r^{2-\frac{2}{p}}\ze|_{p,S}+|r^{3-\frac{2}{p}}\b|+|r^{2-\frac{2}{p}}\Gag\cdot\Gag|_{p,S}\les \frac{\II_0+\De_0+\ep^2}{|u|^\frac{s-3}{2}},\\
    |r^{1-\frac{2}{p}}(r\nab)^{\leq 1}\hchb|_{p,S}\les |r^{2-\frac{2}{p}}\nab\trchbc|_{p,S}+|r^{2-\frac{2}{p}}\ze|_{p,S}+|r^{2-\frac{2}{p}}\bb|+|r^{1-\frac{2}{p}}\Gab\cdot\Gag|_{p,S}\les \frac{\II_0+\De_0+\ep^2}{r|u|^\frac{s-1}{2}}.
\end{align*}
Hence, we have
\begin{equation}\label{hchhchbstar}
    \mo_0(\hch)+\mo_1(\hch)+\mo_0(\hchb)+\mo_1(\hchb) \les \II_0+\De_0+\ep^2.
\end{equation}
{\bf Step 5.} Estimate for $r^2\left|\ov{\trch}-\frac{2}{r}\right|$. \\ \\
We have from Corollary \ref{dav} and \eqref{nab3geo} that\footnote{Recall that $\Om=\frac{1}{2}$ on $\Cb_*$.}
\begin{align*}
    \nab_3 (\ov{\trch})=\ov{\nab_3(\trch)}+\ov{\trchbc\,\trchc}=-\frac{1}{2}\ov{\trch\,\trchb}+2\ov{\rho}+\Gag\cdot\Gab.
\end{align*}
Notice that
\begin{equation*}
    \nabla_3\left(\frac{2}{r}\right)=-\frac{1}{r}\overline{\tr\unchi}.
\end{equation*}
Hence, we infer from Lemma \ref{chav} that
\begin{align*}
\nab_3\left(\ov{\trch}-\frac{2}{r}\right)+\frac{1}{2}\trchb\left(\ov{\trch}-\frac{2}{r}\right)&=\frac{1}{2}\trchb\left(\ov{\trch}-\frac{2}{r}\right)-\frac{1}{2}\ov{\trch\,\trchb}+\frac{1}{r}\ov{\trchb}+2\ov{\rho}+\Gag\cdot\Gab\\
&=\frac{1}{2}\ov{\trchb}\left(\ov{\trch}-\frac{2}{r}\right)-\frac{1}{2}\ov{\trch}\,\ov{\trchb}+\frac{1}{r}\ov{\trchb}+2\ov{\rho}+\Gag\cdot(\Gaa,\Gab)\\
&=2\ov{\rho}+\Gag\cdot(\Gaa,\Gab).
\end{align*}
Applying Lemma \ref{evolutionlemma}, we obtain
\begin{align*}
    \left|r\left(\ov{\trch}-\frac{2}{r}\right)\right|_{\infty,S}&\les \left|r\left(\ov{\trch}-\frac{2}{r}\right)\right|_{\infty,S_*}+\int_{u_0(\ub_*)}^u|r\ov{\rho}|_{\infty,S}+|r\Gag\cdot(\Gaa,\Gab)|_{\infty,S} \\
    &\les\frac{\mathcal{I}_0}{r}+\int_{u_0}^u \frac{\De_0+\II_0+\ep^2}{r^2}\\
    &\les\frac{\De_0+\II_0+\ep^2}{r}.
\end{align*}
Hence, we deduce
\begin{equation}\label{trch2rstar}
    \sup_{\Cb_*}\left|r^2\left(\ov{\trch}-\frac{2}{r}\right)\right|\les \De_0+\II_0+\ep^2.
\end{equation}
In view of \eqref{esttrchbcstar}, \eqref{trchcstar}--\eqref{trch2rstar}, we obtain 
\begin{align*}
    \mo^*(\Cb_*)\les \II_0+\De_0+\ep^2.
\end{align*}
This concludes the proof of Theorem M2.
\end{proof}
\subsection{Extension argument (Theorem M4)}\label{ssec12.1}
In this section, we prove Theorem M4, which we recall below for convenience.
\begin{m4}
We consider the spacetime $\kk$ and its double null foliation $(u,\unu)$, which satisfy the assumptions:
\begin{equation}\label{kkinitial}
\mo\les\epsilon_0,\qquad \mr\les \epsilon_0,\qquad \osc\les\ep_0.
\end{equation}
Then, we can extend the spacetime $\kk=V(u_0,\unu_*)$ and the double null foliation $(u,\unu)$ to a new spacetime $\wideparen{\kk}=\wideparen{V}(u_0,\unu_*+\nu)$, where $\nu$ is sufficiently small, and an associated double null foliation $(\upa,\unu)$. Moreover, the new foliation $(\upa,\ub)$ is geodesic on the new last slice $\Cb_{**}:=\Cb_{\ub_*+\nu}$ and
the new norms satisfy
\begin{align}
\wideparen{\mo}\les\ep_0,\qquad\wideparen{\mr}\les\ep_0,\qquad\wideparen{\osc}\les\ep_0.
\end{align}
\end{m4}
Recall that we have introduced three different double null foliations in different regions:
\begin{itemize}
    \item $(u,\ub)$ in the bootstrap region $\kk$;
    \item $(u',\ub)$ in the initial layer region $\kk'$;
    \item $(\upa,\ub)$ in the extended bootstrap region $\wideparen{\kk}$.
\end{itemize}
We introduce the following definition.
\begin{df}\label{generalchange}
    For the change of frame $(f,\la)$ from a double null foliation $(u^{(1)},\ub)$ to another double null foliation $(u^{(2)},\ub)$ in their common domain $\MM$, we define the following norms:
    \begin{equation*}
        \osc[u^{(1)},u^{(2)}](\MM):=\osc[u^{(1)},u^{(2)}](\MM)(f)+\osc[u^{(1)},u^{(2)}](\MM)(\la)+\osc[u^{(1)},u^{(2)}](\MM)(r),
    \end{equation*}
    where
    \begin{align*}
        \osc[u^{(1)},u^{(2)}](f)&=\sup_\MM\left|r^{(1)}|u^{(1)}|^\frac{s-3}{2}\dk^{\leq 1}f\right|,\\
        \osc[u^{(1)},u^{(2)}](\la)&=\sup_\MM\left|r^{(1)}\ovla\right|,\\
        \osc[u^{(1)},u^{(2)}](r)&=\sup_\MM\left|r^{(1)}-r^{(2)}\right|.\\
    \end{align*}
\end{df}
\begin{rk}
Definition \ref{generalchange} immediately implies
    \begin{align}
    \begin{split}\label{osctriangle}
        \osc[u^{(1)},u^{(3)}](\MM)&\les \osc[u^{(1)},u^{(2)}](\MM)+\osc[u^{(2)},u^{(3)}](\MM),\\
        \osc[u^{(1)},u^{(2)}](\MM_1\cup\MM_2)&\les\osc[u^{(1)},u^{(2)}](\MM_1)+\osc[u^{(1)},u^{(2)}](\MM_2).
    \end{split}
    \end{align}
To simplify the notations, we ignore the domain $\MM$ and the optical functions $[u^{(1)},u^{(2)}]$ when they are clear in the context. For example, we denoted
    \begin{equation*}
        \osc=\osc[u,u'](\kk\cap\kk'),\qquad \wideparen{\osc}=\osc[\upa,u'](\wideparen{\kk}\cap\kk'),
    \end{equation*}
    in the statements of Theorems M0 and M4.
\end{rk}
\begin{figure}
  \centering
  \includegraphics[width=1\textwidth]{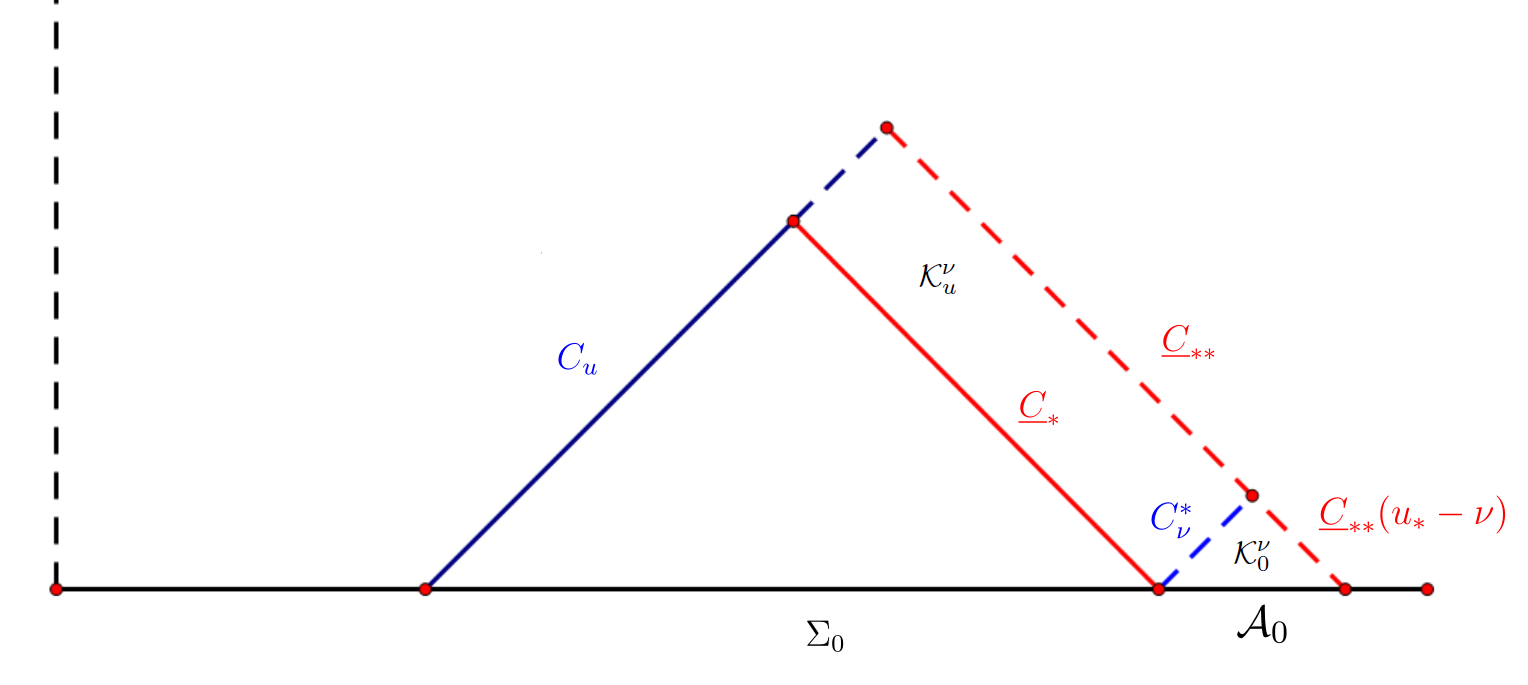}
  \caption{Extension argument}\label{fig5}
\end{figure}
\begin{proof}[Proof of Theorem M4]
\noindent{\bf Step 1.} Extension of the domain.\\ \\
We define the region
\begin{align*}
    \mathcal{A}_0:=\{p\in\Sigma_0/\,w(p)\in[\unu_*,\unu_*+\nu]\}.
\end{align*}
For $\nu$ small enough, we consider the future dependence domain $\kk_0^\nu$ of $\mathcal{A}_0$, with
\begin{equation*}
    \kk_0^\nu:=V'(u_*,\unu_*+\nu)\subset \kk'.
\end{equation*}
The boundary of $\kk_0^\nu$ contain \\
(1) the part $\mathcal{A}_0$; \\
(2) a part of an incoming cone denoted by
\begin{equation*}
    \unc_{**}(u_*-\nu):=\Cb_{\ub_*+\nu}\cap \{u_*-\nu \leq u' \leq u_*\},
\end{equation*}
which is a part of the new last slice $\unc_{**}:=\unc_{\unu_*+\nu}$, where
\begin{equation*}
    u_*-\nu:=u'\big|_{\unc_{\unu_*+\nu}\cap \Sigma_0},\qquad u_*=u'\big|_{\unc_{\unu_*}\cap\Sigma_0};
\end{equation*}
(3) a part of the outgoing cone emanating from $S_{0}(\ub_*)=\unc_{\unu_*}\cap\Si_0$ and in $V'(u_*,\ub_*+\nu)$, denoted by $C^*_\nu$, see Figure \ref{fig5}.\\ \\
For the double null foliation $(u',\unu)$ in $\kk_0^\nu$, we have:
\begin{equation}\label{kk0v}
{\mo}'(\kk_0^\nu)\les\ep_0,\qquad{\mr}'(\kk_0^\nu)\les\ep_0,
\end{equation}
for $\nu$ small enough by local existence. By the local existence theorem of the characteristic Cauchy problem introduced in \cite{luk}, we can extend the spacetime $\kk\cup\kk_0^\nu$ to the future domain of dependence of $C^*_\nu\cup\unc_*$, denoted by $\kk_u^\nu$. Moreover, we extend the outgoing cones of $\kk$ to this new domain, construct the incoming cones from $C^*_\nu$ and denote $(\uti,\ub)$ the global foliation in $\widetilde{\kk}:=\kk\cup\kk_u^\nu\cup\kk_0^\nu$. As a consequence of Theorem 1 of \cite{luk}, we have
\begin{equation}\label{kkuv}
    \tio(\kk_u^\nu)\les\ep_0,\qquad \tir(\kk_u^\nu)\les\epsilon_0,\qquad \osc[\uti,u'](\kk_u^\nu\cap\kk')\les\ep_0.
\end{equation}
Combining \eqref{kkinitial}, \eqref{kk0v} and \eqref{kkuv}, we obtain
\begin{equation}\label{kktilde}
    \tio(\widetilde{\kk})\les\ep_0,\qquad \tir(\widetilde{\kk})\les\epsilon_0,\qquad \osc[\uti,u'](\widetilde{\kk}\cap\kk')\les\ep_0.
\end{equation}
\noindent{\bf Step 2.} Estimates on $\Cb_{**}$. \\ \\
Notice that the double null foliation $(\widetilde{u},\ub)$ is in general not geodesic on $\Cb_{**}$ and only continuous across $C_\nu^*$. Next, we construct from $S'(u_*-\nu,\ub_*+\nu)$ as in section \ref{sssec7.1.2} another new double null foliation $(\upa,\ub)$ which is a geodesic foliation on $\Cb_{**}$.  We add the following bootstrap assumptions on $\Cb_*$, which will be improved at the end of this step:
\begin{equation}\label{bootfbr}
\osc[\uti,\upa](\Cb_{**})\leq\ep,\qquad \wideparen{\mo}^*(\Cb_{**})\leq\ep.
\end{equation}
Applying \eqref{bootfbr} and Proposition \ref{transformation} and recalling that two foliations have the same incoming cones, we deduce by proceeding similarly to the proof of Lemma \ref{a1} that\footnote{The estimate for $\ov{\rho}$ can be done similarly as in Proposition \ref{barrho}.}
\begin{equation}
    \wideparen{\ur}_{[1]}(\Cb_{**})\les \widetilde{\ur}_{[1]}(\Cb_{**})\les\ep_0.
\end{equation}
As a direct consequence of Theorem M2, we obtain
\begin{equation}\label{Ostarparen}
    \wideparen{\mo}^*(\Cb_{**})\les\ep_0.
\end{equation}
Then, by proceeding similarly to the proof of Lemma \ref{oscillation}, we obtain
\begin{equation}\label{osccbss}
    \osc[\uti,\upa](\Cb_{**})\les\ep_0,
\end{equation}
which improves \eqref{bootfbr}.\\ \\
\noindent{\bf Step 3.} Estimates in $\wideparen{\kk}$.\\ \\
Let $U_1$ the set of $|u_0|\leq \ub_1\leq\ub_{**}+\nu$ such that the following holds:
\begin{align}
\begin{split}\label{bootparenmo}
    \wideparen{\mo}(\wideparen{V}_1)\leq\ep,\qquad\qquad \osc[\uti,\upa](\wideparen{V}_1)\leq\frac{\ep}{2},
\end{split}
\end{align}
where
\begin{equation}
    \wideparen{V}_1:=\wideparen{V}(u_0,\ub_*+\nu)\cap \{\ub_1\leq \ub\leq \ub_{**}+\nu\}.
\end{equation}
Notice that \eqref{Ostarparen} and \eqref{osccbss} imply that $U_1\ne\emptyset$. We apply Proposition \ref{transformation} and proceed similarly to the proof of Lemma \ref{a1} to deduce\footnote{Recall that the incoming cones are invariant and that $\ov{\rho}$ can be estimated as in Proposition \ref{barrho}.}
\begin{equation}
    \wideparen{\ur}_0^S \les\widetilde{\ur}_0^S\les\ep_0.
\end{equation}
Next, we estimate $\wideparen{\mr}_0^S[\b]$. For this, recall that
\begin{equation*}
    |r^{-\frac{2}{p}}\nab_3 \b|_{p,\widetilde{S}} = |r^{-2-\frac{2}{p}}\Gag^{(1)}|_{p,\widetilde{S}}\les \frac{\ep}{r^4|u|^\frac{s-3}{2}}.
\end{equation*}
For any sphere $\wideparen{S}=\wideparen{S}(\upa,\ub)$, we take a sphere $\widetilde{S}=\widetilde{S}(\widetilde{u},\ub)$ satisfying $\wideparen{S}\cap\widetilde{S}\ne\emptyset$. The assumption \eqref{bootparenmo} implies
\begin{equation*}
    \sup_{\wideparen{S}}\left|\widetilde{u}-\frac{1}{|\wideparen{S}|}\int_{\wideparen{S}}\widetilde{u}\right|\les\frac{\ep}{r|u|^\frac{s-3}{2}}.
\end{equation*}
Applying Proposition \ref{transformation} and Lemma \ref{deforsp}, we deduce
\begin{align*}
    |r^{-\frac{2}{p}}\wideparen{\b}|_{p,\wideparen{S}}\les |r^{-\frac{2}{p}}\widetilde{\b}|_{p,\wideparen{S}}+\lot\les \frac{\widetilde{\mr}_0^S[\b]}{r^\frac{7}{2}|u|^\frac{s-4}{2}}+\frac{\ep^2}{r^5|u|^{s-3}}\les \frac{\ep_0}{r^\frac{7}{2}|u|^\frac{s-4}{2}}.
\end{align*}
Hence, we deduce
\begin{equation}\label{Rparen}
    \wideparen{\mr}_0^S +\wideparen{\ur}_0^S \les\ep_0.
\end{equation}
Combining \eqref{bootparenmo} and \eqref{kktilde} and applying \eqref{osctriangle}, we infer
\begin{equation*}
    \osc[\upa,u'](\wideparen{V}_1\cap\kk')\leq\frac{\ep}{2}+C\ep_0\leq\ep.
\end{equation*}
We now apply Theorems M0 and M3 to conclude\footnote{Remark that Theorem M3 also applied in the region $\wideparen{V}_1$.}
\begin{equation}\label{oscregion}
    \osc[\upa,u'](\wideparen{V}_1\cap\kk')\les\ep_0,\qquad \quad\wideparen{\mo}(\wideparen{V}_1)\les\ep_0.
\end{equation}
Proceeding similarly to the proof of Lemma \ref{oscillation}, we obtain
\begin{equation*}
    \osc[\uti,\upa](\wideparen{V}_1)\les \ep_0,
\end{equation*}
which improves \eqref{bootparenmo}. Thus, we have $\inf U_1=|u_0|$ and hence,
\begin{equation}\label{oscall}
    \osc[\upa,\uti](\wideparen{\kk})\les\ep_0,\qquad \quad\wideparen{\mo}(\wideparen{\kk})\les\ep_0.
\end{equation}
Notice that \eqref{kktilde} and \eqref{oscall} implies
\begin{equation}\label{oscupa}
    \osc[\upa,u'](\wideparen{\kk}\cap\kk')\les \ep_0.
\end{equation}
Finally, we apply Theorem M1 with \eqref{Rparen}, \eqref{oscall} and \eqref{oscupa} to conclude
\begin{equation}\label{mrparen}
    \wideparen{\mr}\les\ep_0.
\end{equation}
Combining \eqref{oscall}, \eqref{oscupa} and \eqref{mrparen}, this concludes the proof of Theorem M4.
\end{proof}
\appendix
\renewcommand{\appendixname}{Appendix~\Alph{section}}
\section{Proof of Lemma \ref{equivalence}}\label{secb}
In this appendix, we use the notations introduced in Remark \ref{notationprime}.\\ \\
First, we compare $r'$ and $w$. By construction, we have
\begin{align*}
    r'=w=w_0 \quad \mbox{ on }\pr K.
\end{align*} 
We add the bootstrap assumption:
\begin{equation}\label{bootwr'}
    |r'-w|\leq\ep\log r',\quad \mbox{ on }\Si_0\setminus K,
\end{equation}
to ensure the equivalence $r'\simeq w$. Applying \eqref{change} and Lemma \ref{dint}, we deduce
\begin{equation*}
    (\Om'e'_4-\Om'e'_3)(r'-w)=\frac{\ov{\Om'\trch'}}{2}r'-\frac{\ov{\Om'\trchb'}}{2}r'-1+\frac{|f|^2}{4}=r'\Gaa'+O(f^2).
\end{equation*}
Integrating from $w_0$ along $(\Om'e'_4-\Om'e'_3)$ and using $\osc(f)\leq\ep$, we infer
\begin{equation}\label{r'wSi0}
|r'-w|\les\ep+\int_{w_0}^{r'}\left(\frac{\ep^2}{{r'}^{s-1}}+\frac{\ep_0}{r'}\right)dr'\les \ep_0 \log r',
\end{equation}
which improves \eqref{bootwr'}.\\ \\
Next, we compare $u$ and $u'$. By construction, we have
\begin{align*}
    u'=u_0 \quad \mbox{ on }\pr K,
\end{align*}
and $u=u_0$ at a point $p\in \pr K$. Hence, we have from \eqref{change} and $\osc(f)\leq\ep$ that
\begin{align*}
    \sup_{\pr K}|u'-u|&\les |r\nab'(u'-u)|_{\infty,\pr K}+|u'-u|(p)\les \left|r\nab(u)+\frac{1}{2}e_3(u) f\right|_{\infty,\pr K}\les |f|_{\infty,\pr K}\les \ep_0.
\end{align*}
Applying \eqref{change}, we obtain
\begin{align*}
    \Om'e'_4(u)-\Om'e'_3(u)=\frac{\la\Om'}{4}|f|^2e_3(u)-\frac{\Om'}{\la\Om}=\frac{\Om}{4}|f|^2-\frac{{\Om'}^2}{\Om^2}.
\end{align*}
On the other hand, we have
\begin{equation*}
    (\Om'e'_4-\Om'e'_3)(u')=-1.
\end{equation*}
Then, we obtain
\begin{equation*}
    (\Om'e'_4-\Om'e'_3)(u-u')=\frac{\Om}{4}|f|^2+1-\frac{{\Om'}^2}{\Om^2}.
\end{equation*}
Integrating it along $(\Om'e'_4-\Om'e'_3)$, we infer from $\osc\leq\ep$ that
\begin{equation}\label{uu'Si0}
    |u'-u|\les \ep + \int_{w_0}^{r'}\left(\frac{\ep^2}{{r'}^{s-1}}+\frac{\ep}{r'}\right)dr'\les\ep\log r',\quad \mbox{ on }\Si_0\setminus K.
\end{equation}
Notice that
\begin{equation}\label{e4uu'}
    e'_4(u-u')=\frac{|f|^2}{4\Om'}=O\left(\frac{\ep^2}{{r'}^{s-1}}\right).
\end{equation}
Integrating \eqref{e4uu'} from $\Si_0\setminus K$ to $\kk'$ and recalling that the height of $\kk'$ is finite\footnote{Recall that we defined $\kk':=\{p/\, 0\leq u'(p)+\ub(p)\leq 2\de_0\}$ and $\de_0<1$.}
we obtain immediately from \eqref{uu'Si0} that
\begin{equation}\label{uu'kk}
    |u'-u|\les\ep\log r',\quad \mbox{ in }\kk_{(0)}.
\end{equation}
Finally, we compare $r$ and $\frac{\ub-u}{2}$. Recalling that $w=\frac{\ub-u'}{2}$ and applying \eqref{uu'kk}, we obtain
\begin{equation*}
    \left|w-\frac{\unu-u}{2}\right|\les|u'-u|\les \ep\log r' ,\quad \mbox{ in }\kk'.
\end{equation*}
Then, we have, using also \eqref{r'wSi0} and $\osc(r)\leq\ep$,
\begin{equation*}
    \left|r-\frac{\unu-u}{2}\right|\les|r'-r|+|r'-w|+\left|w-\frac{\ub-u}{2}\right|\les \ep\log r'\quad \mbox{ in }\kk'.
\end{equation*}
Next, we have from Lemma \ref{dint}
\begin{equation*}
    \frac{d}{d\unu}\left(r-\frac{\ub-u}{2}\right)=\frac{r}{2}\left(\overline{\Omega\tr\chi}-\frac{1}{r}\right)=r\Gaa.
\end{equation*}
We integrate it along $C_u$ from a sphere $(u,v_0)\in\kk_{(0)}$ to obtain on $\kk$
\begin{equation*}
    \left|r(u,\ub)-\frac{\ub-u}{2}\right|\les\ep\log r'(u,v_0)+\ep\log r(u,\ub)\les\ep \log r(u,\ub).
\end{equation*}
This concludes the proof of Lemma \ref{equivalence}.
\section{Proof of the case \texorpdfstring{$s\in (3,4)$}{}}\label{secc}
In sections \ref{sec8}-\ref{sec11}, we have provided the proof of Theorem \ref{th8.1} in the case of $s\in[4,6]$. In this appendix, we outline the proof in the case $s\in(3,4)$, which is similar to the case of $s\in[4,6]$.
\subsection{Fundamental norms}\label{secc1}
The definitions of some of the norms are different from section \ref{ssec8.1}. In the sequel, we only mention the norms that differ from the ones in section \ref{ssec8.1}.\\ \\
We define $\RR_0^{S}[\b]$ and $\RRb_0^S[\rhoc,\si]$ as follows: 
\begin{align*}
    \mr_0^{S}[\b]&:=\sup_{\mathcal{K}}\sup_{p\in [2,4]}|r^{\frac{s+3}{2}-\frac{2}{p}}\beta |_{p,S},\\
    \ur_0^{p,S}[\rhoc,\si]&:=\sup_{\mathcal{K}}\sup_{p\in [2,4]}|r^{\frac{s+2}{2}-\frac{2}{p}}|u|^{\frac{1}{2}} (\chr,\sigma) |_{p,S}.
\end{align*}
The following flux of curvature components are different from the case $s\in[4,6]$ for $q=0,1$:
\begin{align*}
         \RR_{q}[\beta](u,\ub)&:= \Vert r^{\frac{s}{2}}(r\nab)^q\b\Vert_{2,\cuv},\\
         \RRb_{q}[(\rhoc,\sigma)](u,\ub)&:=\Vert r^{\frac{s}{2}}(r\nab)^q(\rhoc,\si)\Vert_{2,\ucuv}.
\end{align*}
The definition of $\mo_q^{p,S}(\Om\omb)(u,\unu)$ is different from the case $s\in[4,6]$:
\begin{align*}
\mo_0^{p,S}(\Om\underline{\omega})(u,\unu):=|r^{\frac{s}{2}-\frac{2}{p}}|u|^\frac{4-s}{2}\omb|_{p,S(u,\unu)}.
\end{align*}
The other $\mo$ norms are the same as in the case $s\in[4,6]$.
\subsection{Estimates for Ricci coefficients and curvature components}\label{ssecc2}
We discuss the curvature estimates of section \ref{sec9} in this case. Recall that we have four Bianchi pairs: $(\alpha,\beta)$, $(\beta,(\chr,-\sigma))$, $((\chr,\sigma),\unb)$ and $(\unb,\una)$. In section \ref{sec9}, we took respectively $p=s,4,2,0$ to estimate the Bianchi pairs. In the case $s\in(3,4)$, we take respectively $p=s,s,2,0$ to estimate the Bianchi pairs. The method is then exactly the same as in section \ref{sec9}. \\ \\
Concerning the control of Ricci coefficients, we can proceed by the same method as in section \ref{sec10}. Recall that Proposition \ref{prop10.4} used the fact that there exists a constant $\de>0$ such that
\begin{align*}
    \sup_{p\in [2,4]}|r^{3+\de-\frac{2}{p}}\b|_{p,S}\les \De_0.
\end{align*}
This still holds true since we have $\frac{s+3}{2}>3$ for $s>3$. The other propositions still hold true since we only used $s>3$ in their proofs. 
\subsection{Conclusion}\label{ssecc3}
The proof of Theorems M0, M2, M4 and Lemma \ref{equivalence} remain exactly the same since all the arguments also applied to the case $s>3$. Consequently, we have Theorem \ref{th8.1} in this case. Hence, we deduce that Theorem \ref{th8.1} holds true for $s\in(3,6]$.
\section{Proof of the case \texorpdfstring{$s>6$}{}}\label{secd}
In this appendix, we prove Theorem \ref{th8.1} in the case $s>6$. Then, we compare the result to the peeling decay for curvature components obtained in \cite{kncqg}.
\subsection{Fundamental norms}\label{ssecd1}
The definitions of $\mr$--norms for $\a$ and $\b$ are different from section \ref{ssec8.1}. We denote
\begin{align}
    s_0:=\min\left\{s,{\frac{29}{4}}\right\}.
\end{align}
We define for $q=0,1$:
\begin{align}
    \begin{split}\label{newbeta}
    \ur_q[\b]&:=\|r^{3}|u|^{\frac{s-6}{2}}(r\nab)^q\b\|_{2,\ucuv},\\
    \mr_q[\a]&:=\|r^3|u|^\frac{s-6}{2}(r\nab)^q\a\|_{2,\cuv},\\
    \mr_1[\a_4]&:=|u|^\frac{s-s_0}{2}\|r^\frac{s_0}{2}\ac\|_{2,\cuv}.
    \end{split}
\end{align}
We also define:
\begin{equation}
    \ur_0^S[\b]:=\sup_\kk\sup_{p\in [2,4]}|r^{4-\frac{2}{p}}|u|^\frac{s-5}{2}\b|_{p,S}.
\end{equation}
The norms $\mr_0^S[\a]$ are defined as follows:
\begin{align}
\begin{split}\label{peelinga}
    \mr_0^S[\a]&:=\sup_\kk|r^{\frac{s+1}{2}}\a|_{2,S},\qquad\qquad\quad\, s\in(6,7),\\
    \mr_0^S[\a]&:=\sup_\kk|r^{4}(\log r)^{-\frac{1}{2}}\a|_{2,S},\qquad\,\, s=7, \\
    \mr_0^S[\a]&:=\sup_\kk|r^{4}|u|^\frac{s-7}{2}\a|_{2,S},\qquad\quad\;\;\, s>7.
\end{split}
\end{align}
All the other norms are defined as in section \ref{ssec8.1}.
\subsection{Optimal constant for Poincar\'e inequality}
\begin{prop}\label{poincare}
    For any $\a\in\sk_2$, we have the following Poincar\'e inequality:
    \begin{align*}
        |rd_2\a|^2_{{2,S}} \geq c_2 |\a|^2_{{2,S}},
    \end{align*}
    for sphere $S=S(u,\ub)$, where $c_2$ is a constant satisfying:
    \begin{equation}\label{poinc2}
        c_2={2}-O(\ep).
    \end{equation}
\end{prop}
\begin{proof}
    {See for example Proposition 9.3.2 in \cite{DHRT}.}
\end{proof}
\subsection{Curvature estimates}\label{secd3}
In section \ref{sec9}, we took respectively $p=s,4,2,0$ to estimate the Bianchi pairs $(\a,\b)$, $(\b,(\rhoc,\si))$, $((\rhoc,\si),\bb)$ and $(\bb,\aa)$. In the case $s>6$, we take respectively $p=6,4,2,0$ for the Bianchi pairs. Proceeding as in section \ref{sec9}, we deduce from Propositions \ref{estab}-\ref{barrho} and \ref{prop7.8}
\begin{align}
\begin{split}\label{oldest}
&\sum_{q=0}^1\left(\mr_q[\a]+\mr_q[\b]+\ur_q[\b]+\mr_q[\rhoc,\si]+\ur_q[\rhoc,\si]+\mr_q[\bb]+\ur_q[\bb]+\ur_q[\aa]\right)\\
&+\ur_1[\aa_3]+\mr_0^S[\b]+\ur_0^S[\b]+\ur_0^S[\rhoc,\si]+\ur_0^S[\bb]+\ur_0^S[\aa]+\ur_0^S[\ov{\rho}]\les\ep_0.
\end{split}
\end{align}
It remains to estimate $\a$ and $\ac$. To this end, we prove the following divergence identity.
\begin{lem}\label{keyidentity}
We have the following identity for any real number $p$:
\begin{align}
\begin{split}\label{estkeypointteu}
&{\bdiv(r^p|\ac|^2e_3)+2\bdiv(r^p|\as|^2e_4)+\left(p+2\right)r^{p-1}|\ac|^2+2\left(8-p\right)r^{p-1}|\as|^2-8r^{p-1}\ac\cdot\a}\\
=&{4r^pd_1(\ac\cdot\as)+r^p(\ac,\as)\cdot(\Gaa\cdot\rg+\Gag^{(1)}\c\b)+r^p\Gaa\cdot(|\ac|^2,|\as|^2).}
\end{split}
\end{align}
\end{lem}
\begin{proof}
We recall from \eqref{teu}
\begin{align*}
    \nab_3\ac&=-2d_2^*\as+\frac{4\a}{r}+\Gaa\cdot\b^{(1)}+\Gag^{(1)}\c\b,\\
    \nab_4\as+\frac{5}{2}\trch \,\as&=d_2\ac+\Gaa\cdot\b^{(1)}+\Gag^{(1)}\cdot\b.
\end{align*}
Applying Lemma \ref{keypoint} with $\psi_{(1)}=\ac$, $\psi_{(2)}=\as$, $a_{(1)}=0$, $a_{(2)}=\frac{5}{2}$, $h_{(1)}=\frac{4\a}{r}+\Gaa\cdot\rg+\Gag^{(1)}\c\b$, $h_{(2)}=\Gaa\cdot\rg+\Gag^{(1)}\c\b$ and $k=2$, we obtain \eqref{estkeypointteu} as stated.
\end{proof}
Next, we prove the following analog of Proposition \ref{esta4}.
\begin{prop}\label{newalpha}
We have the following estimate:
\begin{align*}
    \mr_1[\a_4]\les\ep_0.
\end{align*}
\end{prop}
\begin{proof}
Integrating \eqref{estkeypointteu} with $p=s_0$ and proceeding as in Proposition \ref{estab}, we obtain
\begin{align*}
    &\int_{C_u}r^{s_0}|\ac|^2 +\int_{\Cb_\ub}r^{s_0}{|\as|^2}+\int_V r^{s_0-1}\big((s_0+2)|\ac|^2+2(8-s_0){|\as|^2-8\ac\cdot\a}\big)\\
    \les &\int_{\Si_0\cap V} r^{s_0}(|\ac|^2+|\as|^2)+\int_V r^{s_0}|(\ac,\as)||\Gaa||\rg|+r^{s_0}|(\ac,\as)||\Gag^{(1)}||\b|+ r^{s_0}|\Gaa|(|\ac|^2+|\as|^2).
\end{align*}
First, we proceed as in Lemma \ref{gainu} to obtain
\begin{align*}
    \int_{\Si_0\cap V} r^{s_0}(|\ac|^2+|\as|^2)\les |u|^{s_0-s}\int_{\Si_0\cap V} r^s (|\ac|^2+|\as|^2)\les \frac{\ep_0^2}{|u|^{s-s_0}}.
\end{align*}
Next, we have
\begin{align*}
    \int_V r^{s_0}|(\ac,\as)||\Gaa||\rg|&\les\ep\int_{V}r^{s_0-2}|\ac||\rg|+\ep\int_V r^{s_0-2}|\as||\rg|\\
    &\les \ep\int_{-\ub}^u du \frac{1}{|u|^\frac{8-s_0}{2}}\left(\int_{\cuv} r^{s_0}|\ac|^2 \right)^\frac{1}{2}\left(\int_{\cuv} r^{4}|\rg|^2 \right)^\frac{1}{2}\\
    &+\ep\int_{|u|}^\ub d\ub\frac{1}{r^{\frac{10-s_0}{2}}}\left(\int_{\ucuv} r^{s_0}|\as|^2 \right)^\frac{1}{2}\left(\int_{\ucuv} r^6|\rg|^2 \right)^\frac{1}{2}\\
    &\les\int_{-\ub}^u du \frac{\ep^3}{|u|^{\frac{8-s_0}{2}}|u|^\frac{s-s_0}{2}|u|^\frac{s-4}{2}}+\int_{-\ub}^u d\ub \frac{\ep^3}{r^{\frac{10-s_0}{2}}|u|^\frac{s-s_0}{2}|u|^{\frac{s-6}{2}}}\\
    &\les \int_{-\ub}^u \frac{\ep^3}{|u|^{2+s-s_0}}du+\int_{-\ub}^u d\ub \frac{\ep^3}{r^{\frac{10-s_0}{2}}|u|^\frac{2s-s_0-6}{2}}\\
    &\les\frac{\ep_0^2}{|u|^{1+s-s_0}},
\end{align*}
where we used $s_0<8$. Then, we compute
\begin{align*}
    \int_V r^{s_0} |(\ac,\as)||\Gag^{(1)}||\b|&\les \int_{V}r^{s_0}|\ac||\Gag^{(1)}||\b|+\int_V r^{s_0}|\as||\Gag^{(1)}||\b|\\
        &\les\int_{-\ub}^u du \left(\int_{\cuv} r^{s_0}|\ac|^2 \right)^\frac{1}{2}\left(\int_{\cuv} r^{s_0}|\Gag^{(1)}|^2|\b|^2 \right)^\frac{1}{2}\\
        &+\int_{|u|}^\ub d\ub \left(\int_{\ucuv} r^{s_0}|\as|^2 \right)^\frac{1}{2}\left(\int_{\ucuv} r^{s_0}|\Gag^{(1)}|^2|\b|^2 \right)^\frac{1}{2}\\
        &\les\ep\int_{-\ub}^u \frac{du}{|u|^\frac{s-s_0}{2}} \left(\int_{|u|}^\ub d\ub\, r^{s_0}\int_S |\Gag^{(1)}|^2|\b|^2 \right)^\frac{1}{2}\\
        &+\ep\int_{|u|}^\ub\frac{d\ub}{|u|^\frac{s-s_0}{2}}\left(\int_{-\ub}^u du\, r^{s_0}\int_S |\Gag^{(1)}|^2|\b|^2 \right)^\frac{1}{2}\\
        &\les\ep\int_{-\ub}^u \frac{du}{|u|^\frac{s-s_0}{2}} \left(\int_{|u|}^\ub d\ub\, r^{s_0-10}|r^\frac{3}{2}\Gag^{(1)}|^2_{4,S}|r^\frac{7}{2}\b|_{4,S}^2 \right)^\frac{1}{2}\\
        &+\ep\int_{|u|}^\ub\frac{d\ub}{|u|^\frac{s-s_0}{2}}\left(\int_{-\ub}^u du\, r^{s_0-10}|r^\frac{3}{2}\Gag^{(1)}|^2_{4,S}|r^\frac{7}{2}\b|_{4,S}^2 \right)^\frac{1}{2}\\
        &\les\ep^3\int_{-\ub}^u \frac{du}{|u|^\frac{s-s_0}{2}}\left(\int_{|u|}^\ub r^{s_0-10}|u|^{8-2s} d\ub\right)^\frac{1}{2}\\
        &+\ep^3\int_{|u|}^\ub \frac{d\ub}{|u|^\frac{s-s_0}{2}}\left(\int_{-\ub}^u r^{s_0-10} |u|^{8-2s}du\right)^\frac{1}{2}\\
        &\les \ep^3\int_{-\ub}^u \frac{du}{|u|^\frac{s-s_0}{2}}r^\frac{s_0-9}{2} |u|^{4-s}+\ep^3\int_{|u|}^\ub \frac{d\ub}{|u|^\frac{s-s_0}{2}}r^{\frac{s_0}{2}-5}|u|^{\frac{9}{2}-s}\\
        &\les \frac{\ep^3}{|u|^\frac{s-s_0}{2}|u|^{s-5}|u|^\frac{9-s_0}{2}}+\frac{\ep^3}{r^{4-\frac{s_0}{2}}|u|^\frac{s-s_0}{2}|u|^\frac{2s-9}{2}}\\
        &\les \frac{\ep_0^2}{|u|^{\frac{s-1}{2}+s-s_0}}.
    \end{align*}
    We also have
    \begin{align*}
        \int_V r^{s_0}|\Gaa|(|\ac|^2+|\as|^2)\les\ep\int_V r^{s_0-2}(|\ac|^2+|\as|^2).
    \end{align*}
    Moreover, applying Proposition \ref{poincare}, we infer
    \begin{align*}
       &{\int_V r^{s_0-1}\big((s_0+2)|\ac|^2+2(8-s_0){|\as|^2-8\ac\cdot\a}\big)}\\
       {\geq}&{\int_V r^{s_0-1}\left((s_0-6)|\ac|^2+\frac{8-s_0}{4}|\as|^2+\left(\frac{7}{4}(8-s_0)c_2-2\right)|\a|^2+8|\ac|^2-8\ac\cdot\a+2|\a|^2\right)}\\
       {\geq}&{\int_V r^{s_0-1}\left((s_0-6)|\ac|^2+\frac{3}{16}|\as|^2+\frac{19}{64}|\a|^2+2|\a-2\ac|^2\right)}\\
       {\gtrsim}&{\int_V r^{s_0-1}\left(|\ac|^2+|\as|^2+|\a|^2\right),}
    \end{align*}
    where we used {$6<s_0\leq\frac{29}{4}$ and $c_2\geq\frac{7}{4}$} for $\ep$ small enough. Combining the above estimates, we obtain
    \begin{align*}
        \int_{\cuv}r^{s_0}|\ac|^2 +\int_{\ucuv}r^{s_0}|\as|^2+\int_V r^{s_0-1}\left(|\ac|^2 +|\as|^2+|\a|^2\right)\les \frac{\ep_0^2}{|u|^{s-s_0}}+\ep \int_V r^{s_0-2}(|\ac|^2+|\as|^2).
    \end{align*}
    For $\ep$ small enough, we deduce
    \begin{align*}
        \int_{\cuv}r^{s_0}|\ac|^2 +\int_{\ucuv}r^{s_0}|\as|^2+\int_V r^{s_0-1}\left(|\ac|^2 +|\as|^2+|\a|^2\right)\les \frac{\ep_0^2}{|u|^{s-s_0}}.
    \end{align*}
    This concludes the proof of Proposition \ref{newalpha}.
\end{proof}
\begin{prop}\label{finalpeeling}
    We have the following estimate:
    \begin{equation*}
    \mr_0^S[\a]\les \ep_0.
    \end{equation*}
\end{prop}
\begin{proof}
We have from Lemma \ref{dint} that
\begin{align*}
    \Om e_4\left( \int_S |r^4\a|^2 \right)&=\int_S 2\Om\nab_4(r^4\a)\cdot (r^4\a)+\Om\trch |r^4\a|^2\\
    &=\int_S 2\Om r^4\a\cdot ( \nab_4 (r^4\a)+ r^3\a )+\Om\left(\trch-\frac{2}{r}\right)|r^4\a|^2\\
    &=\int_S 2\Om r^4\a\cdot (\nab_4(r^5\a )r^{-1}+e_4(r^{-1})r^5\a+r^3\a )+\Gaa|r^4\a|^2\\
    &=\int_S 2\Om r^4\a\cdot (\nab_4(r^5\a )r^{-1})+\Gaa|r^4\a|^2\\
    &=\int_S 2\Om r^4\a\cdot r^3\ac+\Gaa|r^4\a|^2.
\end{align*}
Hence, we obtain
\begin{align*}
    \left|e_4\left( \int_S |r^4\a|^2 \right)\right| \les |r^4\a|_{2,S} |r^3\ac|_{2,S} +\ep |r^4\a|_{{2,S}}|r^2\a|_{2,S},
\end{align*}
which implies 
\begin{align*}
    e_4(|r^4\a|_{2,S})\les |r^3\ac|_{2,S}+\ep |r^2\a|_{2,S}.
\end{align*}
Integrating along $C_u$, applying initial assumption and Proposition \ref{newalpha}, we infer
    \begin{align}
    \begin{split}\label{obstruction}
        |r^4\a|_{2,S}&\les \ep_0+\int_{|u|}^\ub |r^3\ac|_{2,S}d\ub+\ep |r^2\a|_{2,S}\\
        &\les\ep_0+\int_{|u|}^\ub r^{3-\frac{s_0}{2}}|r^\frac{s_0}{2}\ac|_{2,S} d\ub+\int_{|u|}^\ub\frac{\ep}{r^{\frac{3}{2}}}|r^\frac{7}{2}\a|_{2,S}d\ub\\
        &\les\ep_0+\left(\int_{|u|}^\ub r^{6-s_0} d\ub\right)^\frac{1}{2} \left(\int_{|u|}^\ub|r^\frac{s_0}{2}\ac|^2_{2,S} d\ub\right)^\frac{1}{2}+\int_{|u|}^\ub\frac{\ep^2}{r^{\frac{3}{2}}}d\ub\\
        &\les\ep_0+\frac{\ep_0}{|u|^{\frac{s-s_0}{2}}}\left(\int_{|u|}^\ub r^{6-s_0} d\ub\right)^\frac{1}{2}.
    \end{split}
    \end{align}
    Notice that we have
    $$
    \int_{|u|}^\ub r^{6-s_0}d\ub=\left\{
\begin{aligned}
&\frac{1}{7-s_0}(\ub^{7-s_0}-|u|^{7-s_0}),   \qquad\qquad s_0\in (6,7),\\
&\log\left(\frac{\ub}{|u|}\right), \qquad\qquad\qquad\qquad\quad\;\;\, s_0=7,\\
&\frac{1}{s_0-7}(|u|^{7-s_0}-\ub^{7-s_0}), \qquad\qquad\,  s_0>7.
\end{aligned}
\right.
    $$
Hence, we obtain from \eqref{obstruction}
\begin{equation}\label{eq6}
|r^4\a|_{2,S}\les\left\{
\begin{aligned}
&\ep_0 r^{\frac{7-s}{2}},   \qquad\qquad s\in (6,7),\\
&\ep_0 (\log r)^\frac{1}{2}, \qquad\;\;\, s=7,\\
&\ep_0 |u|^{\frac{7-s}{2}}, \qquad\quad\;\,  s>7,
\end{aligned}
\right.
\end{equation}
which implies from \eqref{peelinga} that
    \begin{align}\label{peelingalpha}
        \mr_0^S[\a]\les \ep_0.
    \end{align}
This concludes the proof of Proposition \ref{finalpeeling}.
\end{proof}
Combining Propositions \ref{finalpeeling} with \eqref{oldest}, this concludes the proof of Theorem M1 in the case of $s>6$. Notice that Theorems M0, M2, M3, M4 also hold true in the case $s>6$ since we only used $s>3$ in their proofs\footnote{See the discussion in section \ref{ssecc2}.}. Combining with section \ref{ssecc3}, this concludes the proof of Theorem \ref{th8.1} for $s>3$ as stated.
\subsection{Peeling decay for curvature components}
As a consequence of Theorem M1 in the case $s>7$, the curvature components satisfy the following decay:
\begin{align}
    \begin{split}\label{subpeeling}
        |r^{5-\frac{2}{p}}|u|^\frac{s-7}{2}\a|_{p=2,S}&\les\ep_0,\\
        \sup_{p\in[2,4]}|r^{4-\frac{2}{p}}|u|^{\frac{s-5}{2}}\b|_{p,S}&\les\ep_0,\\
        \sup_{p\in[2,4]}|r^{3-\frac{2}{p}}|u|^{\frac{s-3}{2}}(\rhoc,\si)|_{p,S}&\les \ep_0,\\
        \sup_{p\in[2,4]}|r^{2-\frac{2}{p}}|u|^{\frac{s-1}{2}}\bb|_{p,S}&\les\ep_0,\\
        \sup_{p\in[2,4]}|r^{1-\frac{2}{p}} |u|^{\frac{s+1}{2}}\aa|_{p,S}&\les\ep_0.\\
    \end{split}
\end{align}
\begin{rk}
    Assume $s>7$ and that the initial data in $\kk_{(0)}$ has sufficient regularity properties. Then, commuting $r\nab$ with the Bianchi equations in Corollary \ref{prop7.6} and the Teukolsky equation \eqref{teukolsky}, proceeding as in section \ref{secd3}, we deduce the following strong peeling properties: 
    \begin{align}
    \begin{split}\label{strongpeeling}
        |\a|_{\infty,S(u,\ub)}&\les \frac{\ep_0}{r^{5}|u|^\frac{s-7}{2}},\qquad\quad \,|\b|_{\infty,S(u,\ub)}\les \frac{\ep_0}{r^{4}|u|^\frac{s-5}{2}},\\
        |\rhoc,\si|_{\infty,S(u,\ub)}&\les\frac{\ep_0}{r^3|u|^\frac{s-3}{2}},\qquad\quad\; |\bb|_{\infty,S(u,\ub)}\les \frac{\ep_0}{r^2|u|^\frac{s-1}{2}},\qquad\quad|\aa|_{\infty,S(u,\ub)}\les \frac{\ep_0}{r|u|^\frac{s+1}{2}},
    \end{split}
\end{align}
which recovers the results obtained in \cite{kncqg} by the vectorfield method.
\end{rk}
\section*{Declarations}
\addcontentsline{toc}{section}{Declarations}
\noindent{\bf Funding.} No funding was received to assist with the preparation of this manuscript.\\ \\
{\bf Competing Interest statements.} Conflict of interest does not exist in the manuscript.

\end{document}